\documentclass[reqno,11pt]{amsart}
\usepackage{amsmath,amssymb,latexsym,cancel,rotating,hyperref}

\usepackage{eucal}

\usepackage{amscd}
\usepackage[dvips]{color}
\usepackage{multicol}
\usepackage[all]{xy}           %xypic macro for latex2.09
\usepackage{graphicx}
\usepackage{color}
\usepackage{colordvi}
\usepackage{xspace}
\usepackage{tikz}

\usepackage{enumitem}

\numberwithin{equation}{section}

\newtheorem{theorem}{Theorem}[section]

\newtheorem{corollary}[theorem]{Corollary}
\newtheorem{lemma}[theorem]{Lemma}

\newtheorem{remark}[theorem]{Remark}

\newtheorem{definition}[theorem]{Definition}

\input xy
\xyoption{poly}
\xyoption{2cell}
\xyoption{all}

\def\ZZ{\mathbb Z}

\def\n{\nabla}

\def\ZZ{\mathbb{Z}}

\renewcommand{\eqref}[1]{{\rm (\ref{#1})}}

\begin{document}
\title[Quantum cluster algebra of type $A_2^{(1)}$]
{Explicit cluster multiplication formulas for the quantum cluster algebra of type $A_2^{(1)}$}

\author{Danting Yang, Xueqing Chen, Ming Ding and Fan Xu}
\address{School of Mathematics and Information Science\\
Guangzhou University, Guangzhou 510006, P.R.China}
\email{dantingyang@e.gzhu.edu.cn (D.Yang)}
\address{Department of Mathematics,
 University of Wisconsin-Whitewater\\
800 W. Main Street, Whitewater, WI.53190. USA}
\email{chenx@uww.edu (X.Chen)}
\address{School of Mathematics and Information Science\\
Guangzhou University, Guangzhou 510006, P.R.China}
\email{dingming@gzhu.edu.cn (M.Ding)}
\address{Department of Mathematical Sciences\\
Tsinghua University\\
Beijing 100084, P.~R.~China} \email{fanxu@mail.tsinghua.edu.cn (F.Xu)}

%\keywords{$\mathbb{Z}[q^{\pm \frac{1}{2}}]-$basis, quantum cluster
%algebra.}

%\date{April 10, 2010}
%\thanks{Ming Ding was supported by NSF of China (No. 11301282) and Specialized Research Fund for the Doctoral Program of Higher Education (No. 20130031120004) and Fan Xu was supported by NSF of China (No. 11471177).}

%    General info
%\subjclass[2000]{Primary  16G20, 17B67; Secondary  17B35, 18E30}

%\date{\today}


\maketitle

\begin{abstract}
Let $Q$ be an affine  quiver of  type $A_2^{(1)}$. We  explicitly construct the cluster multiplication formulas for the  quantum cluster algebra of $Q$ with principal coefficients. As applications, we obtain: (1)\  an exact expression for every quantum cluster variable as  a polynomial in terms of the quantum cluster variables in  clusters which are one-step mutations from the initial cluster; (2)\  an explicit bar-invariant positive $\mathbb{ZP}$-basis.
\end{abstract}

%\tableofcontents

\section{Background}
Quantum cluster algebras, as the noncommutative deformations of cluster algebras defined in  \cite{ca1}, were
introduced by Berenstein and Zelevinsky  \cite{berzel} to provide an algebraic framework for a better understanding of dual canonical bases in
quantum coordinate rings.  A  (quantum)  cluster algebra   is generated by so-called (quantum) cluster variables, which are grouped into a collection of finite subsets called clusters. The remarkable Laurent phenomenon says
that any (quantum) cluster variable is a Laurent polynomial in  the (quantum) cluster variables belonging
to every cluster.  We say that an  element of a (quantum)  cluster algebra is  positive if its Laurent expansion in every cluster has non-negative coefficients.

 Cluster multiplication formulas play a crucial role in the study of cluster theory and  attract a lot of attention. In the setting of acyclic cluster algebras, Sherman and Zelevinsky \cite{SZ} firstly gave  the  cluster multiplication formulas in rank 2 cluster algebras of finite and affine types, which was generalized to affine type $A_{2}^{(1)}$ by Cerulli \cite{CI-1}.  Caldero and Keller \cite{ck}  constructed the cluster multiplication formulas   for finite types. This result  was generalized to affine types by Hubery \cite{Hubery} and  acyclic types by Xiao and Xu \cite{XX}, and by Xu \cite{X}.  In the setting of acyclic quantum cluster algebras, Ding and Xu \cite{dx} firstly constructed the  cluster multiplication formulas of the Kronecker  quantum  cluster algebra. Later, Bai, Chen, Ding and Xu \cite{BCDX} extended this result to affine type $A_{2}^{(2)}$. Recently, Chen, Ding and Zhang \cite{cdz} constructed the cluster multiplication formulas  for acyclic types through certain quotients of derived Hall algebras. It is worth remarking that although the  structure constants in the cluster multiplication formulas \cite{cdz} have the homological interpretations, it is very difficult to explicitly compute them in general.

One of the most important problems in cluster theory is to  construct natural bases, among which there  is a specific one of particular interest called a positive basis  (see for example \cite{SZ,Q3,li-pan}).
 We say that a basis of a cluster algebra is positive if its structure constants are positive. Note that  one of the advantages of cluster multiplication formulas is that it provides  a possible way to construct explicitly such positive bases of (quantum) cluster algebras \cite{SZ,CI-1,ck,dx, BCDX}.

In this paper, we concentrate on the quantum cluster algebra associated with an affine quiver of type $A_2^{(1)}$. The paper is organized as follows. Section 2 reviews the basics about quantum cluster algebras and describes the quantum cluster algebra of type $A_2^{(1)}$ with principal coefficients. In Section 3, we  establish the cluster multiplication formulas for the  quantum cluster algebra of type $A_2^{(1)}$ with principal coefficients, whose structure constants are computed explicitly. Note that some of these  formulas can be viewed as a quantum analogue of the constant coefficient linear relations, from which we obtain a simple description for every quantum cluster variable as  a polynomial in terms of the quantum cluster variables in  clusters which are one-step mutations from the initial cluster. In Section 4, with the help of the cluster multiplication formulas, we  provide an explicit   bar-invariant positive $\mathbb{ZP}$-basis. The results extend some existing ones of the corresponding  cluster algebra in \cite{CI-1} to the quantum version.
\section{Preliminaries}
\subsection{Quantum cluster algebras}
In this subsection, we collect some notations and concepts of quantum cluster algebras in \cite{berzel}.

Let $\widetilde{B}=(b_{ij})$ be an $m\times n$ integer matrix with $m\geq n$, whose $j$th column is denoted by ${\bf b}_j$ for $1\leq j\leq n$ and principal part $B$ is the upper $n\times n$ submatrix.
The pair $(\Lambda, \widetilde{B})$ is called  compatible if
$$\Lambda ({\bf b}_j,{\bf e}_i)=\delta_{ij}d_j$$
for $1\leq i\leq m$, $1\leq j\leq n$, where the set $\{{\bf e}_1,\cdots,{\bf e}_m\}$  is the standard basis in $\mathbb{Z}^{m}$, the elements in $\{d_1,\cdots,d_n\}$ are positive integers
and $\Lambda: \mathbb{Z}^{m} \times \mathbb{Z}^{m} \rightarrow \mathbb{Z}$ is a skew-symmetric bilinear form on $ \mathbb{Z}^{m}$. With some abuse of notation, we
use the same symbol $\Lambda$ to denote the skew-symmetric matrix whose entries are given by
$$\lambda_{ij}=\Lambda ({\bf e}_i,{\bf e}_j).$$

 Let $q$  be the formal parameter.
 The based quantum torus  $\mathcal{T}=\mathcal{T}(\Lambda)$  is the $\ZZ[q^{\pm\frac{1}{2}}]$-algebra  with a
distinguished $\ZZ[q^{\pm\frac{1}{2}}]$-basis $\{X^{\bf c}| {\bf c}\in \mathbb{Z}^{m}\}$ and the multiplication given by
$$X^{\bf c}X^{\bf d}=q^{\Lambda({\bf c},{\bf d})/2}X^{{\bf c+d}}\,\,\, \,\,\,\, \text{ for all }\bf{c}, \bf{d} \in  \mathbb{Z}^\text{m}.$$

Note that  $\mathcal{T}$ is   contained in $\mathcal{F}$, its skew-field
of fractions. An initial  quantum seed $(\Lambda, \widetilde{B}, X)$ of $\mathcal{F}$
  consists of a compatible pair $(\Lambda,\widetilde{B})$ and an initial cluster  $X=\{X_1,\cdots,X_m\}$ with  $X_i:=X^{{\bf e_i}}$ for $1\leq i\leq m$.

  For any rational number $a$, denote the floor (resp.~ceiling) function   by  $\lfloor a\rfloor$ (resp.~$\lceil a\rceil$), and  the function $[a]_+=max(a,0)$.

  Let $k\in \{1,2,\cdots,n\}$. The quantum seed mutation $\mu_k$  in direction $k$ transforms  $(\Lambda, \widetilde{B}, X)$ into a new quantum seed  $\mu_k (\Lambda, \widetilde{B}, X)=(\Lambda',\widetilde{B}',X')$ defined as follows:
\begin{enumerate}
\item [$\bullet$] $\Lambda'=E^{T}\Lambda E$, where the
$m\times m$ matrix $E=(e_{ij})$ is given by
\[e_{ij}=\begin{cases}
\delta_{ij} & \text{if $j\ne k$;}\\
-1 & \text{if $i=j=k$;}\\
[-b_{ik}]_{+} & \text{if $i\ne j = k$.}
\end{cases}
\]
\item [$\bullet$] $\widetilde{B}'=(b'_{ij})$ is given by
\[b'_{ij}=\begin{cases}
-b_{ij} & \text{if $i=k$ or $j=k$;}\\
b_{ij}+\frac{1}{2}(|b_{ik}|b_{kj} +b_{ik}|b_{kj}|)& \text{otherwise.}
\end{cases}
\]
\item [$\bullet$] $X'=\{X'_1,\cdots,X'_m\}$, where $X_j'=X_j$ for any $j\neq k$, and
$$X_k'=X^{-{\bf e}_k+\sum_{1\leq i\leq m}[b_{ik}]_{+} {\bf e}_i}+X^{ -{\bf e}_k+\sum_{1\leq i\leq m}[-b_{ik}]_{+} {\bf e}_i}.$$
\end{enumerate}

 For any quantum seed $(\Lambda^{\ast}, \widetilde{B}^{\ast}, X^{\ast})$, which is obtained from the initial quantum seed $(\Lambda,\widetilde{B},X)$ by a finite
sequence of mutations, we call the set $\{(X^{\ast})^{{\bf e}_i}~|1\leq i\leq n\}$ the  cluster whose elements are called
the  quantum cluster variables, and the
elements in  $\mathbb{P}:=\{X_i~|~n+1\leq i\leq m\}$  the  coefficients. The quantum cluster monomials are the monomials of the quantum cluster
variables contained in any common quantum seed.  Let
$\ZZ\mathbb{P}$ (resp. $\mathbb{Z}_{\geq 0}\mathbb{P}$) be the ring of  Laurent polynomials in  $\mathbb{P}$ with coefficients in $\ZZ[q^{\pm\frac{1}{2}}]$
(resp.  $\ZZ_{\geq 0}[q^{\pm\frac{1}{2}}]$). The
quantum cluster algebra $\mathcal{A}_{q}(\Lambda,\widetilde{B})$ is  the
$\ZZ\mathbb{P}$-subalgebra of $\mathcal{F}$ generated by all
quantum cluster variables.  In fact, $\mathcal{A}_{q}(\Lambda,\widetilde{B})$ is
contained in the quantum torus  associated with any quantum seed according to the Laurent phenomenon proved in \cite{berzel}.

One can define the $\mathbb{Z}$-linear bar-involution of $\mathcal{T}$:
$$\overline{q^{\frac{l}{2}}X^{\mathbf{c}}}=q^{-\frac{l}{2}}X^{\mathbf{c}},\ \   \text{ for  all}\   l\in \mathbb{Z}  \text { and }\mathbf{c}\in \mathbb{Z}^{m}.$$

It is known that $\overline{fg} = \overline{g}\ \overline{f}$ for any $f, g \in \mathcal{T}$ and that
quantum cluster monomials are bar-invariant.
\subsection{Quantum cluster algebra  of type $A_2^{(1)}$ with principal coefficients}
We consider the compatible pair $(\Lambda,\widetilde{B})$ as follows:

$$\Lambda=\left(\begin{array}{cccccc} 0 & 0& 0& -1& 0& 0\\
0& 0& 0& 0& -1& 0\\
0& 0& 0& 0& 0& -1\\
1& 0& 0& 0& -1& -1\\
0& 1& 0& 1& 0& -1\\
0 & 0& 1& 1& 1& 0\end{array}\right),\
\widetilde{B}=\left(\begin{array}{ccc} 0 & 1& 1\\
-1& 0& 1\\
-1& -1& 0\\
1& 0& 0\\
0& 1& 0\\
0 & 0& 1\end{array}\right).$$

For the skew-symmetric principal submatrix of $\widetilde{B}$, one can associate a  quiver $Q$:
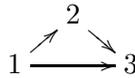
$$
\xymatrix @R=5pt@C=10pt{
         &&2\ar[dr]&\\
&1\ar[ur]\ar[rr]&  &3}
$$
which is exactly an affine quiver of type $A_2^{(1)}$.  The quantum cluster algebra $\mathcal{A}_{q} (Q)$ associated with the initial quantum seed $\Sigma=(\Lambda,\widetilde{B},X)$, where $X=\{X_1, X_2,X_3, y_1,y_2,y_3\}$, is called the  quantum cluster algebra of type $A_2^{(1)}$ with principal coefficients. Note that $y_1,y_2,y_3$ are coefficients. In the rest of the paper, we will work only on the quantum cluster algebra $\mathcal{A}_{q} (Q)$.

We define\\
\[\begin{array}{l}
w={X^{(1,-1,0,0,1,0)}}+{X^{(0,-1,1,0,0,0)}},\\
z={X^{(0,1,-1,1,0,1)}}+{X^{(-1,0,-1,1,0,0)}}+{X^{(-1,1,0,0,0,0)}}.\\
\end{array}\]

The following result is a quantum analogue of \cite[Lemma 2.1.2]{CI-0}.

\begin{lemma}\label{LemmaAlgStructure}
The unlabeled seeds of $\mathcal{A}_{q} (Q)$ with the initial seed $\Sigma=\Sigma_{1}$ are given by
\begin{eqnarray*}
\Sigma_{m} &=&(\Lambda_{m}, \widetilde{B}_{m}, \{X_{m}, X_{m+1}, X_{m+2}, y_1,y_2,y_3\}),\\
\Sigma^{cyc}_{2m-1} &=&( \Lambda_{2m-1}^{cyc}, \widetilde{B}_{2m-1}^{cyc},\{X_{2m-1},w, X_{2m+1}, y_1,y_2,y_3\}),\\
\Sigma^{cyc}_{2m} &=&(\Lambda_{2m}^{cyc},\widetilde{B}_{2m}^{cyc}, \{X_{2m}, z, X_{2m+2}, y_1,y_2,y_3\})
\end{eqnarray*}
for any $m\in\ZZ$, and they are mutually related by
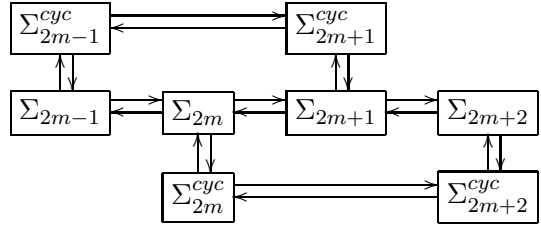
\begin{equation*}\label{DiagramMutations}
\xymatrix@!R=8pt@C=20pt{
*+[F]{\Sigma^{cyc}_{2m-1}}\ar@<1ex>[rr]\ar@<1ex>[d]&   &  *+[F]{\Sigma^{cyc}_{2m+1}}\ar@<1ex>[d]\ar[ll]    &                 &\\
*+[F]{\Sigma_{2m-1}}\ar@<1ex>[r]\ar[u]    &*+[F]{\Sigma_{2m}}\ar@<1ex>[r]\ar@<1ex>[d]\ar[l]         &   *+[F]{\Sigma_{2m+1}}\ar@<1ex>[r]\ar[l]\ar[u]       & *+[F]{\Sigma_{2m+2}}\ar[l]\ar@<1ex>[d]\\
                              & *+[F]{\Sigma^{cyc}_{2m}}\ar@<1ex>[rr]\ar[u]   &                                   & *+[F]{\Sigma^{cyc}_{2m+2}}\ar[u]\ar[ll]
}
\end{equation*}
where  arrows from left to right (resp. from right to left) are mutations in direction~$1$
(resp.~$3$) and vertical arrows (in both directions) are mutations in direction~$2$.
The  matrices  are given by the following for any $n\in \mathbb{Z}_{>0}$:

\[\begin{array}{cc}
\widetilde{B}_{2n-1}=\left(\def\objectstyle{\scriptstyle}\def\labelstyle{\scriptstyle}\vcenter{\xymatrix@R=0pt@C=0pt{0&1&1\\-1&0&1\\-1&-1&0\\n&0&1-n\\n-1&1&1-n\\n-1&0&2-n}}\right),
\Lambda_{2n-1}=\left(\def\objectstyle{\scriptstyle}\def\labelstyle{\scriptstyle}\vcenter{\xymatrix@R=0pt@C=0pt {0 & 0& 0& n-2& 0& 1-n\\
0& 0& 0& n-1& -1& 1-n\\
0& 0& 0& n-1& 0& -n\\
2-n& 1-n& 1-n& 0& -1& -1\\
0& 1& 0& 1& 0& -1\\
n-1 & n-1& n& 1& 1& 0}}\right);
\end{array}\]

\[\begin{array}{cc}
\widetilde{B}_{2n}=\left(\def\objectstyle{\scriptstyle}\def\labelstyle{\scriptstyle}\vcenter{\xymatrix@R=0pt@C=0pt{0&1&1\\-1&0&1\\-1&-1&0\\n&1&-n\\n&0&1-n\\n-1&1&1-n}}\right),
\Lambda_{2n}=\left(\def\objectstyle{\scriptstyle}\def\labelstyle{\scriptstyle}\vcenter{\xymatrix@R=0pt@C=0pt {0 & 0& 0& n-1& -1& 1-n\\
0& 0& 0& n-1& 0& -n\\
0& 0& 0& n& -1& -n\\
1-n& 1-n& -n& 0& -1& -1\\
1& 0& 1& 1& 0& -1\\
n-1& n& n& 1& 1& 0}}\right);
\end{array}
\]

\[
\begin{array}{cc}
\widetilde{B}_{0}=\left(\def\objectstyle{\scriptstyle}\def\labelstyle{\scriptstyle}\vcenter{\xymatrix@R=0pt@C=0pt{0&1&1\\-1&0&1\\-1&-1&0\\0&1&0\\0&0&1\\-1&0&0}}\right),
\Lambda_{0}=\left(\def\objectstyle{\scriptstyle}\def\labelstyle{\scriptstyle}\vcenter{\xymatrix@R=0pt@C=0pt {0 & 0& 0& 0& 0& 1\\
0& 0& 0& -1& 0& 0\\
0& 0& 0& 0& -1& 0\\
0& 1& 0& 0& -1& -1\\
0& 0& 1& 1& 0& -1\\
-1& 0& 0& 1& 1& 0}}\right);
\end{array}
\]

\[
\begin{array}{cc}
\widetilde{B}_{-2n-1}=\left(\def\objectstyle{\scriptstyle}\def\labelstyle{\scriptstyle}\vcenter{\xymatrix@R=0pt@C=0pt{0&1&1\\-1&0&1\\-1&-1&0\\n-1&-1&1-n\\n-1&0&-n\\n&-1&-n}}\right),
\Lambda_{-2n-1}=\left(\def\objectstyle{\scriptstyle}\def\labelstyle{\scriptstyle}\vcenter{\xymatrix@R=0pt@C=0pt {0 & 0& 0& n& 1& -n\\
0& 0& 0& n& 0& 1-n\\
0& 0& 0& n-1& 1& 1-n\\
-n& -n& 1-n& 0& -1& -1\\
-1& 0& -1& 1& 0& -1\\
n & n-1& n-1& 1& 1& 0}}\right);
\end{array}
\]

\[
\begin{array}{cc}
\widetilde{B}_{-2n}=\left(\def\objectstyle{\scriptstyle}\def\labelstyle{\scriptstyle}\vcenter{\xymatrix@R=0pt@C=0pt{0&1&1\\-1&0&1\\-1&-1&0\\n-2&0&1-n\\n-1&-1&1-n\\n-1&0&-n}}\right),
\Lambda_{-2n}=\left(\def\objectstyle{\scriptstyle}\def\labelstyle{\scriptstyle}\vcenter{\xymatrix@R=0pt@C=0pt {0 & 0& 0& n& 0& 1-n\\
0& 0& 0& n-1& 1& 1-n\\
0& 0& 0& n-1& 0& 2-n\\
-n& 1-n& 1-n& 0& -1& -1\\
0& -1& 0& 1& 0& -1\\
n-1 & n-1& n-2& 1& 1& 0}}\right);
\end{array}
\]

\[
\begin{array}{cc}
\widetilde{B}_{-1}=\left(\def\objectstyle{\scriptstyle}\def\labelstyle{\scriptstyle}\vcenter{\xymatrix@R=0pt@C=0pt{0&1&1\\-1&0&1\\-1&-1&0\\0&0&1\\-1&0&0\\0&-1&0}}\right),
\Lambda_{-1}=\left(\def\objectstyle{\scriptstyle}\def\labelstyle{\scriptstyle}\vcenter{\xymatrix@R=0pt@C=0pt {0 & 0& 0& 0& 1& 0\\
0& 0& 0& 0& 0& 1\\
0& 0& 0& -1& 0& 0\\
0& 0& 1& 0& -1& -1\\
-1& 0& 0& 1& 0& -1\\
0 & -1& 0& 1& 1& 0}}\right);
\end{array}
\]

\[
\begin{array}{cc} \widetilde{B}^{cyc}_{-2n-1}=\left(\def\objectstyle{\scriptstyle}\def\labelstyle{\scriptstyle}\vcenter{\xymatrix@R=0pt@C=0pt{0&-1&2\\1&0&-1\\-2&1&0\\n-2&1&1-n\\n-1&0&-n\\n-1&1&-n}}\right),
\Lambda^{cyc}_{-2n-1}=\left(\def\objectstyle{\scriptstyle}\def\labelstyle{\scriptstyle}\vcenter{\xymatrix@R=0pt@C=0pt {0 & 0& 0& n& 1& -n\\
0& 0& 0& 0& 1& -1\\
0& 0& 0& n-1& 1& 1-n\\
-n& 0& 1-n& 0& -1& -1\\
-1& -1& -1& 1& 0& -1\\
n& 1& n-1& 1& 1& 0}}\right);
\end{array}
\]

\[
\begin{array}{cccccc} \widetilde{B}^{cyc}_{2n-1}=\left(\def\objectstyle{\scriptstyle}\def\labelstyle{\scriptstyle}\vcenter{\xymatrix@R=0pt@C=0pt{0&-1&2\\1&0&-1\\-2&1&0\\n&0&1-n\\n-1&-1&2-n\\n-1&0&2-n}}\right),
\Lambda^{cyc}_{2n-1}=\left(\def\objectstyle{\scriptstyle}\def\labelstyle{\scriptstyle}\vcenter{\xymatrix@R=0pt@C=0pt {0 & 0& 0& n-2& 0& 1-n\\
0& 0& 0& 0& 1& -1\\
0& 0& 0& n-1& 0& -n\\
2-n& 0& 1-n& 0& -1& -1\\
0& -1& 0& 1& 0& -1\\
n-1& 1& n& 1& 1& 0}}\right);
\end{array}
\]

\[
\begin{array}{cc}
\widetilde{B}^{cyc}_{-1}=\left(\def\objectstyle{\scriptstyle}\def\labelstyle{\scriptstyle}\vcenter{\xymatrix@R=0pt@C=0pt{0&-1&2\\1&0&-1\\-2&1&0\\0&0&1\\-1&0&0\\-1&1&0}}\right),
\Lambda^{cyc}_{-2n-1}=\left(\def\objectstyle{\scriptstyle}\def\labelstyle{\scriptstyle}\vcenter{\xymatrix@R=0pt@C=0pt {0 & 0& 0& 0& 1& 0\\
0& 0& 0& 0& 1& -1\\
0& 0& 0& -1&0& 0\\
0& 0& 1& 0& -1& -1\\
-1& -1& 0& 1& 0& -1\\
0& 1& 0& 1& 1& 0}}\right);
\end{array}
\]

\[
\begin{array}{cc}
\widetilde{B}^{cyc}_{-2n}=\left(\def\objectstyle{\scriptstyle}\def\labelstyle{\scriptstyle}\vcenter{\xymatrix@R=0pt@C=0pt{0&-1&2\\1&0&-1\\-2&1&0\\n-2&0&1-n\\n-2&1&1-n\\n-1&0&-n}}\right),
\Lambda^{cyc}_{-2n}=\left(\def\objectstyle{\scriptstyle}\def\labelstyle{\scriptstyle}\vcenter{\xymatrix@R=0pt@C=0pt {0 & 0& 0& n& 0& 1-n\\
0& 0& 0& 1& -1& 0\\
0& 0& 0& n-1& 0& 2-n\\
-n& -1& 1-n& 0& -1& -1\\
0& 1& 0& 1& 0& -1\\
n-1& 0& n-2& 1& 1& 0}}\right);
\end{array}
\]

\[
\begin{array}{cc}
\widetilde{B}^{cyc}_{2n}=\left(\def\objectstyle{\scriptstyle}\def\labelstyle{\scriptstyle}\vcenter{\xymatrix@R=0pt@C=0pt{0&-1&2\\1&0&-1\\-2&1&0\\n&-1&1-n\\n&0&1-n\\n-1&-1&2-n}}\right),
\Lambda^{cyc}_{2n}=\left(\def\objectstyle{\scriptstyle}\def\labelstyle{\scriptstyle}\vcenter{\xymatrix@R=0pt@C=0pt {0 & 0& 0& n-1& -1& 1-n\\
0& 0& 0& 1& -1& 0\\
0& 0& 0& n& -1& -n\\
1-n& -1& -n& 0& -1& -1\\
1& 1& 1& 1& 0& -1\\
n-1& 0& n& 1& 1& 0}}\right);
\end{array}
\]

\[
\begin{array}{cc}
\widetilde{B}^{cyc}_{0}=\left(\def\objectstyle{\scriptstyle}\def\labelstyle{\scriptstyle}\vcenter{\xymatrix@R=0pt@C=0pt{0&-1&2\\1&0&-1\\-2&1&0\\0&-1&1\\0&0&1\\-1&0&0}}\right),
\Lambda^{cyc}_{0}=\left(\def\objectstyle{\scriptstyle}\def\labelstyle{\scriptstyle}\vcenter{\xymatrix@R=0pt@C=0pt {0 & 0& 0& 0& 0& 1\\
0& 0& 0& 1& -1& 0\\
0& 0& 0& 0& -1& 0\\
0& -1& 0& 0& -1& -1\\
0& 1& 1& 1& 0& -1\\
-1& 0& 0& 1& 1& 0}}\right).
\end{array}
\]
\end{lemma}
\begin{proof}
The proof follows from \cite[Lemma 2.1.2]{CI-0} and \cite[Proposition 2.12]{CL}.
\end{proof}
By Lemma \ref{LemmaAlgStructure}, one can get the following exchange relations.

\begin{lemma}\label{1}
For any  $n\in \mathbb{Z}_{>0}$, we have that
\[\begin{array}{l}
{X_{2n}}{X_{2n + 3}} = {q^{\frac{{3{n^2} - 2n - 1}}{2}}}y_1^ny_2^ny_3^{n - 1} + {X_{2n + 1}}{X_{2n + 2}};\\
{X_{2n - 1}}{X_{2n + 2}} = {q^{\frac{{3{n^2} - 4n}}{2}}}y_1^ny_2^{n - 1}y_3^{n - 1} + {X_{2n}}{X_{2n + 1}};\\
{X_{ - (2n + 1)}}{X_{ - (2n - 2)}} = {q^{\frac{{3{n^2} - 4n}}{2}}}y_1^{n - 1}y_2^{n - 1}y_3^n + {X_{ - (2n - 1)}}{X_{ - 2n}};\\
{X_{ - (2n + 2)}}{X_{ - (2n - 1)}} = {q^{\frac{{3{n^2} - 2n - 1}}{2}}}y_1^{n - 1}y_2^ny_3^n + {X_{ - 2n}}{X_{ - (2n + 1)}};\\
{X_0}{X_3} = {q^{\frac{1}{2}}}{X_1}{X_2}{y_3} + 1;\\
{X_{ - 1}}{X_2} = {q^{\frac{1}{2}}}{X_0}{X_1}{y_2} + 1;\\
{X_{ - 2}}{X_1} = {q^{\frac{1}{2}}}{X_{ - 1}}{X_0}{y_1} + 1;\\
w{X_{2n}} = {q^{\frac{1}{2}}}{X_{2n - 1}}{y_2} + {X_{2n + 1}};\\
w{X_{ - 2n}} = {X_{ - (2n - 1)}}{y_1}{y_3} + {X_{ - (2n + 1)}};\\
z{X_{2n + 1}} = q{X_{2n}}{y_1}{y_3} + {X_{2n + 2}};\\
z{X_{ - (2n - 1)}} = {q^{ - \frac{1}{2}}}{X_{ - (2n - 2)}}{y_2} + {X_{ - 2n}};\\
w{X_0} = {q^{ - \frac{1}{2}}}{X_1}{y_3} + {X_{ - 1}};\\
z{X_1} = {q^{\frac{1}{2}}}{X_0}{y_1} + {X_2};\\
{X_{2n - 1}}{X_{2n + 3}} = {q^{\frac{{3{n^2} - 4n}}{2}}}wy_1^ny_2^{n - 1}y_3^{n - 1} + X_{2n + 1}^2;\\
{X_{ - (2n + 3)}}{X_{ - (2n - 1)}} = {q^{\frac{{3{n^2} - 2n - 1}}{2}}}wy_1^{n - 1}y_2^ny_3^n + X_{ - (2n + 1)}^2;\\
{X_{2n}}{X_{2n + 4}} = {q^{\frac{{3{n^2} - 2n - 1}}{2}}}zy_1^ny_2^ny_3^{n - 1} + X_{2n + 2}^2;\\
{X_{ - (2n + 2)}}{X_{ - (2n - 2)}} = {q^{\frac{{3{n^2} - 4n}}{2}}}zy_1^{n - 1}y_2^{n - 1}y_3^n + X_{ - 2n}^2;\\
{X_{ - 1}}{X_3} = qX_1^2{y_2}{y_3} + w;\\
{X_{ - 3}}{X_1} = {q^{\frac{1}{2}}}X_{ - 1}^2{y_1} + w;\\
{X_0}{X_4} = {q^{\frac{1}{2}}}X_2^2{y_3} + z;\\
{X_{ - 2}}{X_2} = qX_0^2{y_1}{y_2} + z.
\end{array}\]
\end{lemma}

We  define $u=wz-q^{-\frac{1}{2}}{X^{(0,0,0,1,0,1)}}-q^{\frac{1}{2}}{X^{(0,0,0,0,1,0)}}$, and $u_n\ (n\in\ZZ)$  by\\
$$u_{-\infty}=0,\ u_0=1,\ u_1=u,\ u_2=u_{1}^{2}-2{X^{(0,0,0,1,1,1)}},\ u_{n+1}=u_{1}u_{n}-{X^{(0,0,0,1,1,1)}}u_{n-1}(n\geq2).$$

It follows that $u_n\in \mathcal{A}_{q} (Q)$ for any $n\in\ZZ$.

\section{The cluster multiplication formulas}
In this section, we will establish the cluster multiplication formulas for $\mathcal{A}_{q} (Q)$. We start with some technical lemmas as follows.
\begin{lemma}\label{2}
For any  $n,p\in \mathbb{Z}_{>0}$, we have that
$$(1)\ {u_1}{X^{(0,0,0,1,1,1)}} = {X^{(0,0,0,1,1,1)}}{u_1};\ (2)\ {u_p}{u_n} = {u_n}{u_p};\  (3)\ {u_n}{X^{(0,0,0,1,1,1)}} = {X^{(0,0,0,1,1,1)}}{u_n}.$$
\end{lemma}

\begin{proof}
(1)\ Note that $u_1={X^{(1,0,-1,1,1,1)}} + {X^{(0,-1,-1,1,1,0)}} + {X^{(-1,-1,0,1,0,0)}} + {X^{(-1,0,1,0,0,0)}}$. Thus, the proof follows from a direct calculation.

(2)\ The proof follows from (1) and induction.

(3)\ According to (1), (2) and by induction,  the statement follows immediately.
\end{proof}

By Lemma \ref{2}, we can get the following Lemma \ref{3} and Lemma \ref{0}  easily.

\begin{lemma}\label{3}
For any  $n,p\in \mathbb{Z}_{>0}$, we have that
\[ {u_n}{u_p}=\begin{cases}
{u_{n + p}} + {X^{p(0,0,0,1,1,1)}}{u_{n - p}} & \text{ if }    n > p \ge 1; \nonumber\\
{u_{2n}} + 2{X^{n(0,0,0,1,1,1)}} & \text{ if }    n = p.\nonumber
\end{cases}
\]
\end{lemma}

\begin{lemma}\label{0}
For any  $n\in \mathbb{Z}_{\geq 0}$, we have that
$$(1)\ \overline{u_n} = {u_n};\ \ (2)\ \overline{{X^{(0,0,0,1,1,1)}}{u_n}} = {X^{(0,0,0,1,1,1)}}{u_n}.$$
\end{lemma}

\begin{lemma}\label{regular}
For any  $n\in \mathbb{Z}_{>0}$, we have that
$(1)\ {u_n}w = w{u_n};\ (2)\ {u_n}z = z{u_n}.$
\end{lemma}
\begin{proof}
By a direct calculation, the equations hold for $n=1$. Then the proof follows by induction on $n$.
\end{proof}

\begin{lemma}\label{4}
For  any $n\in \mathbb{Z}$, we have that
\[ {u_1}{X_n}=\begin{cases}
{q^{\frac{1}{2}}}{X_{-1}}{y_1} + {X_{3}} & \text{ if }    n=1; \\
q{X_{0}}{y_1}{y_2} + {X_{4}} & \text{ if }    n=2; \\
{q^{\frac{5}{2}}}{X_{n - 2}}{y_1}{y_2}{y_3} + {X_{n + 2}} & \text{ if }    n\ge 3; \\
{q^{ - \frac{1}{2}}}{X_{2}}{y_3} + {X_{ - 2}} & \text{ if }    n=0; \\
{X_{1}}{y_2}{y_3} + {X_{ - 3}} & \text{ if }    n=-1; \\
{q^{\frac{1}{2}}}{X_{ n + 2}}{y_1}{y_2}{y_3} + {X_{ n - 2}} & \text{ if }    n \leq -2.
\end{cases}
\]
\end{lemma}
\begin{proof}
When $n=-1,0,1,2$, the proof follows from a direct calculation. We only prove it for $n\ge 3$, and it  is similar for $n \leq -2$.

When $n = 3$, we have
\begin{eqnarray*}
{u_1}{X_3}
 &=& (q{X_0}{y_1}{y_2}X_2^{ - 1} + {X_4}X_2^{ - 1}){X_3}\\
 &=& q({q^{\frac{1}{2}}}{X_1}{X_2}{y_3} + 1){y_1}{y_2}X_2^{ - 1} + {X_3}{X_4}X_2^{ - 1}\\
 &=& {q^{\frac{5}{2}}}{X_1}{y_1}{y_2}{y_3} + (q{y_1}{y_2} + {X_3}{X_{\rm{4}}})X_2^{ - 1}\\
 &=& {q^{\frac{5}{2}}}{X_1}{y_1}{y_2}{y_3} + {X_{\rm{5}}}{\rm{.}}
\end{eqnarray*}

Suppose that the case holds for $3 \leq n \leq k$. When $n=k+1$, by Lemma \ref{1} we have

(a)\ if $k$ is odd, the proof follows that
\begin{eqnarray*}
{u_1}{X_{k + 1}}
 &=& ({q^{\frac{5}{2}}}{X_{k - 2}}{y_1}{y_2}{y_3}X_k^{ - 1} + {X_{k + 2}}X_k^{ - 1}){X_{k + 1}}\\
 &=& {q^{\frac{7}{2}}}({q^{\frac{{3{{(\frac{{k - 1}}{2})}^2} - 4(\frac{{k - 1}}{2})}}{2}}}y_1^{\frac{{k - 1}}{2}}y_2^{\frac{{k - 1}}{2} - 1}y_3^{\frac{{k - 1}}{2} - 1} + {X_{k -
 1}}{X_k}){y_1}{y_2}{y_3}X_k^{ - 1} + {X_{k + 1}}{X_{k + 2}}X_k^{ - 1}\\
 &=& {q^{\frac{5}{2}}}{X_{k - 1}}{y_1}{y_2}{y_3} + ({q^{\frac{{3{{(\frac{{k + 1}}{2})}^2} - 4(\frac{{k + 1}}{2}) + 2}}{2}}}y_1^{\frac{{k + 1}}{2}}y_2^{\frac{{k + 1}}{2} - 1}y_3^{\frac{{k + 1}}{2} -
 1} + {X_{k + 1}}{X_{k + 2}})X_k^{ - 1}\\
 &=& {q^{\frac{5}{2}}}{X_{k - 1}}{y_1}{y_2}{y_3} + {X_{k + 3}};
\end{eqnarray*}

(b)\ if  $k$ is even, the proof is similar.

Hence, the proof is finished.
\end{proof}

We firstly give the cluster multiplication formulas between $u_n$ and cluster variables.
\begin{theorem}\label{5}
For any $m, n \in \ZZ_{>0}$, we have that

(1)\ ${u_n}{X_m} = {q^{\frac{{3{n^2} + 2n}}{2}}}{X_{m - 2n}}y_1^ny_2^ny_3^n + {X_{m + 2n}}\ for\ m \ge 2n+1;$

(2)\ ${u_n}{X_{2m - 1}} = {q^{\frac{{3{m^2} - 4m + 2}}{2}}}{X_{2m - 1 - 2n}}y_1^my_2^{m - 1}y_3^{m - 1} + {X_{2m - 1 + 2n}}\ for\ 1 \le 2m - 1 \le 2n;$

(3)\ ${u_n}{X_{2m}} = {q^{\frac{{3{m^2} - 2m + 1}}{2}}}{X_{2m - 2n}}y_1^my_2^my_3^{m - 1} + {X_{2m + 2n}}\ for\ 1 \le 2m \le 2n;$

(4)\ ${u_n}{X_{ - m}} = {q^{\frac{{3{n^2} - 2n}}{2}}}{X_{ - m + 2n}}y_1^ny_2^ny_3^n + {X_{ - m - 2n}}\ for \ m \geq 2n;$

(5)\ ${u_n}{X_{ - (2m - 1)}} = {q^{\frac{{3{m^2} - 2m - 1}}{2}}}{X_{ - (2m - 1) + 2n}}y_1^{m - 1}y_2^my_3^m + {X_{ - (2m - 1) - 2n}}\ for\ 0 \le 2m - 1 \le 2n - 1;$

(6)\ ${u_n}{X_{ - (2m - 2)}} = {q^{\frac{{3{m^2} - 4m}}{2}}}{X_{ - (2m - 2) + 2n}}y_1^{m - 1}y_2^{m - 1}y_3^m + {X_{ - (2m - 2) - 2n}}\ for\ 0 \le 2m - 2 \le 2n - 1.$
\end{theorem}
\begin{proof}
We only prove (1)-(3), and  the proofs of (4)-(6) are similar.

(1)\ When $n = 1$, by Lemma \ref{4} the case holds. When $n=2$, by Lemma \ref{4} we have
\begin{eqnarray*}
{u_{2}}{X_m}
 &=& ({u_1^2} - 2{X^{(0,0,0,1,1,1)}}){X_m}\\
 &=& {u_1}({q^{\frac{5}{2}}}{X_{m - 2}}{y_1}{y_2}{y_3} + {X_{m + 2}})- 2{q^{\frac{3}{2}}}{y_1}{y_2}{y_3}{X_{m}}\\
 &=& {q^{\frac{5}{2}}}({q^{\frac{5}{2}}}{X_{m - 4}}{y_1}{y_2}{y_3} + {X_{m}})y_1y_2y_3 + {q^{\frac{5}{2}}}{X_{m}}{y_1}{y_2}{y_3}
 + {X_{m + 4}} - 2{q^{\frac{3}{2}}}{y_1}{y_2}{y_3}{X_{m}}\\
 &=& {q^{8}}{X_{m - 4}}{y_1^{2}}{y_2^{2}}{y_3^{2}} + {X_{m + 4}}.
\end{eqnarray*}

Suppose that the case holds for $2 \leq n \leq k$. When $n=k+1$, the proof follows that
\begin{eqnarray*}
{u_{k + 1}}{X_m}&=& ({u_1}{u_k} - {X^{(0,0,0,1,1,1)}}{u_{k - 1}}){X_m}\\
 &=& {u_1}({q^{\frac{{3{k^2} + 2k}}{2}}}{X_{m - 2k}}y_1^ky_2^ky_3^k + {X_{m + 2k}})\\
 &&- {X^{(0,0,0,1,1,1)}}({q^{\frac{{3{{(k - 1)}^2} + 2(k - 1)}}{2}}}{X_{m - 2(k - 1)}}y_1^{k - 1}y_2^{k - 1}y_3^{k - 1} + {X_{m + 2(k - 1)}})\\
 &=& {q^{\frac{{3{k^2} + 2k}}{2}}}({q^{\frac{5}{2}}}{X_{m - 2k - 2}}{y_1}{y_2}{y_3} + {X_{m - 2k + 2}})y_1^ky_2^ky_3^k + {q^{\frac{5}{2}}}{X_{m + 2k - 2}}{y_1}{y_2}{y_3}\\
 &&+ {X_{m + 2k + 2}} - {q^{\frac{{3{{(k - 1)}^2} + 2(k - 1)}}{2}}}{q^{\frac{5}{2}}}{X_{m - 2k + 2}}{y_1}{y_2}{y_3}y_1^{k - 1}y_2^{k - 1}y_3^{k - 1}\\
 &&- {q^{\frac{5}{2}}}{X_{m + 2k - 2}}{y_1}{y_2}{y_3}\\
 &=& {q^{\frac{{3{{(k + 1)}^2} + 2(k + 1)}}{2}}}{X_{m - 2(k + 1)}}y_1^{k + 1}y_2^{k + 1}y_3^{k + 1} + {X_{m + 2(k + 1)}}.
\end{eqnarray*}

(2)\ When $n = 1$, by Lemma \ref{4} the case holds. When $n = 2$, by Lemma \ref{4} we have
\begin{eqnarray*}
{u_{2}}{X_1}
 &=& ({u_1^2} - 2{X^{(0,0,0,1,1,1)}}){X_1}\\
 &=& {u_1}({q^{\frac{1}{2}}}{X_{ - 1}}{y_1} + {X_3})- 2{q^{\frac{5}{2}}}{X_{1}}{y_1}{y_2}{y_3}\\
 &=& {q^{\frac{1}{2}}}({X_{1}}{y_2}{y_3} + {X_{-3}})y_1 + {q^{\frac{5}{2}}}{X_{1}}{y_1}{y_2}{y_3}
 + {X_{5}} - 2{q^{\frac{5}{2}}}{X_{1}}{y_1}{y_2}{y_3}\\
 &=& {q^{\frac{1}{2}}}{X_{-3}}{y_1} + {X_{5}},
\end{eqnarray*}
and  ${u_{2}}{X_3} = {q^3}{X_{-1}}{y_1^2}{y_2}{y_3} + {X_{7}}$ similarly.

Suppose that the case holds for $2 \leq n \leq k$. When $n=k+1$, we have that

(a)\ if $1 \le 2m - 1 \le 2k-2$, the proof follows that
\begin{eqnarray*}
&&{u_{k + 1}}{X_{2m - 1}}\\
 &=& ({u_1}{u_k} - {X^{(0,0,0,1,1,1)}}{u_{k - 1}}){X_{2m - 1}}\\
 &=& {u_1}({q^{\frac{{3{m^2} - 4m + 2}}{2}}}{X_{2m - 1 - 2k}}y_1^my_2^{m - 1}y_3^{m - 1} + {X_{2m - 1 + 2k}})\\
 &&- {q^{\frac{3}{2}}}{y_1}{y_2}{y_3}({q^{\frac{{3{m^2} - 4m + 2}}{2}}}{X_{2m + 1 - 2k}}y_1^my_2^{m - 1}y_3^{m - 1} + {X_{2m - 3 + 2k}})\\
 &=& {q^{\frac{{3{m^2} - 4m + 2}}{2}}}({q^{\frac{1}{2}}}{X_{2m + 1 - 2k}}{y_1}{y_2}{y_3} + {X_{2m - 3 - 2k}})y_1^my_2^{m - 1}y_3^{m - 1} + {q^{\frac{5}{2}}}{X_{2m - 3 + 2k}}{y_1}{y_2}{y_3}\\
 &&+ {X_{2m + 1 + 2k}} - {q^{\frac{3}{2}}}{y_1}{y_2}{y_3}{q^{\frac{{3{m^2} - 4m + 2}}{2}}}{X_{2m + 1 - 2k}}y_1^my_2^{m - 1}y_3^{m - 1} - {q^{\frac{5}{2}}}{X_{2m - 3 + 2k}}{y_1}{y_2}{y_3}\\
 &=& {q^{\frac{{3{m^2} - 4m + 2}}{2}}}{X_{2m - 1 - 2(k + 1)}}y_1^my_2^{m - 1}y_3^{m - 1} + {X_{2m - 1 + 2(k + 1)}};
\end{eqnarray*}

(b)\ if  $2m-1=2k-1$, the proof follows that
\begin{eqnarray*}
{u_{k + 1}}{X_{2k - 1}}
 &=& ({u_1}{u_k} - {X^{(0,0,0,1,1,1)}}{u_{k - 1}}){X_{2k - 1}}\\
 &=& {u_1}({q^{\frac{{3{k^2} - 4k + 2}}{2}}}{X_{ - 1}}y_1^ky_2^{k - 1}y_3^{k - 1} + {X_{4k - 1}})\\
 &&- {q^{\frac{3}{2}}}{y_1}{y_2}{y_3}({q^{\frac{{3{{(k - 1)}^2} + 2(k - 1)}}{2}}}{X_1}y_1^{k - 1}y_2^{k - 1}y_3^{k - 1} + {X_{4k - 3}})\\
 &=& {q^{\frac{{3{k^2} - 4k + 2}}{2}}}({X_1}{y_2}{y_3} + {X_{ - 3}})y_1^ky_2^{k - 1}y_3^{k - 1} + {q^{\frac{5}{2}}}{X_{4k - 3}}{y_1}{y_2}{y_3} + {X_{4k + 1}}\\
 &&- {q^{\frac{3}{2}}}{y_1}{y_2}{y_3}{q^{\frac{{3{{(k - 1)}^2} + 2(k - 1)}}{2}}}{X_1}y_1^{k - 1}y_2^{k - 1}y_3^{k - 1} - {q^{\frac{5}{2}}}{X_{4k - 3}}{y_1}{y_2}{y_3}\\
 &=& {q^{\frac{{3{k^2} - 4k + 2}}{2}}}{X_{ - 3}}y_1^ky_2^{k - 1}y_3^{k - 1} + {X_{4k + 1}};
\end{eqnarray*}

(c)\ if  $2m-1=2k+1$, the proof follows that
\begin{eqnarray*}
{u_{k + 1}}{X_{2k + 1}}
 &=& ({u_1}{u_k} - {X^{(0,0,0,1,1,1)}}{u_{k - 1}}){X_{2k + 1}}\\
 &=& {u_1}({q^{\frac{{3{k^2} + 2k}}{2}}}{X_1}y_1^ky_2^ky_3^k + {X_{4k + 1}})\\
 &&- {q^{\frac{3}{2}}}{y_1}{y_2}{y_3}({q^{\frac{{3{{(k - 1)}^2} + 2(k - 1)}}{2}}}{X_3}y_1^{k - 1}y_2^{k - 1}y_3^{k - 1} + {X_{4k - 1}})\\
 &=& {q^{\frac{{3{k^2} + 2k}}{2}}}({q^{\frac{1}{2}}}{X_{ - 1}}{y_1} + {X_3})y_1^ky_2^ky_3^k + {q^{\frac{5}{2}}}{X_{4k - 1}}{y_1}{y_2}{y_3} + {X_{4k + 3}}\\
 &&- {q^{\frac{3}{2}}}{y_1}{y_2}{y_3}{q^{\frac{{3{{(k - 1)}^2} + 2(k - 1)}}{2}}}{X_3}y_1^{k - 1}y_2^{k - 1}y_3^{k - 1} - {q^{\frac{5}{2}}}{X_{4k - 1}}{y_1}{y_2}{y_3}\\
 &=& {q^{\frac{{3{k^2} + 2k + 1}}{2}}}{X_{ - 1}}y_1^{k + 1}y_2^ky_3^k + {X_{4k + 3}}.
\end{eqnarray*}

(3)\ When $n = 1$, by Lemma \ref{4} the case holds. When $n = 2$, by Lemma \ref{4} we have
\begin{eqnarray*}
{u_{2}}{X_2}
 &=& ({u_1^2} - 2{X^{(0,0,0,1,1,1)}}){X_2}\\
 &=& {u_1}(q{X_0}{y_1}{y_2} + {X_4})- 2{q^{\frac{5}{2}}}{X_{2}}{y_1}{y_2}{y_3}\\
 &=& q({q^{-\frac{1}{2}}}{X_{2}}{y_3} + {X_{-2}}){y_1}{y_2} + {q^{\frac{5}{2}}}{X_{2}}{y_1}{y_2}{y_3}
 + {X_{6}} - 2{q^{\frac{5}{2}}}{X_{2}}{y_1}{y_2}{y_3}\\
 &=& {q}{X_{-2}}{y_1}{y_2} + {X_{6}},
\end{eqnarray*}
and ${u_{2}}{X_4}={q^{\frac{9}{2}}}{X_{0}}{y_1^2}{y_2^2}{y_3} + {X_{8}}$ similarly.

Suppose that the case holds for $2 \leq n \leq k$. When $n=k+1$, we have that

(a) if $1 \le 2m \le 2k-2$,  the proof follows that
\begin{eqnarray*}
{u_{k + 1}}{X_{2m}}
 &=& ({u_1}{u_k} - {X^{(0,0,0,1,1,1)}}{u_{k - 1}}){X_{2m}}\\
 &=& {u_1}({q^{\frac{{3{m^2} - 2m + 1}}{2}}}{X_{2m - 2k}}y_1^my_2^my_3^{m - 1} + {X_{2m + 2k}})\\
 &&- {q^{\frac{3}{2}}}{y_1}{y_2}{y_3}({q^{\frac{{3{m^2} - 2m + 1}}{2}}}{X_{2m + 2 - 2k}}y_1^my_2^my_3^{m - 1} + {X_{2m - 2 + 2k}})\\
 &=& {q^{\frac{{3{m^2} - 2m + 1}}{2}}}({q^{\frac{1}{2}}}{X_{2m + 2 - 2k}}{y_1}{y_2}{y_3} + {X_{2m - 2 - 2k}})y_1^my_2^my_3^{m - 1} + {q^{\frac{5}{2}}}{X_{2m - 2 + 2k}}{y_1}{y_2}{y_3}\\
 &&+ {X_{2m + 2 + 2k}} - {q^{\frac{3}{2}}}{y_1}{y_2}{y_3}{q^{\frac{{3{m^2} - 2m + 1}}{2}}}{X_{2m + 2 - 2k}}y_1^my_2^my_3^{m - 1} - {q^{\frac{5}{2}}}{X_{2m - 2 + 2k}}{y_1}{y_2}{y_3}\\
 &=& {q^{\frac{{3{m^2} - 2m + 1}}{2}}}{X_{2m - 2(k + 1)}}y_1^my_2^my_3^{m - 1} + {X_{2m + 2(k + 1)}};
\end{eqnarray*}

(b) if  $2m=2k$,  the proof follows that
\begin{eqnarray*}
{u_{k + 1}}{X_{2k}}
 &=& ({u_1}{u_k} - {X^{(0,0,0,1,1,1)}}{u_{k - 1}}){X_{2k}}\\
 &=& {u_1}({q^{\frac{{3{k^2} - 2k + 1}}{2}}}{X_0}y_1^ky_2^ky_3^{k - 1} + {X_{4k}})\\
 &&- {q^{\frac{3}{2}}}{y_1}{y_2}{y_3}({q^{\frac{{3{{(k - 1)}^2} + 2(k - 1)}}{2}}}{X_2}y_1^{k - 1}y_2^{k - 1}y_3^{k - 1} + {X_{4k - 2}})\\
 &=& {q^{\frac{{3{k^2} - 2k + 1}}{2}}}({q^{ - \frac{1}{2}}}{X_2}{y_3} + {X_{ - 2}})y_1^ky_2^ky_3^{k - 1} + {q^{\frac{5}{2}}}{X_{4k - 2}}{y_1}{y_2}{y_3} + {X_{4k + 2}}\\
 &&- {q^{\frac{3}{2}}}{y_1}{y_2}{y_3}{q^{\frac{{3{{(k - 1)}^2} + 2(k - 1)}}{2}}}{X_2}y_1^{k - 1}y_2^{k - 1}y_3^{k - 1} - {q^{\frac{5}{2}}}{X_{4k - 2}}{y_1}{y_2}{y_3}\\
 &=& {q^{\frac{{3{k^2} - 2k + 1}}{2}}}{X_{ - 2}}y_1^ky_2^ky_3^{k - 1} + {X_{4k + 2}};
\end{eqnarray*}

(c) if  $2m=2k+2$,  the proof follows that
\begin{eqnarray*}
{u_{k + 1}}{X_{2k + 2}}
 &=& ({u_1}{u_k} - {X^{(0,0,0,1,1,1)}}{u_{k - 1}}){X_{2k + 2}}\\
 &=& {u_1}({q^{\frac{{3{k^2} + 2k}}{2}}}{X_2}y_1^ky_2^ky_3^k + {X_{4k + 2}})\\
 &&- {q^{\frac{3}{2}}}{y_1}{y_2}{y_3}({q^{\frac{{3{{(k - 1)}^2} + 2(k - 1)}}{2}}}{X_4}y_1^{k - 1}y_2^{k - 1}y_3^{k - 1} + {X_{4k}})\\
 &=& {q^{\frac{{3{k^2} + 2k}}{2}}}(q{X_0}{y_1}{y_2} + {X_4})y_1^ky_2^ky_3^k + {q^{\frac{5}{2}}}{X_{4k}}{y_1}{y_2}{y_3} + {X_{4k + 4}}\\
 &&- {q^{\frac{3}{2}}}{y_1}{y_2}{y_3}{q^{\frac{{3{{(k - 1)}^2} + 2(k - 1)}}{2}}}{X_4}y_1^{k - 1}y_2^{k - 1}y_3^{k - 1} - {q^{\frac{5}{2}}}{X_{4k}}{y_1}{y_2}{y_3}\\
 &=& {q^{\frac{{3{k^2} + 4k + 2}}{2}}}{X_0}y_1^{k + 1}y_2^{k + 1}y_3^k + {X_{4k + 4}}.
\end{eqnarray*}
Hence, the proof is finished.
\end{proof}

\begin{remark}
By specializing $q=1$ and $y_{1}=y_{2}=y_{3}=1$, the obtained formulas can be viewed as generalizations of the constant coefficient
linear relations between cluster variables of type $A_2^{(1)}$.  Note that the constant coefficient
linear relations for cluster algebras from affine quivers has been studied by \cite{KS,Pa}.
\end{remark}

\begin{corollary}\label{cor}
For any  $n\in \mathbb{Z}_{>0}$, we have that

(1)\ ${X_{2n + 3}} = \sum\limits_{l = 0}^{\left\lfloor {\frac{n}{2}} \right\rfloor } {{q^{\frac{{5l}}{2}}}{u_{n - 2l}}{X_3}{{({y_1}{y_2}{y_3})}^l}}  - \sum\limits_{l = 0}^{\left\lfloor {\frac{{n - 1}}{2}} \right\rfloor } {{q^{\frac{{5(l + 1)}}{2}}}{u_{n - 1 - 2l}}{X_1}{{({y_1}{y_2}{y_3})}^{l + 1}};}$

(2)\ ${X_{2n + 4}} = \sum\limits_{l = 0}^{\left\lfloor {\frac{n}{2}} \right\rfloor } {{q^{\frac{{5l}}{2}}}{u_{n - 2l}}{X_4}{{({y_1}{y_2}{y_3})}^l}}  - \sum\limits_{l = 0}^{\left\lfloor {\frac{{n - 1}}{2}} \right\rfloor } {{q^{\frac{{5(l + 1)}}{2}}}{u_{n - 1 - 2l}}{X_2}{{({y_1}{y_2}{y_3})}^{l + 1}};}$

(3)\ ${X_{ - 2n}} = \sum\limits_{l = 0}^{\left\lfloor {\frac{n}{2}} \right\rfloor } {{q^{\frac{l}{2}}}{u_{n - 2l}}{X_0}{{({y_1}{y_2}{y_3})}^l}}  - \sum\limits_{l = 0}^{\left\lfloor {\frac{{n - 1}}{2}} \right\rfloor } {{q^{\frac{{l - 1}}{2}}}{u_{n - 1 - 2l}}{X_2}{y_3}{{({y_1}{y_2}{y_3})}^l};}$

(4)\ ${X_{ - (2n + 1)}} = \sum\limits_{l = 0}^{\left\lfloor {\frac{n}{2}} \right\rfloor } {{q^{\frac{l}{2}}}{u_{n - 2l}}{X_{ - 1}}{{({y_1}{y_2}{y_3})}^l}}  - \sum\limits_{l = 0}^{\left\lfloor {\frac{{n - 1}}{2}} \right\rfloor } {{q^{\frac{l}{2}}}{u_{n - 1 - 2l}}{X_1}{y_2}{y_3}{{({y_1}{y_2}{y_3})}^l}.}$
\end{corollary}
\begin{proof}
We only prove (1)and (2), and the proofs of (3) and (4)  are similar.

(1)\ When $n=1$ and $n=2$, the case follows from Theorem \ref{5}. Suppose that the case holds for $2\leq n \le k$. For $n=k+1$, by Lemma \ref{2} and \ref{4} we have that

(a)\ if $k$ is even, the proof follows that
\begin{eqnarray*}
{X_{2k + 5}} &=& {u_1}{X_{2k + 3}} - {q^{\frac{5}{2}}}{X_{2k + 1}}{y_1}{y_2}{y_3}\\
 &=& {u_1}({u_k}{X_3} + {q^{\frac{5}{2}}}{u_{k - 2}}{X_3}{y_1}{y_2}{y_3} +  \cdots  + {q^{\frac{{5k}}{4}}}{u_0}{X_3}{({y_1}{y_2}{y_3})^{\frac{k}{2}}})\\
 &&- {u_1}({q^{\frac{5}{2}}}{u_{k - 1}}{X_1}{y_1}{y_2}{y_3} + {q^5}{u_{k - 3}}{X_1}{({y_1}{y_2}{y_3})^2} +  \cdots  + {q^{\frac{{5k}}{4}}}{u_1}{X_1}{({y_1}{y_2}{y_3})^{\frac{k}{2}}})\\
 &&- {q^{\frac{5}{2}}}({u_{k - 1}}{X_3} + {q^{\frac{5}{2}}}{u_{k - 3}}{X_3}{y_1}{y_2}{y_3} +  \cdots  + {q^{\frac{{5k - 10}}{4}}}{u_1}{X_3}{({y_1}{y_2}{y_3})^{\frac{{k - 2}}{2}}}){y_1}{y_2}{y_3}\\
 &&+ {q^{\frac{5}{2}}}({q^{\frac{5}{2}}}{u_{k - 2}}{X_1}{y_1}{y_2}{y_3} + {q^5}{u_{k - 4}}{X_1}{({y_1}{y_2}{y_3})^2} +  \cdots  + {q^{\frac{{5k}}{4}}}{u_0}{X_1}{({y_1}{y_2}{y_3})^{\frac{k}{2}}}){y_1}{y_2}{y_3}\\
 &=& \sum\limits_{l = 0}^{\left\lfloor {\frac{{k + 1}}{2}} \right\rfloor } {{q^{\frac{{5l}}{2}}}{u_{k + 1 - 2l}}{X_3}{{({y_1}{y_2}{y_3})}^l}}  - \sum\limits_{l = 0}^{\left\lfloor {\frac{k}{2}} \right\rfloor } {{q^{\frac{{5(l + 1)}}{2}}}{u_{k - 2l}}{X_1}{{({y_1}{y_2}{y_3})}^{l + 1}};}
\end{eqnarray*}

(b)\ if $k$ is odd, the proof is similar.

(2)\ When $n=1$ and $n=2$, the case follows from Theorem \ref{5}. Suppose that the case holds for $2\leq n \le k$. For $n=k+1$, by Lemma \ref{2} and \ref{4} we have that

(a)\ if $k$ is even, the proof follows that
\begin{eqnarray*}
{X_{2k + 6}} &=& {u_1}{X_{2k + 4}} - {q^{\frac{5}{2}}}{X_{2k + 2}}{y_1}{y_2}{y_3}\\
 &=& {u_1}({u_k}{X_4} + {q^{\frac{5}{2}}}{u_{k - 2}}{X_4}{y_1}{y_2}{y_3} +  \cdots  + {q^{\frac{{5k}}{4}}}{u_0}{X_4}{({y_1}{y_2}{y_3})^{\frac{k}{2}}})\\
 &&- {u_1}({q^{\frac{5}{2}}}{u_{k - 1}}{X_2}{y_1}{y_2}{y_3} + {q^5}{u_{k - 3}}{X_2}{({y_1}{y_2}{y_3})^2} +  \cdots  + {q^{\frac{{5k}}{4}}}{u_1}{X_2}{({y_1}{y_2}{y_3})^{\frac{k}{2}}})\\
 &&- {q^{\frac{5}{2}}}({u_{k - 1}}{X_4} + {q^{\frac{5}{2}}}{u_{k - 3}}{X_4}{y_1}{y_2}{y_3} +  \cdots  + {q^{\frac{{5k - 10}}{4}}}{u_1}{X_4}{({y_1}{y_2}{y_3})^{\frac{{k - 2}}{2}}}){y_1}{y_2}{y_3}\\
 &&+ {q^{\frac{5}{2}}}({q^{\frac{5}{2}}}{u_{k - 2}}{X_2}{y_1}{y_2}{y_3} + {q^5}{u_{k - 4}}{X_2}{({y_1}{y_2}{y_3})^2} +  \cdots  + {q^{\frac{{5k}}{4}}}{u_0}{X_2}{({y_1}{y_2}{y_3})^{\frac{k}{2}}}){y_1}{y_2}{y_3}\\
 &=& \sum\limits_{l = 0}^{\left\lfloor {\frac{{k + 1}}{2}} \right\rfloor } {{q^{\frac{{5l}}{2}}}{u_{k + 1 - 2l}}{X_4}{{({y_1}{y_2}{y_3})}^l}}  - \sum\limits_{l = 0}^{\left\lfloor {\frac{k}{2}} \right\rfloor } {{q^{\frac{{5(l + 1)}}{2}}}{u_{k - 2l}}{X_2}{{({y_1}{y_2}{y_3})}^{l + 1}};}
\end{eqnarray*}

(b)\ if $k$ is odd, the proof is similar.
\end{proof}

\begin{remark}
Note that $z={X_0}{X_4} - {q^{\frac{1}{2}}}X_2^2{y_3}$. By the definition of $u_n$, Corollary \ref{cor} and the equation $X_{ - 1} = w{X_0} - {q^{ - \frac{1}{2}}}{X_1}{y_3}$ provide an exact expression for every quantum cluster variable as  a polynomial in $\mathbb{Z}\mathbb{P}[X_0,X_1,X_2,X_3,X_4,w]$.
\end{remark}

\begin{theorem}\label{7}
For any  $m,n\in \mathbb{Z}_{\geq0}$, we have that

(1)\ when $m$ is odd, then
\begin{eqnarray*}
(i)\ {X_m}{X_{m + 2n + 3}} &=& {X_{m + n + 1}}{X_{m + n + 2}}\\
 &&+ {q^{\frac{{5m - 5}}{4}}}{y_1}{({y_1}{y_2}{y_3})^{\frac{{m - 1}}{2}}}\sum\limits_{l = 0}^{\left\lfloor {\frac{n}{2}} \right\rfloor } {(\sum\limits_{i = 0}^l {{q^{\frac{{2i - 1 + 3l}}{2}}}} )}
 {u_{n - 2l}}{({y_1}{y_2}{y_3})^l}\\
 &&+ {q^{\frac{{5m - 5}}{4}}}{y_1}{y_2}{({y_1}{y_2}{y_3})^{\frac{{m - 1}}{2}}}\sum\limits_{l = 0}^{\left\lfloor {\frac{{n - 1}}{2}} \right\rfloor } {(\sum\limits_{i = 0}^l {{q^{\frac{{2i +
 3l}}{2}}}} )} {u_{n - 1 - 2l}}{({y_1}{y_2}{y_3})^l};\\
(ii)\ {X_m}{X_{m + 2n}} &=& {X_{m + 2\left\lfloor {\frac{n}{2}} \right\rfloor }}{X_{m + 2\left\lceil {\frac{n}{2}} \right\rceil }}\\
 &&+ {q^{\frac{{5m - 5}}{4}}}{y_1}w{({y_1}{y_2}{y_3})^{\frac{{m - 1}}{2}}}\sum\limits_{l = 0}^{\left\lfloor {\frac{{n - 2}}{2}} \right\rfloor } {(\sum\limits_{i = 0}^l {{q^{\frac{{2i - 1 +
 3l}}{2}}}} )} {u_{n - 2 - 2l}}{({y_1}{y_2}{y_3})^l};
\end{eqnarray*}

\begin{eqnarray*}
(iii)\ {X_{ - m}}{X_{ - (m + 2n + 3)}} &=& {X_{ - (m + n + 1)}}{X_{ - (m + n + 2)}}\\
 &&+ {q^{\frac{{m - 1}}{4}}}{y_2}{y_3}{({y_1}{y_2}{y_3})^{\frac{{m - 1}}{2}}}\sum\limits_{l = 0}^{\left\lfloor {\frac{n}{2}} \right\rfloor } {(\sum\limits_{i = 0}^l {{q^{\frac{{2i + 2 + l}}{2}}}} )}
 {u_{n - 2l}}{({y_1}{y_2}{y_3})^l}\\
 &&+ {q^{\frac{{m - 1}}{4}}}{y_3}{({y_1}{y_2}{y_3})^{\frac{{m + 1}}{2}}}\sum\limits_{l = 0}^{\left\lfloor {\frac{{n - 1}}{2}} \right\rfloor } {(\sum\limits_{i = 0}^l {{q^{\frac{{2i + 2 + l}}{2}}}}
 )} {u_{n - 1 - 2l}}{({y_1}{y_2}{y_3})^l};\\
(iv)\ \ \ \ {X_{ - m}}{X_{ - (m + 2n)}} &=& {X_{ - (m + 2\left\lfloor {\frac{n}{2}} \right\rfloor )}}{X_{ - (m + 2\left\lceil {\frac{n}{2}} \right\rceil )}}\\
 &&+ {q^{\frac{{m - 1}}{4}}}w{y_2}{y_3}{({y_1}{y_2}{y_3})^{\frac{{m - 1}}{2}}}\sum\limits_{l = 0}^{\left\lfloor {\frac{{n - 2}}{2}} \right\rfloor } {(\sum\limits_{i = 0}^l {{q^{\frac{{2i + 2 +
 l}}{2}}}} )} {u_{n - 2 - 2l}}{({y_1}{y_2}{y_3})^l};
\end{eqnarray*}

(2)\ when $m$ is even, then
\begin{eqnarray*}
(i)\ {X_m}{X_{m + 2n + 3}} &=& {X_{m + n + 1}}{X_{m + n + 2}}\\
 &&+ {q^{\frac{{5m - 10}}{4}}}{y_1}{y_2}{({y_1}{y_2}{y_3})^{\frac{{m - 2}}{2}}}\sum\limits_{l = 0}^{\left\lfloor {\frac{n}{2}} \right\rfloor } {(\sum\limits_{i = 0}^l {{q^{\frac{{2i + 3l}}{2}}}} )}{u_{n - 2l}}{({y_1}{y_2}{y_3})^l}\\
 &&+ {q^{\frac{{5m}}{4}}}{y_1}{({y_1}{y_2}{y_3})^{\frac{m}{2}}}\sum\limits_{l = 0}^{\left\lfloor {\frac{{n - 1}}{2}} \right\rfloor } {(\sum\limits_{i = 0}^l {{q^{\frac{{2i - 1 + 3l}}{2}}}} )} {u_{n
 - 1 - 2l}}{({y_1}{y_2}{y_3})^l}\ for\ m > 0;\\
(ii)\ \ \ \ {X_m}{X_{m + 2n}} &=& {X_{m + 2\left\lfloor {\frac{n}{2}} \right\rfloor }}{X_{m + 2\left\lceil {\frac{n}{2}} \right\rceil }}\\
&&+ {q^{\frac{{5m - 10}}{4}}}{y_1}{y_2}z{({y_1}{y_2}{y_3})^{\frac{{m - 2}}{2}}}\sum\limits_{l = 0}^{\left\lfloor {\frac{{n - 2}}{2}} \right\rfloor } {(\sum\limits_{i = 0}^l {{q^{\frac{{2i +
 3l}}{2}}}} )} {u_{n - 2 - 2l}}{({y_1}{y_2}{y_3})^l}\ for\ m > 0;\\
\end{eqnarray*}
\begin{eqnarray*}
(iii)\ {X_{ - m}}{X_{ - (m + 2n + 3)}} &=& {X_{ - (m + n + 1)}}{X_{ - (m + n + 2)}}\\
 &&+ {q^{\frac{{m - 2}}{4}}}{y_3}{({y_1}{y_2}{y_3})^{\frac{m}{2}}}\sum\limits_{l = 0}^{\left\lfloor {\frac{n}{2}} \right\rfloor } {(\sum\limits_{i = 0}^l {{q^{\frac{{2i + 2 + l}}{2}}}} )} {u_{n -
 2l}}{({y_1}{y_2}{y_3})^l}\\
 &&+ {q^{\frac{{m - 2}}{4}}}{y_2}{y_3}{({y_1}{y_2}{y_3})^{\frac{m}{2}}}\sum\limits_{l = 0}^{\left\lfloor {\frac{{n - 1}}{2}} \right\rfloor } {(\sum\limits_{i = 0}^l {{q^{\frac{{2i + 3 + l}}{2}}}} )}
 {u_{n - 1 - 2l}}{({y_1}{y_2}{y_3})^l};\\
(iv)\ \ \ \ {X_{ - m}}{X_{ - (m + 2n)}} &=& {X_{ - (m + 2\left\lfloor {\frac{n}{2}} \right\rfloor )}}{X_{ - (m + 2\left\lceil {\frac{n}{2}} \right\rceil )}}\\
 &&+ {q^{\frac{{m - 2}}{4}}}z{y_3}{({y_1}{y_2}{y_3})^{\frac{m}{2}}}\sum\limits_{l = 0}^{\left\lfloor {\frac{{n - 2}}{2}} \right\rfloor } {(\sum\limits_{i = 0}^l {{q^{\frac{{2i + 2 + l}}{2}}}} )}
 {u_{n - 2 - 2l}}{({y_1}{y_2}{y_3})^l}.
\end{eqnarray*}
\end{theorem}
\begin{proof}
We only prove (1), and the proof of (2) is similar.

(1)(i)\ When $n = 0$, the proof follows from Lemma \ref{1}.
When $n = 1$, by Lemma \ref{1}, \ref{2} and \ref{4} we have
\begin{eqnarray*}
&&{X_m}{X_{m + 5}}\\
 &=& {X_m}({X_{m + 3}}{u_1} - {q^{ - \frac{5}{2}}}{y_3}{y_2}{y_1}{X_{m + 1}})\\
 &=& {X_{m + 1}}{X_{m + 2}}{u_1} - {q^{\frac{3}{2}}}{X_m}{X_{m + 1}}{y_1}{y_2}{y_3} + {q^{\frac{{5m - 7}}{4}}}{y_1}{({y_1}{y_2}{y_3})^{\frac{{m - 1}}{2}}}{u_1}\\
 &=& {X_{m + 1}}({q^{ - \frac{5}{2}}}{y_3}{y_2}{y_1}{X_m} + {X_{m + 4}}) - {q^{\frac{3}{2}}}{X_m}{X_{m + 1}}{y_1}{y_2}{y_3} + {q^{\frac{{5m - 7}}{4}}}{y_1}{({y_1}{y_2}{y_3})^{\frac{{m - 1}}{2}}}{u_1}\\
 &=& {X_{m + 2}}{X_{m + 3}} + {q^{\frac{{5m - 5}}{4}}}{y_1}{y_2}{({y_1}{y_2}{y_3})^{\frac{{m - 1}}{2}}} + {q^{\frac{{5m - 7}}{4}}}{y_1}{({y_1}{y_2}{y_3})^{\frac{{m - 1}}{2}}}{u_1}.
\end{eqnarray*}

Suppose that the case holds for $1\leq n \le k$. When $n=k+1$, by Lemma \ref{1}, \ref{2} and \ref{4} we have that

(a)\ if $k$ is odd, then
\begin{eqnarray*}
&&{X_m}{X_{m + 2k + 5}}\\
 &=& {X_m}({X_{m + 2k + 3}}{u_1} - {q^{ - \frac{5}{2}}}{y_3}{y_2}{y_1}{X_{m + 2k + 1}})\\
 &=& {X_{m + k + 1}}{X_{m + k + 2}}{u_1} - {q^{\frac{3}{2}}}{X_{m + k}}{X_{m + k + 1}}{y_1}{y_2}{y_3} + {q^{\frac{{5m - 5}}{4}}}{y_1}{({y_1}{y_2}{y_3})^{\frac{{m - 1}}{2}}}\\
 &&\cdot ({q^{ - \frac{1}{2}}}{u_k}{u_1} + (q + {q^2}){u_{k - 2}}{u_1}{y_1}{y_2}{y_3} +  \cdots  + ({q^{\frac{{3k - 5}}{4}}} + {q^{\frac{{3k - 1}}{4}}} +  \cdots  + {q^{\frac{{5k - 7}}{4}}})\\
 &&\cdot u_1^2{({y_1}{y_2}{y_3})^{\frac{{k - 1}}{2}}}) + {q^{\frac{{5m - 5}}{4}}}{y_1}{y_2}{({y_1}{y_2}{y_3})^{\frac{{m - 1}}{2}}}({u_{k - 1}}{u_1} + ({q^{\frac{3}{2}}} + {q^{\frac{5}{2}}}){u_{k - 3}}{u_1}{y_1}{y_2}{y_3}\\
 &&+  \cdots  + ({q^{\frac{{3k - 3}}{4}}} + {q^{\frac{{3k + 1}}{4}}} +  \cdots  + {q^{\frac{{5k - 5}}{4}}}){u_1}{({y_1}{y_2}{y_3})^{\frac{{k - 1}}{2}}}) - {q^{\frac{3}{2}}}{q^{\frac{{5m - 5}}{4}}}{y_1}{({y_1}{y_2}{y_3})^{\frac{{m - 1}}{2}}}\\
 &&\cdot ({q^{ - \frac{1}{2}}}{u_{k - 1}} + (q + {q^2}){u_{k - 3}}{y_1}{y_2}{y_3} +  \cdots  + ({q^{\frac{{3k - 5}}{4}}} + {q^{\frac{{3k - 1}}{4}}} +  \cdots  + {q^{\frac{{5k - 7}}{4}}}){u_0}\\
 &&\cdot {({y_1}{y_2}{y_3})^{\frac{{k - 1}}{2}}}){y_1}{y_2}{y_3} - {q^{\frac{3}{2}}}{q^{\frac{{5m - 5}}{4}}}{y_1}{y_2}{({y_1}{y_2}{y_3})^{\frac{{m - 1}}{2}}}({u_{k - 2}} + ({q^{\frac{3}{2}}} + {q^{\frac{5}{2}}})\\
 &&\cdot {u_{k - 4}}{y_1}{y_2}{y_3} +  \cdots  + ({q^{\frac{{3k - 9}}{4}}} + {q^{\frac{{3k - 5}}{4}}} +  \cdots  + {q^{\frac{{5k - 15}}{4}}}){u_1}{({y_1}{y_2}{y_3})^{\frac{{k - 3}}{2}}}){y_1}{y_2}{y_3}\\
 &=& {X_{m + k + 1}}{X_{m + k + 2}}{u_1} - {q^{\frac{3}{2}}}{X_{m + k}}{X_{m + k + 1}}{y_1}{y_2}{y_3} + {q^{\frac{{5m - 5}}{4}}}{y_1}{({y_1}{y_2}{y_3})^{\frac{{m - 1}}{2}}}\\
 &&\cdot ({q^{ - \frac{1}{2}}}{u_{k + 1}} + (q + {q^2}){u_{k - 1}}{y_1}{y_2}{y_3} +  \cdots  + ({q^{\frac{{3k + 1}}{4}}} + {q^{\frac{{3k + 5}}{4}}} +  \cdots  + {q^{\frac{{5k - 1}}{4}}})\\
 &&\cdot {u_0}{({y_1}{y_2}{y_3})^{\frac{{k + 1}}{2}}}) + {q^{\frac{{5m - 5}}{4}}}{y_1}{y_2}{({y_1}{y_2}{y_3})^{\frac{{m - 1}}{2}}}({u_k} + ({q^{\frac{3}{2}}} + {q^{\frac{5}{2}}}){u_{k - 2}}{y_1}{y_2}{y_3}\\
 &&+  \cdots  + ({q^{\frac{{3k - 3}}{4}}} + {q^{\frac{{3k + 1}}{4}}} +  \cdots  + {q^{\frac{{5k - 5}}{4}}}){u_1}{({y_1}{y_2}{y_3})^{\frac{{k - 1}}{2}}}),
\end{eqnarray*}
thus, the proof follows that
\begin{eqnarray*}
&&{X_{m + k + 1}}{X_{m + k + 2}}{u_1} - {q^{\frac{3}{2}}}{X_{m + k}}{X_{m + k + 1}}{y_1}{y_2}{y_3}\\
 &=& {X_{m + k + 1}}({q^{ - \frac{5}{2}}}{y_3}{y_2}{y_1}{X_{m + k}} + {X_{m + k + 4}}) - {q^{\frac{3}{2}}}{X_{m + k}}{X_{m + k + 1}}{y_1}{y_2}{y_3}\\
 &=& {q^{\frac{{5m - 5}}{4}}}{y_1}{({y_1}{y_2}{y_3})^{\frac{{m - 1}}{2}}}{q^{\frac{{5k + 3}}{4}}}{({y_1}{y_2}{y_3})^{\frac{{k + 1}}{2}}} + {X_{m + k + 2}}{X_{m + k + 3}};
\end{eqnarray*}

(b)\ if $k$ is even, the proof is similar.

(ii)\ It is obvious for $n=0$ and $n=1$. When $n = 2$, the proof follows from  Lemma \ref{1}.
When $n = 3$, by Lemma \ref{2} and \ref{4} we have
\begin{eqnarray*}
&&{X_m}{X_{m + 6}}\\
 &=& {X_m}({X_{m + 4}}{u_1} - {q^{ - \frac{5}{2}}}{y_3}{y_2}{y_1}{X_{m + 2}})\\
 &=& {X_{m + 2}}{X_{m + 2}}{u_1} - {q^{\frac{3}{2}}}{X_m}{X_{m + 2}}{y_1}{y_2}{y_3} + {q^{\frac{{5m - 7}}{4}}}{y_1}w{({y_1}{y_2}{y_3})^{\frac{{m - 1}}{2}}}{u_1}\\
 &=& {X_{m + 2}}({q^{ - \frac{5}{2}}}{y_3}{y_2}{y_1}{X_m} + {X_{m + 4}}) - {q^{\frac{3}{2}}}{X_m}{X_{m + 2}}{y_1}{y_2}{y_3} + {q^{\frac{{5m - 7}}{4}}}{y_1}w{u_1}{({y_1}{y_2}{y_3})^{\frac{{m - 1}}{2}}}\\
 &=& {X_{m + 2}}{X_{m + 4}} + {q^{\frac{{5m - 7}}{4}}}{y_1}w{u_1}{({y_1}{y_2}{y_3})^{\frac{{m - 1}}{2}}}.
\end{eqnarray*}

Suppose that the case holds for $3 \leq n \le k$. When $n=k+1$, by Lemma \ref{2} and \ref{4} we have that

(a)\ if $k$ is odd, then
\begin{eqnarray*}
&&{X_m}{X_{m + 2k + 2}}\\
 &=& {X_m}({X_{m + 2k}}{u_1} - {q^{ - \frac{5}{2}}}{y_3}{y_2}{y_1}{X_{m + 2k - 2}})\\
 &=& {X_{m + k - 1}}{X_{m + k + 1}}{u_1} - {q^{\frac{3}{2}}}{X_{m + k - 1}}{X_{m + k - 1}}{y_1}{y_2}{y_3} + {q^{\frac{{5m - 5}}{4}}}{y_1}w{({y_1}{y_2}{y_3})^{\frac{{m - 1}}{2}}}\\
 &&\cdot ({q^{ - \frac{1}{2}}}{u_{k - 2}}{u_1} + (q + {q^2}){u_{k - 4}}{u_1}{y_1}{y_2}{y_3} +  \cdots  + ({q^{\frac{{3k - 11}}{4}}} + {q^{\frac{{3k - 7}}{4}}} +  \cdots  + {q^{\frac{{5k - 17}}{4}}})\\
 &&\cdot u_1^2{({y_1}{y_2}{y_3})^{\frac{{k - 3}}{2}}}) - {q^{\frac{3}{2}}}{q^{\frac{{5m - 5}}{4}}}{y_1}w{({y_1}{y_2}{y_3})^{\frac{{m - 1}}{2}}}({q^{ - \frac{1}{2}}}{u_{k - 3}} + (q + {q^2}){u_{k - 5}}{y_1}{y_2}{y_3}\\
 &&+  \cdots  + ({q^{\frac{{3k - 11}}{4}}} + {q^{\frac{{3k - 7}}{4}}} +  \cdots  + {q^{\frac{{5k - 17}}{4}}}){u_0}{({y_1}{y_2}{y_3})^{\frac{{k - 3}}{2}}}){y_1}{y_2}{y_3}\\
 &=& {X_{m + k - 1}}{X_{m + k + 1}}{u_1} - {q^{\frac{3}{2}}}{X_{m + k - 1}}{X_{m + k - 1}}{y_1}{y_2}{y_3} + {q^{\frac{{5m - 5}}{4}}}{y_1}w{({y_1}{y_2}{y_3})^{\frac{{m - 1}}{2}}}({q^{ - \frac{1}{2}}}{u_{k - 1}}\\
 &&+ (q + {q^2}){u_{k - 3}}{y_1}{y_2}{y_3} +  \cdots  + ({q^{\frac{{3k - 5}}{4}}} + {q^{\frac{{3k - 1}}{4}}} +  \cdots  + {q^{\frac{{5k - 11}}{4}}}){u_0}{({y_1}{y_2}{y_3})^{\frac{{k - 1}}{2}}}),
\end{eqnarray*}
thus, the proof follows that
\begin{eqnarray*}
&&{X_{m + k - 1}}{X_{m + k + 1}}{u_1} - {q^{\frac{3}{2}}}{X_{m + k - 1}}{X_{m + k - 1}}{y_1}{y_2}{y_3}\\
 &=& {X_{m + k - 1}}({q^{ - \frac{5}{2}}}{y_3}{y_2}{y_1}{X_{m + k - 1}} + {X_{m + k + 3}}) - {q^{\frac{3}{2}}}{X_{m + k - 1}}{X_{m + k - 1}}{y_1}{y_2}{y_3}\\
 &=& {q^{\frac{{5m - 5}}{4}}}{y_1}w{({y_1}{y_2}{y_3})^{\frac{{m - 1}}{2}}}{q^{\frac{{5k - 7}}{4}}}{({y_1}{y_2}{y_3})^{\frac{{k - 1}}{2}}} + {X_{m + k + 1}^{2}};
\end{eqnarray*}

(b)\ if $k$ is even, the proof is similar.

(iii)\ When $n = 0$, the proof follows from Lemma \ref{1}. When $n = 1$, by Lemma \ref{1}, \ref{2} and \ref{4} we have
\begin{eqnarray*}
&&{X_{ - m}}{X_{ - (m + 5)}}\\
 &=& {X_{ - m}}({X_{ - (m + 3)}}{u_1} - {q^{ - \frac{1}{2}}}{y_3}{y_2}{y_1}{X_{ - (m + 1)}})\\
   \end{eqnarray*}
\begin{eqnarray*}
 &=& {X_{ - (m + 1)}}{X_{ - (m + 2)}}{u_1} - {q^{\frac{3}{2}}}{X_{ - m}}{X_{ - (m + 1)}}{y_1}{y_2}{y_3} + {q^{\frac{{m + 3}}{4}}}{y_2}{y_3}{({y_1}{y_2}{y_3})^{\frac{{m - 1}}{2}}}{u_1}\\
 &=& {X_{ - (m + 1)}}({q^{ - \frac{1}{2}}}{y_3}{y_2}{y_1}{X_{ - m}} + {X_{ - (m + 4)}})\\
 &&- {q^{\frac{3}{2}}}{X_{ - m}}{X_{ - (m + 1)}}{y_1}{y_2}{y_3} + {q^{\frac{{m + 3}}{4}}}{y_2}{y_3}{({y_1}{y_2}{y_3})^{\frac{{m - 1}}{2}}}{u_1}\\
 &=& {X_{ - (m + 2)}}{X_{ - (m + 3)}} + {q^{\frac{{m + 3}}{4}}}{y_3}{({y_1}{y_2}{y_3})^{\frac{{m + 1}}{2}}} + {q^{\frac{{m + 3}}{4}}}{y_2}{y_3}{({y_1}{y_2}{y_3})^{\frac{{m - 1}}{2}}}{u_1}.
\end{eqnarray*}

Suppose that the case holds for $1\leq n \le k$. When $n=k+1$, by Lemma \ref{1}, \ref{2} and \ref{4} we have that

(a)\ if $k$ is odd, then
\begin{eqnarray*}
&&{X_{ - m}}{X_{ - (m + 2k + 5)}}\\
 &=& {X_{ - m}}({X_{ - (m + 2k + 3)}}{u_1} - {q^{ - \frac{1}{2}}}{y_3}{y_2}{y_1}{X_{ - (m + 2k + 1)}})\\
 &=& {X_{ - (m + k + 1)}}{X_{ - (m + k + 2)}}{u_1} - {q^{\frac{3}{2}}}{X_{ - (m + k)}}{X_{ - (m + k + 1)}}{y_1}{y_2}{y_3} + {q^{\frac{{m - 1}}{4}}}{y_2}{y_3}{({y_1}{y_2}{y_3})^{\frac{{m - 1}}{2}}}\\
 &&\cdot (q{u_k}{u_1} + ({q^{\frac{3}{2}}} + {q^{\frac{5}{2}}}){u_{k - 2}}{u_1}{y_1}{y_2}{y_3} +  \cdots  + ({q^{\frac{{k + 3}}{4}}} + {q^{\frac{{k + 7}}{4}}} +  \cdots  + {q^{\frac{{3k + 1}}{4}}})\\
 &&\cdot u_1^2{({y_1}{y_2}{y_3})^{\frac{{k - 1}}{2}}}) + {q^{\frac{{m - 1}}{4}}}{y_3}{({y_1}{y_2}{y_3})^{\frac{{m + 1}}{2}}}(q{u_{k - 1}}{u_1} + ({q^{\frac{3}{2}}} + {q^{\frac{5}{2}}}){u_{k - 3}}{u_1}{y_1}{y_2}{y_3}\\
 &&+  \cdots  + ({q^{\frac{{k + 3}}{4}}} + {q^{\frac{{k + 7}}{4}}} +  \cdots  + {q^{\frac{{3k + 1}}{4}}}){u_1}{({y_1}{y_2}{y_3})^{\frac{{k - 1}}{2}}}) - {q^{\frac{3}{2}}}{q^{\frac{{m - 1}}{4}}}{y_2}{y_3}{({y_1}{y_2}{y_3})^{\frac{{m - 1}}{2}}}\\
 &&\cdot (q{u_{k - 1}} + ({q^{\frac{3}{2}}} + {q^{\frac{5}{2}}}){u_{k - 3}}{y_1}{y_2}{y_3} +  \cdots  + ({q^{\frac{{k + 3}}{4}}} + {q^{\frac{{k + 7}}{4}}} +  \cdots  + {q^{\frac{{3k + 1}}{4}}}){u_0}\\
 &&\cdot {({y_1}{y_2}{y_3})^{\frac{{k - 1}}{2}}}){y_1}{y_2}{y_3} - {q^{\frac{3}{2}}}{q^{\frac{{m - 1}}{4}}}{y_3}{({y_1}{y_2}{y_3})^{\frac{{m + 1}}{2}}}(q{u_{k - 2}} + ({q^{\frac{3}{2}}} + {q^{\frac{5}{2}}})\\
 &&\cdot {u_{k - 4}}{y_1}{y_2}{y_3} +  \cdots  + ({q^{\frac{{k + 1}}{4}}} + {q^{\frac{{k + 5}}{4}}} +  \cdots  + {q^{\frac{{3k - 5}}{4}}}){u_1}{({y_1}{y_2}{y_3})^{\frac{{k - 3}}{2}}}){y_1}{y_2}{y_3}\\
 &=& {X_{ - (m + k + 1)}}{X_{ - (m + k + 2)}}{u_1} - {q^{\frac{3}{2}}}{X_{ - (m + k)}}{X_{ - (m + k + 1)}}{y_1}{y_2}{y_3} + {q^{\frac{{m - 1}}{4}}}{y_2}{y_3}{({y_1}{y_2}{y_3})^{\frac{{m - 1}}{2}}}\\
 &&\cdot (q{u_{k + 1}} + ({q^{\frac{3}{2}}} + {q^{\frac{5}{2}}}){u_{k - 1}}{y_1}{y_2}{y_3} +  \cdots  + ({q^{\frac{{k + 9}}{4}}} + {q^{\frac{{k + 13}}{4}}} +  \cdots  + {q^{\frac{{3k + 7}}{4}}})\\
 &&\cdot {u_0}{({y_1}{y_2}{y_3})^{\frac{{k + 1}}{2}}}) + {q^{\frac{{m - 1}}{4}}}{y_3}{({y_1}{y_2}{y_3})^{\frac{{m + 1}}{2}}}(q{u_k} + ({q^{\frac{3}{2}}} + {q^{\frac{5}{2}}}){u_{k - 2}}{y_1}{y_2}{y_3}\\
 &&+  \cdots  + ({q^{\frac{{k + 3}}{4}}} + {q^{\frac{{k + 7}}{4}}} +  \cdots  + {q^{\frac{{3k + 1}}{4}}}){u_1}{({y_1}{y_2}{y_3})^{\frac{{k - 1}}{2}}}),
\end{eqnarray*}
thus, the proof follows that
\begin{eqnarray*}
&&{X_{ - (m + k + 1)}}{X_{ - (m + k + 2)}}{u_1} - {q^{\frac{3}{2}}}{X_{ - (m + k)}}{X_{ - (m + k + 1)}}{y_1}{y_2}{y_3}\\
 &=& {X_{ - (m + k + 1)}}({q^{ - \frac{1}{2}}}{y_3}{y_2}{y_1}{X_{ - (m + k)}} + {X_{ - (m + k + 4)}}) - {q^{\frac{3}{2}}}{X_{ - (m + k)}}{X_{ - (m + k + 1)}}{y_1}{y_2}{y_3}\\
 &=& {q^{\frac{{m - 1}}{4}}}{y_2}{y_3}{({y_1}{y_2}{y_3})^{\frac{{m - 1}}{2}}}{q^{\frac{{k + 5}}{4}}}{({y_1}{y_2}{y_3})^{\frac{{k + 1}}{2}}} + {X_{ - (m + k + 2)}}{X_{ - (m + k + 3)}};
\end{eqnarray*}

(b)\ if $k$ is even, the proof is similar.

(iv)\ It is obvious for $n=0$ and $n=1$. When $n = 2$, the proof follows from Lemma \ref{1}.
When $n = 3$, by Lemma \ref{2} and \ref{4} we have
\begin{eqnarray*}
&&{X_{ - m}}{X_{ - (m + 6)}}\\
 &=& {X_{ - m}}({X_{ - (m + 4)}}{u_1} - {q^{ - \frac{1}{2}}}{y_3}{y_2}{y_1}{X_{ - (m + 2)}})\\
 &=& {X_{ - (m + 2)}}{X_{ - (m + 2)}}{u_1} - {q^{\frac{3}{2}}}{X_{ - m}}{X_{ - (m + 2)}}{y_1}{y_2}{y_3} + {q^{\frac{{m + 3}}{4}}}w{y_2}{y_3}{({y_1}{y_2}{y_3})^{\frac{{m - 1}}{2}}}{u_1}\\
 \end{eqnarray*}
\begin{eqnarray*}
 &=& {X_{ - (m + 2)}}({q^{ - \frac{1}{2}}}{y_3}{y_2}{y_1}{X_{ - m}} + {X_{ - (m + 4)}})\\
 &&- {q^{\frac{3}{2}}}{X_{ - m}}{X_{ - (m + 2)}}{y_1}{y_2}{y_3} + {q^{\frac{{m + 3}}{4}}}w{y_2}{y_3}{({y_1}{y_2}{y_3})^{\frac{{m - 1}}{2}}}{u_1}\\
 &=& {X_{ - (m + 2)}}{X_{ - (m + 4)}} + {q^{\frac{{m + 3}}{4}}}w{y_2}{y_3}{({y_1}{y_2}{y_3})^{\frac{{m - 1}}{2}}}{u_1}.
\end{eqnarray*}

Suppose that the case holds for $3\leq n \le k$. When $n=k+1$, by Lemma \ref{2} and \ref{4} we have that

(a)\ if $k$ is odd, then
\begin{eqnarray*}
&&{X_{ - m}}{X_{ - (m + 2k + 2)}}\\
 &=& {X_{ - m}}({X_{ - (m + 2k)}}{u_1} - {q^{ - \frac{1}{2}}}{y_3}{y_2}{y_1}{X_{ - (m + 2k - 2)}})\\
 &=& {X_{ - (m + k - 1)}}{X_{ - (m + k + 1)}}{u_1} - {q^{\frac{3}{2}}}{X_{ - (m + k - 1)}}{X_{ - (m + k - 1)}}{y_1}{y_2}{y_3} + {q^{\frac{{m - 1}}{4}}}w{y_2}{y_3}{({y_1}{y_2}{y_3})^{\frac{{m - 1}}{2}}}\\
 &&\cdot (q{u_{k - 2}}{u_1} + ({q^{\frac{3}{2}}} + {q^{\frac{5}{2}}}){u_{k - 4}}{u_1}{y_1}{y_2}{y_3} +  \cdots  + ({q^{\frac{{k + 1}}{4}}} + {q^{\frac{{k + 5}}{4}}} +  \cdots  + {q^{\frac{{3k - 5}}{4}}})\\
 &&\cdot u_1^2{({y_1}{y_2}{y_3})^{\frac{{k - 3}}{2}}}) - {q^{\frac{3}{2}}}{q^{\frac{{m - 1}}{4}}}w{y_2}{y_3}{({y_1}{y_2}{y_3})^{\frac{{m - 1}}{2}}}(q{u_{k - 3}} + ({q^{\frac{3}{2}}} + {q^{\frac{5}{2}}}){u_{k - 5}}{y_1}{y_2}{y_3}\\
 &&+  \cdots  + ({q^{\frac{{k + 1}}{4}}} + {q^{\frac{{k + 5}}{4}}} +  \cdots  + {q^{\frac{{3k - 5}}{4}}}){u_0}{({y_1}{y_2}{y_3})^{\frac{{k - 3}}{2}}}){y_1}{y_2}{y_3}\\
 &=& {X_{ - (m + k - 1)}}{X_{ - (m + k + 1)}}{u_1} - {q^{\frac{3}{2}}}{X_{ - (m + k - 1)}}{X_{ - (m + k - 1)}}{y_1}{y_2}{y_3} + {q^{\frac{{m - 1}}{4}}}w{y_2}{y_3}{({y_1}{y_2}{y_3})^{\frac{{m - 1}}{2}}}\\
 &&\cdot (q{u_{k - 1}} + ({q^{\frac{3}{2}}} + {q^{\frac{5}{2}}}){u_{k - 3}}{y_1}{y_2}{y_3} +  \cdots  + ({q^{\frac{{k + 7}}{4}}} + {q^{\frac{{k + 11}}{4}}} +  \cdots  + {q^{\frac{{3k + 1}}{4}}}){({y_1}{y_2}{y_3})^{\frac{{k - 1}}{2}}}),
\end{eqnarray*}
thus, the proof follows that
\begin{eqnarray*}
&&{X_{ - (m + k - 1)}}{X_{ - (m + k + 1)}}{u_1} - {q^{\frac{3}{2}}}{X_{ - (m + k - 1)}}{X_{ - (m + k - 1)}}{y_1}{y_2}{y_3}\\
 &=& {X_{ - (m + k - 1)}}({q^{ - \frac{1}{2}}}{y_3}{y_2}{y_1}{X_{ - (m + k - 1)}} + {X_{ - (m + k + 3)}}) - {q^{\frac{3}{2}}}{X_{ - (m + k - 1)}}{X_{ - (m + k - 1)}}{y_1}{y_2}{y_3}\\
 &=& {q^{\frac{{m - 1}}{4}}}w{y_2}{y_3}{({y_1}{y_2}{y_3})^{\frac{{m - 1}}{2}}}{q^{\frac{{k + 3}}{4}}}{({y_1}{y_2}{y_3})^{\frac{{k - 1}}{2}}} + {X_{ - (m + k + 1)}^{2}};
\end{eqnarray*}

(b)\ if $k$ is even, the proof is similar.
\end{proof}

\begin{lemma}\label{8}
For any $m\in \mathbb{Z}_{\geq 2}$, we have that

(1)\ when $m$ is odd, then
\begin{eqnarray*}
(i)\ {X_{ - m}}{X_1} &=& {q^{\frac{{m - 2}}{{\rm{2}}}}}{y_1}{X_{ - m + 2\left\lfloor {\frac{{m + 1}}{4}} \right\rfloor }}{X_{ - m + 2\left\lceil {\frac{{m + 1}}{4}} \right\rceil }} + w\sum\limits_{l = 0}^{\left\lfloor {\frac{{m - 3}}{4}} \right\rfloor } {(\sum\limits_{i = 0}^l {{q^{\frac{{2i + 3l}}{2}}}} )} {u_{\frac{{m - 3 - 4l}}{2}}}{({y_1}{y_2}{y_3})^l};\\
(ii)\ {X_{ - m}}{X_{\rm{2}}} &=& {q^{\frac{{m + {\rm{1}}}}{{\rm{2}}}}}{y_1}{y_2}{X_{\frac{{1 - m}}{2}}}{X_{\frac{{3 - m}}{2}}}+ \sum\limits_{l = 0}^{\left\lfloor {\frac{{m - 1}}{4}} \right\rfloor } {(\sum\limits_{i = 0}^l {{q^{\frac{{2i + 3l}}{2}}}} )} {u_{\frac{{m - 1 - 4l}}{2}}}{({y_1}{y_2}{y_3})^l}\\
&& + {y_2}\sum\limits_{l = 0}^{\left\lfloor {\frac{{m - 3}}{4}} \right\rfloor } {(\sum\limits_{i = 0}^l {{q^{\frac{{2i + 1 + 3l}}{2}}}} )} {u_{\frac{{m - 3 - 4l}}{2}}}{({y_1}{y_2}{y_3})^l};
\end{eqnarray*}

(2)\ when $m$ is even, then
\begin{eqnarray*}
(i)\ {X_{ - m}}{X_1} &=& {q^{\frac{{m - {\rm{1}}}}{{\rm{2}}}}}{y_1}{X_{\frac{{ - m}}{2}}}{X_{\frac{{2 - m}}{2}}}+ \sum\limits_{l = 0}^{\left\lfloor {\frac{{m - 2}}{4}} \right\rfloor } {(\sum\limits_{i = 0}^l {{q^{\frac{{2i + 3l}}{2}}}} )} {u_{\frac{{m - 2 - 4l}}{2}}}{({y_1}{y_2}{y_3})^l}\\
 &&+ {y_1}{y_3}\sum\limits_{l = 0}^{\left\lfloor {\frac{{m - 4}}{4}} \right\rfloor } {(\sum\limits_{i = 0}^l {{q^{\frac{{2i + 2 + 3l}}{2}}}} )} {u_{\frac{{m - 4 - 4l}}{2}}}{({y_1}{y_2}{y_3})^l};\\
(ii)\ {X_{ - m}}{X_2} &=& {q^{\frac{m}{{\rm{2}}}}}{y_1}{y_2}{X_{ - m + 2\left\lfloor {\frac{{m + 2}}{4}} \right\rfloor }}{X_{ - m + 2\left\lceil {\frac{{m + 2}}{4}} \right\rceil }} + z\sum\limits_{l = 0}^{\left\lfloor {\frac{{m - 2}}{4}} \right\rfloor } {(\sum\limits_{i = 0}^l {{q^{\frac{{2i + 3l}}{2}}}} )} {u_{\frac{{m - 2 - 4l}}{2}}}{({y_1}{y_2}{y_3})^l}.
\end{eqnarray*}
\end{lemma}
\begin{proof}
We only prove (1), and the proof of (2) is similar.

(1)\ (i)\ When $m = 3$, by Lemma \ref{1} the case holds. When $m = 5$, by Lemma \ref{1} and \ref{4} we have
\begin{eqnarray*}
{X_{ - 5}}{X_1}
 &=& ({u_1}{X_{ - 3}} - {q^{\frac{1}{2}}}{X_{ - 1}}{y_1}{y_2}{y_3}){X_1}\\
 &=& {u_1}({q^{\frac{1}{2}}}X_{ - 1}^2{y_1} + w) - {q^{\frac{1}{2}}}{X_{ - 1}}{y_1}{y_2}{y_3}{X_1}\\
 &=& w{u_1} + {q^{\frac{1}{2}}}({X_1}{y_2}{y_3} + {X_{ - 3}}){X_{ - 1}}{y_1}- {q^{\frac{1}{2}}}{X_{ - 1}}{y_1}{y_2}{y_3}{X_1}\\
 &=& w{u_1} + {q^{\frac{3}{2}}}{y_1}{X_{ - 3}}{X_{ - 1}}.
\end{eqnarray*}

Suppose that the case holds for $5\leq m \leq k$. When $m=k+2$, by Lemma \ref{1}, \ref{2} and \ref{4} we have that

(a)\ if $4 \mid k-1$, then
\begin{eqnarray*}
&&{X_{ - (k + 2)}}{X_1}\\
 &=& ({u_1}{X_{ - k}} - {q^{\frac{1}{2}}}{X_{ - k{\rm{ + 2}}}}{y_1}{y_2}{y_{\rm{3}}}){X_1}\\
 &=& {u_1}{q^{\frac{{k - 2}}{2}}}{y_1}{X_{\frac{{ - 1 - k}}{2}}}{X_{\frac{{3 - k}}{2}}} + {u_1}w({u_{\frac{{k - 3}}{2}}} + ({q^{\frac{3}{2}}} + {q^{\frac{5}{2}}}){u_{\frac{{k - 7}}{2}}}{y_1}{y_2}{y_3}\\
 &&+  \cdots  + ({q^{\frac{{3k - 15}}{8}}} + {q^{\frac{{3k - 7}}{8}}} +  \cdots  + {q^{\frac{{5k - 25}}{8}}})u_1^{}{({y_1}{y_2}{y_3})^{\frac{{k - 5}}{4}}})\\
 &&- {q^{\frac{3}{2}}}{q^{\frac{{k - 4}}{2}}}{y_1}{X_{\frac{{3 - k}}{2}}}{X_{\frac{{3 - k}}{2}}}{y_1}{y_2}{y_3} - w({u_{\frac{{k - 5}}{2}}} + ({q^{\frac{3}{2}}} + {q^{\frac{5}{2}}}){u_{\frac{{k - 9}}{2}}}{y_1}{y_2}{y_3}\\
 &&+  \cdots  + ({q^{\frac{{3k - 15}}{8}}} + {q^{\frac{{3k - 7}}{8}}} +  \cdots  + {q^{\frac{{5k - 25}}{8}}}){u_0}{({y_1}{y_2}{y_3})^{\frac{{k - 5}}{4}}}){q^{\frac{3}{2}}}{y_1}{y_2}{y_3}\\
 &=& {q^{\frac{k}{2}}}{y_1}{u_1}{X_{\frac{{ - 1 - k}}{2}}}{X_{\frac{{3 - k}}{2}}} - {q^{\frac{3}{2}}}{q^{\frac{{k - 4}}{2}}}{y_1}{X_{\frac{{3 - k}}{2}}}{X_{\frac{{3 - k}}{2}}}{y_1}{y_2}{y_3} + w({u_{\frac{{k - 1}}{2}}} + ({q^{\frac{3}{2}}}\\
 &&+ {q^{\frac{5}{2}}}){u_{\frac{{k - 5}}{2}}}{y_1}{y_2}{y_3} +  \cdots  + ({q^{\frac{{3k - 3}}{8}}} + {q^{\frac{{3k + 5}}{8}}} +  \cdots  + {q^{\frac{{5k - 13}}{8}}})u_0^{}{({y_1}{y_2}{y_3})^{\frac{{k - 1}}{4}}}),
\end{eqnarray*}
thus, the proof follows that
\begin{eqnarray*}
&&{q^{\frac{k}{2}}}{y_1}{u_1}{X_{\frac{{ - 1 - k}}{2}}}{X_{\frac{{3 - k}}{2}}} - {q^{\frac{3}{2}}}{q^{\frac{{k - 4}}{2}}}{y_1}{X_{\frac{{3 - k}}{2}}}{X_{\frac{{3 - k}}{2}}}{y_1}{y_2}{y_3}\\
 &=& {q^{\frac{k}{2}}}{y_1}({q^{\frac{1}{2}}}{X_{\frac{{3 - k}}{2}}}{y_1}{y_2}{y_3} + {X_{\frac{{ - 5 - k}}{2}}}){X_{\frac{{3 - k}}{2}}} - {q^{\frac{3}{2}}}{q^{\frac{{k - 4}}{2}}}{y_1}{X_{\frac{{3 - k}}{2}}}{X_{\frac{{3 - k}}{2}}}{y_1}{y_2}{y_3}\\
 &=& {q^{\frac{{5k - 5}}{8}}}w{({y_1}{y_2}{y_3})^{\frac{{k - 1}}{4}}} + {q^{\frac{k}{2}}}{y_1}{X_{\frac{{ - k - 1}}{2}}^{2}};
\end{eqnarray*}

(b)\ if $4 \nmid k-1$, the proof is similar.

(ii)\ When $m = 3$,  by Lemma \ref{1} and \ref{4} we have
\begin{eqnarray*}
{X_{ - 3}}{X_2}
 &=& ({u_1}{X_{ - 1}} - {X_1}{y_2}{y_3}){X_2}\\
 &=& {u_1}({q^{\frac{1}{2}}}{X_0}{X_1}{y_2} + 1) - {X_1}{y_2}{y_3}{X_2}\\
 &=& {u_1} + {q^{\frac{1}{2}}}({q^{ - \frac{1}{2}}}{X_2}{y_3} + {X_{ - 2}}){X_1}{y_2} - {X_1}{y_2}{y_3}{X_2}\\
 &=& {u_1} + {q^{\frac{1}{2}}}{y_2} + {q^2}{y_1}{y_2}{X_{ - 1}}{X_0}.
\end{eqnarray*}

When $m = 5$, by (i), Lemma \ref{1} and \ref{4} we have
\begin{eqnarray*}
{X_{ - 5}}{X_2}
 &=& ({u_1}{X_{ - 3}} - {q^{\frac{1}{2}}}{X_{ - 1}}{y_1}{y_2}{y_3}){X_2}\\
 &=& {u_1}({u_1} + {q^{\frac{1}{2}}}{y_2} + {q^2}{y_1}{y_2}{X_{ - 1}}{X_0}) - {q^{\frac{3}{2}}}({q^{\frac{1}{2}}}{X_0}{X_1}{y_2} + 1){y_1}{y_2}{y_3}\\
 &=& {u_2} + {q^{\frac{1}{2}}}{y_2}{u_1} + {q^{\frac{3}{2}}}{y_1}{y_2}{y_3} + q({X_1}{y_2}{y_3} + {X_{ - 3}}){X_0}{y_1}{y_2} - {q^2}{X_0}{X_1}{y_2}{y_1}{y_2}{y_3}\\
 &=& {u_2} + {q^{\frac{1}{2}}}{y_2}{u_1} + ({q^{\frac{3}{2}}} + {q^{\frac{5}{2}}}){y_1}{y_2}{y_3} + {q^3}{y_1}{y_2}{X_{ - 2}}{X_{ - 1}}.
\end{eqnarray*}

Suppose that the case holds for $5\leq m \leq k$. When $m=k+2$, by Lemma \ref{1}, \ref{2} and \ref{4} we have that

(a)\ if $4 \mid k-1$, then
\begin{eqnarray*}
&&{X_{ - (k + 2)}}{X_2}\\
 &=& ({u_1}{X_{ - k}} - {q^{\frac{1}{2}}}{X_{ - k{\rm{ + 2}}}}{y_1}{y_2}{y_{\rm{3}}}){X_2}\\
 &=& {u_1}{q^{\frac{{k + 1}}{2}}}{y_1}{y_2}{X_{\frac{{1 - k}}{2}}}{X_{\frac{{3 - k}}{2}}} + {u_1}({u_{\frac{{k - 1}}{2}}} + ({q^{\frac{3}{2}}} + {q^{\frac{5}{2}}}){u_{\frac{{k - 5}}{2}}}{y_1}{y_2}{y_3} +  \cdots \\
 &&+ ({q^{\frac{{3k - 3}}{8}}} + {q^{\frac{{3k + 5}}{8}}} +  \cdots  + {q^{\frac{{5k - 5}}{8}}})u_0^{}{({y_1}{y_2}{y_3})^{\frac{{k - 1}}{4}}}) + {u_1}{y_2}({q^{\frac{1}{2}}}{u_{\frac{{k - 3}}{2}}}\\
 &&+ ({q^2} + {q^3}){u_{\frac{{k - 7}}{2}}}{y_1}{y_2}{y_3} +  \cdots  + ({q^{\frac{{3k - 11}}{8}}} + {q^{\frac{{3k - 3}}{8}}} +  \cdots  + {q^{\frac{{5k - 21}}{8}}})\\
 &&\cdot {u_1}{({y_1}{y_2}{y_3})^{\frac{{k - 5}}{4}}}) - {q^{\frac{3}{2}}}{q^{\frac{{k - 1}}{2}}}{y_1}{y_2}{X_{\frac{{3 - k}}{2}}}{X_{\frac{{5 - k}}{2}}}{y_1}{y_2}{y_3} - ({u_{\frac{{k - 3}}{2}}} + \\
&&({q^{\frac{3}{2}}} + {q^{\frac{5}{2}}}){u_{\frac{{k - 7}}{2}}}{y_1}{y_2}{y_3} +  \cdots  + ({q^{\frac{{3k - 15}}{8}}} + {q^{\frac{{3k - 7}}{8}}} +  \cdots  + {q^{\frac{{5k - 25}}{8}}}){u_1}\\
 &&\cdot {({y_1}{y_2}{y_3})^{\frac{{k - 5}}{4}}}){q^{\frac{3}{2}}}{y_1}{y_2}{y_3} - {y_2}({q^{\frac{1}{2}}}{u_{\frac{{k - 5}}{2}}} + ({q^2} + {q^3}){u_{\frac{{k - 9}}{2}}}{y_1}{y_2}{y_3}\\
 &&+  \cdots  + ({q^{\frac{{3k - 11}}{8}}} + {q^{\frac{{3k - 3}}{8}}} +  \cdots  + {q^{\frac{{5k - 21}}{8}}}){u_0}{({y_1}{y_2}{y_3})^{\frac{{k - 5}}{4}}}){q^{\frac{3}{2}}}{y_1}{y_2}{y_3}\\
 &=& {q^{\frac{{k + 3}}{2}}}{y_1}{y_2}{u_1}{X_{\frac{{1 - k}}{2}}}{X_{\frac{{3 - k}}{2}}} - {q^{\frac{3}{2}}}{q^{\frac{{k - 1}}{2}}}{y_1}{y_2}{X_{\frac{{3 - k}}{2}}}{X_{\frac{{5 - k}}{2}}}{y_1}{y_2}{y_3} + ({u_{\frac{{k + 1}}{2}}} + ({q^{\frac{3}{2}}} + {q^{\frac{5}{2}}})\\
 &&\cdot {u_{\frac{{k - 3}}{2}}}{y_1}{y_2}{y_3} +  \cdots  + ({q^{\frac{{3k - 3}}{8}}} + {q^{\frac{{3k + 5}}{8}}} +  \cdots  + {q^{\frac{{5k - 5}}{8}}})u_1^{}{({y_1}{y_2}{y_3})^{\frac{{k - 1}}{4}}}) + {y_2}({q^{\frac{1}{2}}}{u_{\frac{{k - 1}}{2}}}\\
 &&+ ({q^2} + {q^3}){u_{\frac{{k - 5}}{2}}}{y_1}{y_2}{y_3} +  \cdots  + ({q^{\frac{{3k + 1}}{8}}} + {q^{\frac{{3k + 9}}{8}}} +  \cdots  + {q^{\frac{{5k - 9}}{8}}})u_0{({y_1}{y_2}{y_3})^{\frac{{k - 1}}{4}}}),
\end{eqnarray*}
thus, the proof follows that
\begin{eqnarray*}
&&{q^{\frac{{k + 3}}{2}}}{y_1}{y_2}{u_1}{X_{\frac{{1 - k}}{2}}}{X_{\frac{{3 - k}}{2}}} - {q^{\frac{3}{2}}}{q^{\frac{{k - 1}}{2}}}{y_1}{y_2}{X_{\frac{{3 - k}}{2}}}{X_{\frac{{5 - k}}{2}}}{y_1}{y_2}{y_3}\\
 &=& {q^{\frac{{k + 3}}{2}}}{y_1}{y_2}({q^{\frac{1}{2}}}{X_{\frac{{5 - k}}{2}}}{y_1}{y_2}{y_3} + {X_{\frac{{ - 3 - k}}{2}}}){X_{\frac{{3 - k}}{2}}} - {q^{\frac{3}{2}}}{q^{\frac{{k - 1}}{2}}}{y_1}{y_2}{X_{\frac{{3 - k}}{2}}}{X_{\frac{{5 - k}}{2}}}{y_1}{y_2}{y_3}\\
 &=& {q^{\frac{{5k - 1}}{8}}}{y_2}{({y_1}{y_2}{y_3})^{\frac{{k - 1}}{4}}} + {q^{\frac{{k + 3}}{2}}}{y_1}{y_2}{X_{\frac{{ - k + 1}}{2}}}{X_{\frac{{ - k - 1}}{2}}};
\end{eqnarray*}

(b)\ if $4 \nmid k-1$, the proof is similar.
\end{proof}

\begin{theorem}\label{9}
For any  $m,n\in \mathbb{Z}_{\geq0}$, we have that

(1)\ when $m$ is odd, then

(i)\ \textcircled{1}\ For $\frac{{m - 1}}{2} \le n \le m - 2$,
\begin{eqnarray*}
{X_{ - m}}{X_{ - m + 2n + 3}} &=& {q^{\frac{{14n - 5m + 9}}{4}}}{y_1}{y_2}{({y_1}{y_2}{y_3})^{\frac{{2n - m + 1}}{2}}}{X_{ - m + n + 1}}{X_{ - m + n + 2}}\\
 &&+ \sum\limits_{l = 0}^{\left\lfloor {\frac{n}{2}} \right\rfloor } {(\sum\limits_{i = 0}^l {{q^{\frac{{2i + 3l}}{2}}}} )} {u_{n - 2l}}{({y_1}{y_2}{y_3})^l} + {y_2}\sum\limits_{l = 0}^{\left\lfloor {\frac{{n - 1}}{2}} \right\rfloor } {(\sum\limits_{i = 0}^l {{q^{\frac{{2i + 1 + 3l}}{2}}}} )} {u_{n - 1 - 2l}}{({y_1}{y_2}{y_3})^l};
\end{eqnarray*}

\textcircled{2}\ For $n=m - 1$,
\begin{eqnarray*}
{X_{ - m}}{X_{ - m + 2n + 3}} &=& {q^{\frac{{5m - 3}}{4}}}{y_2}{({y_1}{y_2}{y_3})^{\frac{{m - 1}}{2}}}{X_{ - m + n + 1}}{X_{ - m + n + 2}}\\
 &&+ \sum\limits_{l = 0}^{\left\lfloor {\frac{n}{2}} \right\rfloor } {(\sum\limits_{i = 0}^l {{q^{\frac{{2i + 3l}}{2}}}} )} {u_{n - 2l}}{({y_1}{y_2}{y_3})^l} + {y_2}\sum\limits_{l = 0}^{\left\lfloor {\frac{{n - 1}}{2}} \right\rfloor } {(\sum\limits_{i = 0}^l {{q^{\frac{{2i + 1 + 3l}}{2}}}} )} {u_{n - 1 - 2l}}{({y_1}{y_2}{y_3})^l};
\end{eqnarray*}

\textcircled{3}\ For $n \ge m$,
\begin{eqnarray*}
{X_{ - m}}{X_{ - m + 2n + 3}} &=& {q^{\frac{{m + 3}}{4}}}{y_2}{y_3}{({y_1}{y_2}{y_3})^{\frac{{m - 1}}{2}}}{X_{ - m + n + 1}}{X_{ - m + n + 2}}\\
 &&+ \sum\limits_{l = 0}^{\left\lfloor {\frac{n}{2}} \right\rfloor } {(\sum\limits_{i = 0}^l {{q^{\frac{{2i + 3l}}{2}}}} )} {u_{n - 2l}}{({y_1}{y_2}{y_3})^l} + {y_2}\sum\limits_{l = 0}^{\left\lfloor {\frac{{n - 1}}{2}} \right\rfloor } {(\sum\limits_{i = 0}^l {{q^{\frac{{2i + 1 + 3l}}{2}}}} )} {u_{n - 1 - 2l}}{({y_1}{y_2}{y_3})^l};
\end{eqnarray*}

(ii)\ \textcircled{1}\ For $\frac{{m + 1}}{2} \le n \le m - 1$,
\begin{eqnarray*}
{X_{ - m}}{X_{ - m + 2n}} &=& {q^{\frac{{14n - 5m - 11}}{4}}}{y_1}{({y_1}{y_2}{y_3})^{\frac{{2n - m - 1}}{2}}}{X_{ - m + 2\left\lfloor {\frac{n}{2}} \right\rfloor }}{X_{ - m +
2\left\lceil {\frac{n}{2}} \right\rceil }}\\
&&+ w\sum\limits_{l = 0}^{\left\lfloor {\frac{{n - 2}}{2}} \right\rfloor } {(\sum\limits_{i = 0}^l {{q^{\frac{{2i + 3l}}{2}}}} )} {u_{n - 2 - 2l}}{({y_1}{y_2}{y_3})^l};
\end{eqnarray*}

\textcircled{2}\ For $n=m$,
\begin{eqnarray*}
{X_{ - m}}{X_{ - m + 2n}} &=& {q^{\frac{{5m - 5}}{4}}}{({y_1}{y_2}{y_3})^{\frac{{m - 1}}{2}}}{X_{ - m + 2\left\lfloor {\frac{n}{2}} \right\rfloor }}{X_{ - m + 2\left\lceil
{\frac{n}{2}} \right\rceil }}\\
&&+ w\sum\limits_{l = 0}^{\left\lfloor {\frac{{n - 2}}{2}} \right\rfloor } {(\sum\limits_{i = 0}^l {{q^{\frac{{2i + 3l}}{2}}}} )} {u_{n - 2 - 2l}}{({y_1}{y_2}{y_3})^l};
\end{eqnarray*}

\textcircled{3}\ For $n \ge m+1$,
\begin{eqnarray*}
{X_{ - m}}{X_{ - m + 2n}} &=& {q^{\frac{{m + 3}}{4}}}{y_2}{y_3}{({y_1}{y_2}{y_3})^{\frac{{m - 1}}{2}}}{X_{ - m + 2\left\lfloor {\frac{n}{2}} \right\rfloor }}{X_{ - m +
2\left\lceil {\frac{n}{2}} \right\rceil }}\\
&&+ w\sum\limits_{l = 0}^{\left\lfloor {\frac{{n - 2}}{2}} \right\rfloor } {(\sum\limits_{i = 0}^l {{q^{\frac{{2i + 3l}}{2}}}} )} {u_{n - 2 - 2l}}{({y_1}{y_2}{y_3})^l};
\end{eqnarray*}

(2)\ when $m$ is even, then

(i)\ \textcircled{1}\ For $\frac{{m - 2}}{2} \le n \le m - 2$,
\begin{eqnarray*}
{X_{ - m}}{X_{ - m + 2n + 3}} &=& {q^{\frac{{14n - 5m + 12}}{4}}}{y_1}{({y_1}{y_2}{y_3})^{\frac{{2n - m + 2}}{2}}}{X_{ - m + n + 1}}{X_{ - m + n + 2}}\\
 &&+ \sum\limits_{l = 0}^{\left\lfloor {\frac{n}{2}} \right\rfloor } {(\sum\limits_{i = 0}^l {{q^{\frac{{2i + 3l}}{2}}}} )} {u_{n - 2l}}{({y_1}{y_2}{y_3})^l}+ {y_{\rm{1}}}{y_3}\sum\limits_{l = 0}^{\left\lfloor {\frac{{n - 1}}{2}} \right\rfloor } {(\sum\limits_{i = 0}^l {{q^{\frac{{2i + 2 + 3l}}{2}}}} )} {u_{n - 1 - 2l}}{({y_1}{y_2}{y_3})^l};
\end{eqnarray*}

\textcircled{2}\ For $n=m - 1$,
\begin{eqnarray*}
{X_{ - m}}{X_{ - m + 2n + 3}} &=& {q^{\frac{{5m}}{4}}}{({y_1}{y_2}{y_3})^{\frac{m}{2}}}{X_{ - m + n + 1}}{X_{ - m + n + 2}}\\
 &&+ \sum\limits_{l = 0}^{\left\lfloor {\frac{n}{2}} \right\rfloor } {(\sum\limits_{i = 0}^l {{q^{\frac{{2i + 3l}}{2}}}} )} {u_{n - 2l}}{({y_1}{y_2}{y_3})^l} + {y_{\rm{1}}}{y_3}\sum\limits_{l = 0}^{\left\lfloor {\frac{{n - 1}}{2}} \right\rfloor } {(\sum\limits_{i = 0}^l {{q^{\frac{{2i + 2 + 3l}}{2}}}} )} {u_{n - 1 - 2l}}{({y_1}{y_2}{y_3})^l};
\end{eqnarray*}

\textcircled{3}\ For $n \ge m$,
\begin{eqnarray*}
{X_{ - m}}{X_{ - m + 2n + 3}} &=& {q^{\frac{{m + {\rm{2}}}}{4}}}{y_3}{({y_1}{y_2}{y_3})^{\frac{m}{2}}}{X_{ - m + n + 1}}{X_{ - m + n + 2}}\\
 &&+ \sum\limits_{l = 0}^{\left\lfloor {\frac{n}{2}} \right\rfloor } {(\sum\limits_{i = 0}^l {{q^{\frac{{2i + 3l}}{2}}}} )} {u_{n - 2l}}{({y_1}{y_2}{y_3})^l} + {y_{\rm{1}}}{y_3}\sum\limits_{l = 0}^{\left\lfloor {\frac{{n - 1}}{2}} \right\rfloor } {(\sum\limits_{i = 0}^l {{q^{\frac{{2i + 2 + 3l}}{2}}}} )} {u_{n - 1 - 2l}}{({y_1}{y_2}{y_3})^l};
\end{eqnarray*}

(ii)\ \textcircled{1}\ For $\frac{{m + 2}}{2} \le n \le m $,
\begin{eqnarray*}
{X_{ - m}}{X_{ - m + 2n}} &=& {q^{\frac{{14n - 5m - 14}}{4}}}{y_1}{y_2}{({y_1}{y_2}{y_3})^{\frac{{2n - m - 2}}{2}}}{X_{ - m + 2\left\lfloor {\frac{n}{2}} \right\rfloor }}{X_{
- m + 2\left\lceil {\frac{n}{2}} \right\rceil }}\\
&&+ z\sum\limits_{l = 0}^{\left\lfloor {\frac{{n - 2}}{2}} \right\rfloor } {(\sum\limits_{i = 0}^l {{q^{\frac{{2i + 3l}}{2}}}} )} {u_{n - 2 - 2l}}{({y_1}{y_2}{y_3})^l};
\end{eqnarray*}

\textcircled{2}\ For $n=m+1$,
\begin{eqnarray*}
{X_{ - m}}{X_{ - m + 2n}} &=& {q^{\frac{{5m}}{4}}}{({y_1}{y_2}{y_3})^{\frac{m}{2}}}{X_{ - m + 2\left\lfloor {\frac{n}{2}} \right\rfloor }}{X_{ - m + 2\left\lceil {\frac{n}{2}}
\right\rceil }}\\
&&+ z\sum\limits_{l = 0}^{\left\lfloor {\frac{{n - 2}}{2}} \right\rfloor } {(\sum\limits_{i = 0}^l {{q^{\frac{{2i + 3l}}{2}}}} )} {u_{n - 2 - 2l}}{({y_1}{y_2}{y_3})^l};
\end{eqnarray*}

\textcircled{3}\ For $n \ge m+2$,
\begin{eqnarray*}
{X_{ - m}}{X_{ - m + 2n}} &=& {q^{\frac{{m + 2}}{4}}}{y_3}{({y_1}{y_2}{y_3})^{\frac{m}{2}}}{X_{ - m + 2\left\lfloor {\frac{n}{2}} \right\rfloor }}{X_{ - m + 2\left\lceil
{\frac{n}{2}} \right\rceil }}\\
&&+ z\sum\limits_{l = 0}^{\left\lfloor {\frac{{n - 2}}{2}} \right\rfloor } {(\sum\limits_{i = 0}^l {{q^{\frac{{2i + 3l}}{2}}}} )} {u_{n - 2 - 2l}}{({y_1}{y_2}{y_3})^l}.
\end{eqnarray*}
\end{theorem}
\begin{proof}
We only prove (1), and the proof of (2) is similar.

(1)\ (i)\ \textcircled{1}\ Note  that $m\geq 3$ in this case. When $m=3$, by Lemma \ref{8}, the case holds. When $m=5$, the case holds by Lemma \ref{8} and the following calculation
\begin{eqnarray*}
&&{X_{ - 5}}{X_{\rm{4}}}\\
 &=& {X_{ - 5}}({X_2}{u_1} - {q^{ - 1}}{y_2}{y_1}{X_0})\\
 &=& ({q^{\frac{1}{2}}}{y_2}{u_1} + {u_2} + ({q^{\frac{3}{2}}} + {q^{\frac{5}{2}}}){y_1}{y_2}{y_3} + q{X_{ - 2}}{X_{ - 1}}{y_1}{y_2}){u_1}\\
 &&- ({q^{ - \frac{3}{2}}}{y_3}{u_1} + {y_2}{y_3} + {X_{ - 3}}{X_{ - 2}}){y_1}{y_2}\\
 &=& {u_3} + {q^{\frac{3}{2}}}{y_1}{y_2}{y_3}{u_1} + {q^{\frac{1}{2}}}{y_2}{u_2} + 2{q^{\frac{1}{2}}}{y_2}{q^{\frac{3}{2}}}{y_1}{y_2}{y_3} + ({q^{\frac{3}{2}}} + {q^{\frac{5}{2}}}){y_1}{y_2}{y_3}{u_1}\\
 &&+ {X_{ - 2}}({y_3}{y_2}{X_1} + {X_{ - 3}}){y_1}{y_2} - {q^{ - \frac{3}{2}}}{y_3}{u_1}{y_1}{y_2} - {y_2}{y_3}{y_1}{y_2} - {X_{ - 3}}{X_{ - 2}}{y_1}{y_2}\\
 &=& {u_3} + {q^{\frac{1}{2}}}{y_2}{u_2} + ({q^{\frac{3}{2}}} + {q^{\frac{5}{2}}}){y_1}{y_2}{y_3}{u_1} + ({q^2} + {q^3}){y_2}{y_1}{y_2}{y_3} + {q^{\frac{3}{2}}}{y_1}{y_2}{y_1}{y_2}{y_3}{X_0}{X_{ - 1}}.
\end{eqnarray*}

Now we consider the case $m \geq 7$. When $n=\frac{{m - 1}}{2}$, by Lemma \ref{8}, the case holds. When $n=\frac{{m - 1}}{2} + 1$, by Lemma \ref{1}, \ref{2}, \ref{4} and Theorem \ref{7} we have that

(a)\ if $4 \nmid m-1$, then
\begin{eqnarray*}
&&{X_{ - m}}{X_4}\\
 &=& {X_{ - m}}({X_2}{u_1} - {q^{ - 1}}{y_2}{y_1}{X_0})\\
 &=& {q^{\frac{{m + 1}}{2}}}{y_1}{y_2}{X_{\frac{{1 - m}}{2}}}{X_{\frac{{3 - m}}{2}}}{u_1} + ({u_{\frac{{m - 1}}{2}}} + ({q^{\frac{3}{2}}} + {q^{\frac{5}{2}}}){u_{\frac{{m - 5}}{2}}}{y_1}{y_2}{y_3} +  \cdots \\
 &&+ ({q^{\frac{{3m - 9}}{8}}} + {q^{\frac{{3m - 1}}{8}}} +  \cdots  + {q^{\frac{{5m - 15}}{8}}})u_1^{}{({y_1}{y_2}{y_3})^{\frac{{m - 3}}{4}}}){u_1} + {y_2}({q^{\frac{1}{2}}}{u_{\frac{{m - 3}}{2}}}\\
 &&+ ({q^2} + {q^3}){u_{\frac{{m - 7}}{2}}}{y_1}{y_2}{y_3} +  \cdots  + ({q^{\frac{{3m - 5}}{8}}} + {q^{\frac{{3m + 3}}{8}}} +  \cdots  + {q^{\frac{{5m - 11}}{8}}})\\
 &&\cdot {u_0}{({y_1}{y_2}{y_3})^{\frac{{m - 3}}{4}}}){u_1} - {q^{ - 1}}{X_{\frac{{1 - m}}{2}}}{X_{\frac{{ - 1 - m}}{2}}}{y_2}{y_1} - {q^{ - \frac{1}{2}}}({q^{ - 1}}{u_{\frac{{m - 3}}{2}}}\\
 &&+({q^{ - \frac{3}{2}}} + {q^{ - \frac{5}{2}}}){y_3}{y_2}{y_1}{u_{\frac{{m - 7}}{2}}} +  \cdots  + ({q^{ - \frac{{m + 5}}{8}}} + {q^{ - \frac{{m + 13}}{8}}} +  \cdots  + {q^{ - \frac{{3m - 1}}{8}}})\\
 &&\cdot {({y_3}{y_2}{y_1})^{\frac{{m - 3}}{4}}}{u_0}){y_3}{y_2}{y_1} - {q^{ - \frac{1}{2}}}({q^{ - \frac{3}{2}}}{u_{\frac{{m - 5}}{2}}} + ({q^{ - 2}} + {q^{ - 3}}){y_3}{y_2}{y_1}{u_{\frac{{m - 9}}{2}}}\\
 &&+  \cdots  + ({q^{ - \frac{{m + 5}}{8}}} + {q^{ - \frac{{m + 13}}{8}}} +  \cdots  + {q^{ - \frac{{3m - 9}}{8}}}){({y_3}{y_2}{y_1})^{\frac{{m - 7}}{4}}}{u_1}){y_3}{y_2}{y_2}{y_1}\\
 &=& {q^{\frac{{m + 1}}{2}}}{y_1}{y_2}{X_{\frac{{1 - m}}{2}}}{X_{\frac{{3 - m}}{2}}}{u_1} - {X_{\frac{{1 - m}}{2}}}{X_{\frac{{ - 1 - m}}{2}}}{y_1}{y_2} + ({u_{\frac{{m + 1}}{2}}} + ({q^{\frac{3}{2}}} + {q^{\frac{5}{2}}}){u_{\frac{{m - 3}}{2}}}{y_1}{y_2}{y_3}\\
 &&+  \cdots  + ({q^{\frac{{3m + 3}}{8}}} + {q^{\frac{{3m + 11}}{8}}} +  \cdots  + {q^{\frac{{5m - 3}}{8}}}){({y_1}{y_2}{y_3})^{\frac{{m + 1}}{4}}}) + {y_2}({q^{\frac{1}{2}}}{u_{\frac{{m - 1}}{2}}} + ({q^2} + {q^3})\\
 &&\cdot {u_{\frac{{m - 5}}{2}}}{y_1}{y_2}{y_3} +  \cdots  + ({q^{\frac{{3m - 5}}{8}}} + {q^{\frac{{3m + 3}}{8}}} +  \cdots  + {q^{\frac{{5m - 11}}{8}}}){u_1}{({y_1}{y_2}{y_3})^{\frac{{m - 3}}{4}}}),
\end{eqnarray*}
thus, the proof follows that
\begin{eqnarray*}
&&{q^{\frac{{m + 1}}{2}}}{y_1}{y_2}{X_{\frac{{1 - m}}{2}}}{X_{\frac{{3 - m}}{2}}}{u_1} - {X_{\frac{{1 - m}}{2}}}{X_{\frac{{ - 1 - m}}{2}}}{y_1}{y_2}\\
 &=& {q^{\frac{{m + 1}}{2}}}{y_1}{y_2}{X_{\frac{{1 - m}}{2}}}({q^{ - \frac{1}{2}}}{y_3}{y_2}{y_1}{X_{\frac{{7 - m}}{2}}} + {X_{\frac{{ - 1 - m}}{2}}}) - {X_{\frac{{1 - m}}{2}}}{X_{\frac{{ - 1 - m}}{2}}}{y_1}{y_2}\\
 &=& {q^{\frac{{5m + 5}}{8}}}{({y_1}{y_2}{y_3})^{\frac{{m + 1}}{4}}} + {q^{\frac{{m + 8}}{2}}}{y_1}{y_2}({y_1}{y_2}{y_3}){X_{\frac{{5 - m}}{2}}}{X_{\frac{{3 - m}}{2}}};
\end{eqnarray*}

(b)\ if $4 \mid m-1$, the proof is similar.

Suppose that the case holds for $\frac{{m - 1}}{2}+1 \leq n \leq k < m-2 $. When $n=k+1$, by Lemma \ref{1}, \ref{2} and \ref{4} we have that

(a)\ if $k$ is even, then
\begin{eqnarray*}
&&{X_{ - m}}{X_{ - m + 2k + 5}}\\
 &=& {X_{ - m}}({X_{ - m + 2k + 3}}{u_1} - {q^{ - \frac{5}{2}}}{y_3}{y_2}{y_1}{X_{ - m + 2k + 1}})\\
 &=& {q^{\frac{{14k - 5m + 9}}{4}}}{y_1}{y_2}{({y_1}{y_2}{y_3})^{\frac{{2k - m + 1}}{2}}}{X_{ - m + k + 1}}{X_{ - m + k + 2}}{u_1} + ({u_k} + ({q^{\frac{3}{2}}} + {q^{\frac{5}{2}}}){u_{k - 2}}{y_1}{y_2}{y_3}\\
 &&+  \cdots  + ({q^{\frac{{3k}}{4}}} + {q^{\frac{{3k + 4}}{4}}} +  \cdots  + {q^{\frac{{5k}}{4}}}){({y_1}{y_2}{y_3})^{\frac{k}{2}}}){u_1} + {y_2}({q^{\frac{1}{2}}}{u_{k - 1}} + ({q^2} + {q^3}){u_{k - 3}}{y_1}{y_2}{y_3}\\
 &&+  \cdots  + ({q^{\frac{{3k - 4}}{4}}} + {q^{\frac{{3k}}{4}}} +  \cdots  + {q^{\frac{{5k - 8}}{4}}}){u_1}{({y_1}{y_2}{y_3})^{\frac{{k - 2}}{2}}}){u_1} - {q^{\frac{3}{2}}}{q^{\frac{{14k - 5m - 5}}{4}}}{y_1}{y_2}\\
 &&\cdot {({y_1}{y_2}{y_3})^{\frac{{2k - m - 1}}{2}}}{X_{ - m + k}}{X_{ - m + k + 1}}{y_1}{y_2}{y_3} - {q^{\frac{3}{2}}}({u_{k - 1}} + ({q^{\frac{3}{2}}} + {q^{\frac{5}{2}}}){u_{k - 3}}{y_1}{y_2}{y_3} +  \cdots \\
 &&+ ({q^{\frac{{3k - 6}}{4}}} + {q^{\frac{{3k - 2}}{4}}} +  \cdots  + {q^{\frac{{5k - 10}}{4}}}){u_1}{({y_1}{y_2}{y_3})^{\frac{{k - 2}}{2}}}){y_1}{y_2}{y_3} - {q^{\frac{3}{2}}}{y_2}({q^{\frac{1}{2}}}{u_{k - 2}}\\
 &&+ ({q^2} + {q^3}){u_{k - 4}}{y_1}{y_2}{y_3} +  \cdots  + ({q^{\frac{{3k - 4}}{4}}} + {q^{\frac{{3k}}{4}}} +  \cdots  + {q^{\frac{{5k - 8}}{4}}}){({y_1}{y_2}{y_3})^{\frac{{k - 2}}{2}}}){y_1}{y_2}{y_3}\\
 &=& {q^{\frac{{14k - 5m + 9}}{4}}}{y_1}{y_2}{({y_1}{y_2}{y_3})^{\frac{{2k - m + 1}}{2}}}{X_{ - m + k + 1}}{X_{ - m + k + 2}}{u_1} - {q^{\frac{3}{2}}}{q^{\frac{{14k - 5m - 5}}{4}}}{y_1}{y_2}\\
 &&\cdot {({y_1}{y_2}{y_3})^{\frac{{2k - m - 1}}{2}}}{X_{ - m + k}}{X_{ - m + k + 1}}{y_1}{y_2}{y_3} + ({u_{k + 1}} + ({q^{\frac{3}{2}}} + {q^{\frac{5}{2}}}){u_{k - 1}}{y_1}{y_2}{y_3} +  \cdots \\
 &&+ ({q^{\frac{{3k}}{4}}} + {q^{\frac{{3k + 4}}{4}}} +  \cdots  + {q^{\frac{{5k}}{4}}}){u_1}{({y_1}{y_2}{y_3})^{\frac{k}{2}}}) + {y_2}({q^{\frac{1}{2}}}{u_k} + ({q^2} + {q^3}){u_{k - 2}}{y_1}{y_2}{y_3}\\
 &&+  \cdots  + ({q^{\frac{{3k + 2}}{4}}} + {q^{\frac{{3k + 6}}{4}}} +  \cdots  + {q^{\frac{{5k - 2}}{4}}}){({y_1}{y_2}{y_3})^{\frac{k}{2}}}),
\end{eqnarray*}
thus, the proof follows that
\begin{eqnarray*}
&&{q^{\frac{{14k - 5m + 9}}{4}}}{y_1}{y_2}{({y_1}{y_2}{y_3})^{\frac{{2k - m + 1}}{2}}}{X_{ - m + k + 1}}{X_{ - m + k + 2}}{u_1}\\
 &&- {q^{\frac{3}{2}}}{q^{\frac{{14k - 5m - 5}}{4}}}{y_1}{y_2}{({y_1}{y_2}{y_3})^{\frac{{2k - m - 1}}{2}}}{X_{ - m + k}}{X_{ - m + k + 1}}{y_1}{y_2}{y_3}\\
 &=& {q^{\frac{{14k - 5m + 9}}{4}}}{y_1}{y_2}{({y_1}{y_2}{y_3})^{\frac{{2k - m + 1}}{2}}}{X_{ - m + k + 2}}({q^{ - \frac{1}{2}}}{y_3}{y_2}{y_1}{X_{- m +k + 3}} + {X_{- m+k - 1}})\\
 &&- {q^{\frac{3}{2}}}{q^{\frac{{14k - 5m - 5}}{4}}}{y_1}{y_2}{({y_1}{y_2}{y_3})^{\frac{{2k - m - 1}}{2}}}{X_{ - m + k}}{X_{ - m + k + 1}}{y_1}{y_2}{y_3}\\
 &=& {q^{\frac{{14k - 5m + 23}}{4}}}{y_1}{y_2}{({y_1}{y_2}{y_3})^{\frac{{2k - m + 3}}{2}}}{X_{ - m + k + 2}}{X_{ - m + k + 3}} + {q^{\frac{{5k + 2}}{4}}}{y_2}{({y_1}{y_2}{y_3})^{\frac{k}{2}}};
\end{eqnarray*}

(b)\ if $k$ is odd, the proof is similar.

\textcircled{2}\ When $m = 1$, by Lemma \ref{1} the case holds. When $m = 3$, by Lemma \ref{1}, \ref{4} and \ref{8} we have
\begin{eqnarray*}
{X_{ - 3}}{X_{4}}
 &=& {X_{ - 3}}({X_2}{u_1} - {q^{ - 1}}{y_2}{y_1}{X_0})\\
 &=& ({q^{\frac{1}{2}}}{y_2} + {u_1} + {q^2}{y_1}{y_2}{X_0}{X_{ - 1}}){u_1}- ({q^{ - \frac{1}{2}}}{y_3} + {X_{ - 2}}{X_{ - 1}}){y_1}{y_2}\\
 &=& {u_2} + 2{q^{\frac{3}{2}}}{y_1}{y_2}{y_3} + {q^{\frac{1}{2}}}{y_2}{u_1} + {q^2}{y_1}{y_2}{X_{ - 1}}({q^{\frac{1}{2}}}{y_3}{X_2} + {X_{ - 2}})\\
 &&- {q^{ - \frac{1}{2}}}{y_3}{y_1}{y_2} - {X_{ - 2}}{X_{ - 1}}{y_1}{y_2}\\
 &=& {u_2} + {q^{\frac{1}{2}}}{y_2}{u_1} + ({q^{\frac{3}{2}}} + {q^{\frac{5}{2}}}){y_1}{y_2}{y_3} + {q^4}{y_1}{y_2}{y_2}{y_3}{X_0}{X_1}.
\end{eqnarray*}

When $m \geq 5$, by Lemma \ref{1}, \ref{2}, \ref{4} and \textcircled{1} of (i) we have
\begin{eqnarray*}
&&{X_{ - m}}{X_{m + {\rm{1}}}}\\
 &=& {X_{ - m}}({X_{m - 1}}{u_1} - {q^{ - \frac{5}{2}}}{y_3}{y_2}{y_1}{X_{m - 3}})\\
 \end{eqnarray*}

\begin{eqnarray*}
 &=& {q^{\frac{{9m - 19}}{4}}}{y_1}{y_2}{({y_1}{y_2}{y_3})^{\frac{{m - 3}}{2}}}{X_{ - 1}}{X_0}{u_1} + ({u_{m - 2}} + ({q^{\frac{3}{2}}} + {q^{\frac{5}{2}}}){u_{m - 4}}{y_1}{y_2}{y_3}\\
 &&+  \cdots  + ({q^{\frac{{3m - 9}}{4}}} + {q^{\frac{{3m - 5}}{4}}} +  \cdots  + {q^{\frac{{5m - 15}}{4}}}){u_1}{({y_1}{y_2}{y_3})^{\frac{{m - 3}}{2}}}){u_1} + {y_2}({q^{\frac{1}{2}}}{u_{m - 3}}\\
 &&+ ({q^2} + {q^3}){u_{m - 5}}{y_1}{y_2}{y_3} +  \cdots  + ({q^{\frac{{3m - 7}}{4}}} + {q^{\frac{{3m - 3}}{4}}} +  \cdots  + {q^{\frac{{5m - 13}}{4}}}){u_0}{({y_1}{y_2}{y_3})^{\frac{{m - 3}}{2}}}){u_1}\\
 &&- {q^{\frac{3}{2}}}{q^{\frac{{9m - 33}}{4}}}{y_1}{y_2}{({y_1}{y_2}{y_3})^{\frac{{m - 5}}{2}}}{X_{ - 2}}{X_{ - 1}}{y_1}{y_2}{y_3} - {q^{\frac{3}{2}}}({u_{m - 3}} + ({q^{\frac{3}{2}}} + {q^{\frac{5}{2}}}){u_{m - 5}}{y_1}{y_2}{y_3}\\
 &&+  \cdots  + ({q^{\frac{{3m - 9}}{4}}} + {q^{\frac{{3m - 5}}{4}}} +  \cdots  + {q^{\frac{{5m - 15}}{4}}}){u_0}{({y_1}{y_2}{y_3})^{\frac{{m - 3}}{2}}}){y_1}{y_2}{y_3} - {q^{\frac{3}{2}}}{y_2}({q^{\frac{1}{2}}}{u_{m - 4}}\\
 &&+ ({q^2} + {q^3}){u_{m - 6}}{y_1}{y_2}{y_3} +  \cdots  + ({q^{\frac{{3m - 13}}{4}}} + {q^{\frac{{3m - 9}}{4}}} +  \cdots  + {q^{\frac{{5m - 23}}{4}}}){u_1}{({y_1}{y_2}{y_3})^{\frac{{m - 5}}{2}}}){y_1}{y_2}{y_3}\\
 &=& {q^{\frac{{9m - 19}}{4}}}{y_1}{y_2}{({y_1}{y_2}{y_3})^{\frac{{m - 3}}{2}}}{X_{ - 1}}{X_0}{u_1} - {q^{\frac{3}{2}}}{q^{\frac{{9m - 33}}{4}}}{y_1}{y_2}{({y_1}{y_2}{y_3})^{\frac{{m - 5}}{2}}}{X_{ - 2}}{X_{ - 1}}{y_1}{y_2}{y_3}\\
 &&+ ({u_{m - 1}} + ({q^{\frac{3}{2}}} + {q^{\frac{5}{2}}}){u_{m - 3}}{y_1}{y_2}{y_3} +  \cdots  + ({q^{\frac{{3m - 3}}{4}}} + {q^{\frac{{3m + 1}}{4}}} +  \cdots  + {q^{\frac{{5m - 9}}{4}}}){u_0}{({y_1}{y_2}{y_3})^{\frac{{m - 1}}{2}}})\\
 &&+ {y_2}({q^{\frac{1}{2}}}{u_{m - 2}} + ({q^2} + {q^3}){u_{m - 4}}{y_1}{y_2}{y_3} +  \cdots  + ({q^{\frac{{3m - 7}}{4}}} + {q^{\frac{{3m - 3}}{4}}} +  \cdots  + {q^{\frac{{5m - 13}}{4}}}){u_1}{({y_1}{y_2}{y_3})^{\frac{{m - 3}}{2}}}),
\end{eqnarray*}
thus, the proof follows that
\begin{eqnarray*}
&&{q^{\frac{{9m - 19}}{4}}}{y_1}{y_2}{({y_1}{y_2}{y_3})^{\frac{{m - 3}}{2}}}{X_{ - 1}}{X_0}{u_1} - {q^{\frac{3}{2}}}{q^{\frac{{9m - 33}}{4}}}{y_1}{y_2}{({y_1}{y_2}{y_3})^{\frac{{m - 5}}{2}}}{X_{ - 2}}{X_{ - 1}}{y_1}{y_2}{y_3}\\
 &=& {q^{\frac{{9m - 19}}{4}}}{y_1}{y_2}{({y_1}{y_2}{y_3})^{\frac{{m - 3}}{2}}}{X_{ - 1}}({q^{\frac{1}{2}}}{y_3}{X_2} + {X_{ - 2}})\\
 &&- {q^{\frac{3}{2}}}{q^{\frac{{9m - 33}}{4}}}{y_1}{y_2}{({y_1}{y_2}{y_3})^{\frac{{m - 5}}{2}}}{X_{ - 2}}{X_{ - 1}}{y_1}{y_2}{y_3}\\
 &=& {q^{\frac{{5m - 3}}{4}}}{y_2}{({y_1}{y_2}{y_3})^{\frac{{m - 1}}{2}}}{X_0}{X_1} + {q^{\frac{{5m - 5}}{4}}}{({y_1}{y_2}{y_3})^{\frac{{m - 1}}{2}}}.
\end{eqnarray*}

\textcircled{3}\ Firstly, we consider the case for $n=m$: when $m = 1$, by Lemma \ref{1} and \ref{4} we have
\begin{eqnarray*}
{X_{ - 1}}{X_{\rm{4}}}
 &=& {X_{ - 1}}({X_2}{u_1} - {q^{ - 1}}{y_2}{y_1}{X_0})\\
 &=& ({q^{\frac{1}{2}}}{X_1}{X_0}{y_2} + 1){u_1} - {y_2}{y_1}{X_0}{X_{ - 1}}\\
 &=& {q^{\frac{1}{2}}}{y_2}{X_0}({q^{ - \frac{1}{2}}}{y_1}{X_{ - 1}} + {X_3}) + {u_1} - {y_2}{y_1}{X_0}{X_{ - 1}}\\
 &=& q{y_2}{y_3}{X_1}{X_2} + {u_1} + {q^{\frac{1}{2}}}{y_2};
\end{eqnarray*}

when $m \geq 3$, by Lemma \ref{1}, \ref{2}, \ref{4}, \textcircled{1}  and \textcircled{2} of (i) we have
\begin{eqnarray*}
&&{X_{ - m}}{X_{m + 3}}\\
 &=& {X_{ - m}}({X_{m + 1}}{u_1} - {q^{ - \frac{5}{2}}}{y_3}{y_2}{y_1}{X_{m - 1}})\\
 &=& {q^{\frac{{5m - 3}}{4}}}{y_2}{({y_1}{y_2}{y_3})^{\frac{{m - 1}}{2}}}{X_0}{X_1}{u_1} + ({u_{m - 1}} + ({q^{\frac{3}{2}}} + {q^{\frac{5}{2}}}){u_{m - 3}}{y_1}{y_2}{y_3} +  \cdots \\
 &&+ ({q^{\frac{{3m - 3}}{4}}} + {q^{\frac{{3m + 1}}{4}}} +  \cdots  + {q^{\frac{{5m - 5}}{4}}}){u_0}{({y_1}{y_2}{y_3})^{\frac{{m - 1}}{2}}}){u_1} + {y_2}({q^{\frac{1}{2}}}{u_{m - 2}} + ({q^2} + {q^3})\\
 &&\cdot {u_{m - 4}}{y_1}{y_2}{y_3} +  \cdots  + ({q^{\frac{{3m - 7}}{4}}} + {q^{\frac{{3m - 3}}{4}}} +  \cdots  + {q^{\frac{{5m - 13}}{4}}}){u_1}{({y_1}{y_2}{y_3})^{\frac{{m - 3}}{2}}}){u_1}\\
 &&- {q^{\frac{3}{2}}}{q^{\frac{{9m - 19}}{4}}}{y_1}{y_2}{({y_1}{y_2}{y_3})^{\frac{{m - 3}}{2}}}{X_{ - 1}}{X_0}{y_1}{y_2}{y_3} - {q^{\frac{3}{2}}}({u_{m - 2}} + ({q^{\frac{3}{2}}} + {q^{\frac{5}{2}}}){u_{m - 4}}{y_1}{y_2}{y_3}\\
 &&+  \cdots  + ({q^{\frac{{3m - 9}}{4}}} + {q^{\frac{{3m - 5}}{4}}} +  \cdots  + {q^{\frac{{5m - 15}}{4}}}){u_1}{({y_1}{y_2}{y_3})^{\frac{{m - 3}}{2}}}){y_1}{y_2}{y_3} - {q^{\frac{3}{2}}}{y_2}({q^{\frac{1}{2}}}{u_{m - 3}}\\
 &&+ ({q^2} + {q^3}){u_{m - 5}}{y_1}{y_2}{y_3} +  \cdots  + ({q^{\frac{{3m - 7}}{4}}} + {q^{\frac{{3m - 3}}{4}}} +  \cdots  + {q^{\frac{{5m - 13}}{4}}}){u_0}{({y_1}{y_2}{y_3})^{\frac{{m - 3}}{2}}}){y_1}{y_2}{y_3}\\
 &=& {q^{\frac{{5m - 3}}{4}}}{y_2}{({y_1}{y_2}{y_3})^{\frac{{m - 1}}{2}}}{X_0}{X_1}{u_1} - {q^{\frac{3}{2}}}{q^{\frac{{9m - 19}}{4}}}{y_1}{y_2}{({y_1}{y_2}{y_3})^{\frac{{m - 3}}{2}}}{X_{ - 1}}{X_0}{y_1}{y_2}{y_3}\\
 &&+ ({u_m} + ({q^{\frac{3}{2}}} + {q^{\frac{5}{2}}}){u_{m - 2}}{y_1}{y_2}{y_3} +  \cdots  + ({q^{\frac{{3m - 3}}{4}}} + {q^{\frac{{3m + 1}}{4}}} +  \cdots  + {q^{\frac{{5m - 5}}{4}}}){u_1}{({y_1}{y_2}{y_3})^{\frac{{m - 1}}{2}}})\\
 &&+ {y_2}({q^{\frac{1}{2}}}{u_{m - 1}} + ({q^2} + {q^3}){u_{m - 3}}{y_1}{y_2}{y_3} +  \cdots  + ({q^{\frac{{3m - 1}}{4}}} + {q^{\frac{{3m + 3}}{4}}} +  \cdots  + {q^{\frac{{5m - 7}}{4}}}){({y_1}{y_2}{y_3})^{\frac{{m - 1}}{2}}}),
\end{eqnarray*}
thus, the proof follows that
\begin{eqnarray*}
&&{q^{\frac{{5m - 3}}{4}}}{y_2}{({y_1}{y_2}{y_3})^{\frac{{m - 1}}{2}}}{X_0}{X_1}{u_1} - {q^{\frac{3}{2}}}{q^{\frac{{9m - 19}}{4}}}{y_1}{y_2}{({y_1}{y_2}{y_3})^{\frac{{m - 3}}{2}}}{X_{ - 1}}{X_0}{y_1}{y_2}{y_3}\\
 &=& {q^{\frac{{5m - 3}}{4}}}{y_2}{({y_1}{y_2}{y_3})^{\frac{{m - 1}}{2}}}{X_0}({q^{ - \frac{1}{2}}}{y_1}{X_{ - 1}} + {X_3})\\
 &&- {q^{\frac{3}{2}}}{q^{\frac{{9m - 19}}{4}}}{y_1}{y_2}{({y_1}{y_2}{y_3})^{\frac{{m - 3}}{2}}}{X_{ - 1}}{X_0}{y_1}{y_2}{y_3}\\
 &=& {q^{\frac{{5m - 3}}{4}}}{y_2}{({y_1}{y_2}{y_3})^{\frac{{m - 1}}{2}}} + {q^{\frac{{m + 3}}{4}}}{y_2}{y_3}{({y_1}{y_2}{y_3})^{\frac{{m - 1}}{2}}}{X_1}{X_2}.
\end{eqnarray*}

Secondly, we consider the case for $n=m+1$: by Lemma \ref{1}, \ref{2} and \ref{4} we have
\begin{eqnarray*}
&&{X_{ - m}}{X_{m + {\rm{5}}}}\\
 &=& {X_{ - m}}({X_{m + {\rm{3}}}}{u_1} - {q^{ - \frac{5}{2}}}{y_3}{y_2}{y_1}{X_{m{\rm{ + }}1}})\\
 &=& {q^{\frac{{m{\rm{ + }}3}}{4}}}{y_2}{y_3}{({y_1}{y_2}{y_3})^{\frac{{m - 1}}{2}}}{X_{\rm{1}}}{X_{\rm{2}}}{u_1} + ({u_m} + ({q^{\frac{3}{2}}} + {q^{\frac{5}{2}}}){u_{m - {\rm{2}}}}{y_1}{y_2}{y_3} +  \cdots \\
 &&+ ({q^{\frac{{3m - 3}}{4}}} + {q^{\frac{{3m + 1}}{4}}} +  \cdots  + {q^{\frac{{5m - 5}}{4}}}){u_{\rm{1}}}{({y_1}{y_2}{y_3})^{\frac{{m - 1}}{2}}}){u_1} + {y_2}({q^{\frac{1}{2}}}{u_{m - {\rm{1}}}} + ({q^2} + {q^3})\\
 &&\cdot {u_{m - {\rm{3}}}}{y_1}{y_2}{y_3} +  \cdots  + ({q^{\frac{{3m - {\rm{1}}}}{4}}} + {q^{\frac{{3m{\rm{ + }}3}}{4}}} +  \cdots  + {q^{\frac{{5m - 3}}{4}}}){u_{\rm{0}}}{({y_1}{y_2}{y_3})^{\frac{{m - {\rm{1}}}}{2}}}){u_1}\\
 &&- {q^{\frac{3}{2}}}{q^{\frac{{{\rm{5}}m - {\rm{3}}}}{4}}}{y_2}{({y_1}{y_2}{y_3})^{\frac{{m - {\rm{1}}}}{2}}}{X_{\rm{0}}}{X_{\rm{1}}}{y_1}{y_2}{y_3} - {q^{\frac{3}{2}}}({u_{m - {\rm{1}}}} + ({q^{\frac{3}{2}}} + {q^{\frac{5}{2}}}){u_{m - {\rm{3}}}}{y_1}{y_2}{y_3}\\
 &&+  \cdots  + ({q^{\frac{{3m - {\rm{3}}}}{4}}} + {q^{\frac{{3m{\rm{ + 1}}}}{4}}} +  \cdots  + {q^{\frac{{5m - 5}}{4}}}){u_{\rm{0}}}{({y_1}{y_2}{y_3})^{\frac{{m - {\rm{1}}}}{2}}}){y_1}{y_2}{y_3} - {q^{\frac{3}{2}}}{y_2}({q^{\frac{1}{2}}}{u_{m - {\rm{2}}}}\\
 &&+ ({q^2} + {q^3}){u_{m - {\rm{4}}}}{y_1}{y_2}{y_3} +  \cdots  + ({q^{\frac{{3m - 7}}{4}}} + {q^{\frac{{3m - 3}}{4}}} +  \cdots  + {q^{\frac{{5m - 13}}{4}}}){u_{\rm{1}}}{({y_1}{y_2}{y_3})^{\frac{{m - 3}}{2}}}){y_1}{y_2}{y_3}\\
 &=& {q^{\frac{{m{\rm{ + }}3}}{4}}}{y_2}{y_3}{({y_1}{y_2}{y_3})^{\frac{{m - 1}}{2}}}{X_{\rm{1}}}{X_{\rm{2}}}{u_1} - {q^{\frac{3}{2}}}{q^{\frac{{{\rm{5}}m - {\rm{3}}}}{4}}}{y_2}{({y_1}{y_2}{y_3})^{\frac{{m - {\rm{1}}}}{2}}}{X_{\rm{0}}}{X_{\rm{1}}}{y_1}{y_2}{y_3}\\
 &&+ ({u_{m{\rm{ + 1}}}} + ({q^{\frac{3}{2}}} + {q^{\frac{5}{2}}}){u_{m - {\rm{1}}}}{y_1}{y_2}{y_3} +  \cdots  + ({q^{\frac{{3m{\rm{ + }}3}}{4}}} + {q^{\frac{{3m + {\rm{7}}}}{4}}} +  \cdots  + {q^{\frac{{5m{\rm{ + 1}}}}{4}}}){u_{\rm{0}}}{({y_1}{y_2}{y_3})^{\frac{{m{\rm{ + }}1}}{2}}})\\
 &&+ {y_2}({q^{\frac{1}{2}}}{u_m} + ({q^2} + {q^3}){u_{m - {\rm{2}}}}{y_1}{y_2}{y_3} +  \cdots  + ({q^{\frac{{3m - 1}}{4}}} + {q^{\frac{{3m + 3}}{4}}} +  \cdots  + {q^{\frac{{5m - {\rm{3}}}}{4}}}){u_{\rm{1}}}{({y_1}{y_2}{y_3})^{\frac{{m - 1}}{2}}}),
\end{eqnarray*}
thus, the proof follows that
\begin{eqnarray*}
&&{q^{\frac{{m{\rm{ + }}3}}{4}}}{y_2}{y_3}{({y_1}{y_2}{y_3})^{\frac{{m - 1}}{2}}}{X_{\rm{1}}}{X_{\rm{2}}}{u_1} - {q^{\frac{3}{2}}}{q^{\frac{{{\rm{5}}m - {\rm{3}}}}{4}}}{y_2}{({y_1}{y_2}{y_3})^{\frac{{m - {\rm{1}}}}{2}}}{X_{\rm{0}}}{X_{\rm{1}}}{y_1}{y_2}{y_3}\\
 &=& {q^{\frac{{m{\rm{ + }}3}}{4}}}{y_2}{y_3}{({y_1}{y_2}{y_3})^{\frac{{m - 1}}{2}}}{X_{\rm{1}}}({q^{ - {\rm{1}}}}{y_2}{y_1}{X_{\rm{0}}} + {X_{\rm{4}}})\\
 &&- {q^{\frac{3}{2}}}{q^{\frac{{{\rm{5}}m - {\rm{3}}}}{4}}}{y_2}{({y_1}{y_2}{y_3})^{\frac{{m - {\rm{1}}}}{2}}}{X_{\rm{0}}}{X_{\rm{1}}}{y_1}{y_2}{y_3}\\
 &=& {q^{\frac{{5m + 5}}{4}}}{({y_1}{y_2}{y_3})^{\frac{{m + 1}}{2}}} + {q^{\frac{{m + 3}}{4}}}{y_2}{y_3}{({y_1}{y_2}{y_3})^{\frac{{m - 1}}{2}}}{X_2}{X_3}.
\end{eqnarray*}

Now suppose that the case holds for $m+1 \leq n \leq k $. When $n=k+1$, by Lemma \ref{1}, \ref{2} and \ref{4} we have that

(a)\ if $k$ is even, then
\begin{eqnarray*}
&&{X_{ - m}}{X_{ - m + 2k + 5}}\\
 &=& {X_{ - m}}({X_{ - m + 2k + 3}}{u_1} - {q^{ - \frac{5}{2}}}{y_3}{y_2}{y_1}{X_{ - m + 2k + 1}})\\
 &=& {q^{\frac{{m + {\rm{3}}}}{4}}}{y_{\rm{2}}}{y_{\rm{3}}}{({y_1}{y_2}{y_3})^{\frac{{m - 1}}{2}}}{X_{ - m + k + 1}}{X_{ - m + k + 2}}{u_1} + ({u_k} + ({q^{\frac{3}{2}}} + {q^{\frac{5}{2}}}){u_{k - 2}}{y_1}{y_2}{y_3}\\
 &&+  \cdots  + ({q^{\frac{{3k}}{4}}} + {q^{\frac{{3k + 4}}{4}}} +  \cdots  + {q^{\frac{{5k}}{4}}}){({y_1}{y_2}{y_3})^{\frac{k}{2}}}){u_1} + {y_2}({q^{\frac{1}{2}}}{u_{k - 1}} + ({q^2} + {q^3})\\
 &&\cdot {u_{k - 3}}{y_1}{y_2}{y_3} +  \cdots  + ({q^{\frac{{3k - 4}}{4}}} + {q^{\frac{{3k}}{4}}} +  \cdots  + {q^{\frac{{5k - 8}}{4}}}){u_1}{({y_1}{y_2}{y_3})^{\frac{{k - 2}}{2}}}){u_1} - {q^{\frac{3}{2}}}{q^{\frac{{m + 3}}{4}}}\\
 &&\cdot {y_2}{y_3}{({y_1}{y_2}{y_3})^{\frac{{m - 1}}{2}}}{X_{ - m + k}}{X_{ - m + k + 1}}{y_1}{y_2}{y_3} - {q^{\frac{3}{2}}}({u_{k - 1}} + ({q^{\frac{3}{2}}} + {q^{\frac{5}{2}}}){u_{k - 3}}{y_1}{y_2}{y_3}\\
   \end{eqnarray*}

\begin{eqnarray*}
 &&+  \cdots  + ({q^{\frac{{3k - 6}}{4}}} + {q^{\frac{{3k - 2}}{4}}} +  \cdots  + {q^{\frac{{5k - 10}}{4}}}){u_1}{({y_1}{y_2}{y_3})^{\frac{{k - 2}}{2}}}){y_1}{y_2}{y_3} - {q^{\frac{3}{2}}}{y_2}({q^{\frac{1}{2}}}{u_{k - 2}}\\
 &&+ ({q^2} + {q^3}){u_{k - 4}}{y_1}{y_2}{y_3} +  \cdots  + ({q^{\frac{{3k - 4}}{4}}} + {q^{\frac{{3k}}{4}}} +  \cdots  + {q^{\frac{{5k - 8}}{4}}}){({y_1}{y_2}{y_3})^{\frac{{k - 2}}{2}}}){y_1}{y_2}{y_3}\\
 &=& {q^{\frac{{m + {\rm{3}}}}{4}}}{y_{\rm{2}}}{y_{\rm{3}}}{({y_1}{y_2}{y_3})^{\frac{{m - 1}}{2}}}{X_{ - m + k + 1}}{X_{ - m + k + 2}}{u_1} - {q^{\frac{3}{2}}}{q^{\frac{{m + 3}}{4}}}{y_2}{y_3}{({y_1}{y_2}{y_3})^{\frac{{m - 1}}{2}}}\\
 &&\cdot {X_{ - m + k}}{X_{ - m + k + 1}}{y_1}{y_2}{y_3} + ({u_{k + 1}} + ({q^{\frac{3}{2}}} + {q^{\frac{5}{2}}}){u_{k - 1}}{y_1}{y_2}{y_3} +  \cdots \\
 &&+ ({q^{\frac{{3k}}{4}}} + {q^{\frac{{3k + 4}}{4}}} +  \cdots  + {q^{\frac{{5k}}{4}}}){u_1}{({y_1}{y_2}{y_3})^{\frac{k}{2}}}) + {y_2}({q^{\frac{1}{2}}}{u_k} + ({q^2} + {q^3})\\
 &&\cdot {u_{k - 2}}{y_1}{y_2}{y_3} +  \cdots  + ({q^{\frac{{3k + 2}}{4}}} + {q^{\frac{{3k + 6}}{4}}} +  \cdots  + {q^{\frac{{5k - 2}}{4}}}){({y_1}{y_2}{y_3})^{\frac{k}{2}}}),
\end{eqnarray*}
thus, the proof follows that
\begin{eqnarray*}
&&{q^{\frac{{m + {\rm{3}}}}{4}}}{y_{\rm{2}}}{y_{\rm{3}}}{({y_1}{y_2}{y_3})^{\frac{{m - 1}}{2}}}{X_{ - m + k + 1}}{X_{ - m + k + 2}}{u_1}\\
 &&- {q^{\frac{3}{2}}}{q^{\frac{{m + 3}}{4}}}{y_2}{y_3}{({y_1}{y_2}{y_3})^{\frac{{m - 1}}{2}}}{X_{ - m + k}}{X_{ - m + k + 1}}{y_1}{y_2}{y_3}\\
 &=& {q^{\frac{{m + {\rm{3}}}}{4}}}{y_{\rm{2}}}{y_{\rm{3}}}{({y_1}{y_2}{y_3})^{\frac{{m - 1}}{2}}}{X_{ - m + k + 1}}({q^{ - \frac{5}{2}}}{y_3}{y_2}{y_1}{X_{ - m + k}} + {X_{ - m + k + 4}})\\
 &&- {q^{\frac{3}{2}}}{q^{\frac{{m + 3}}{4}}}{y_2}{y_3}{({y_1}{y_2}{y_3})^{\frac{{m - 1}}{2}}}{X_{ - m + k}}{X_{ - m + k + 1}}{y_1}{y_2}{y_3}\\
 &=& {q^{\frac{{m + {\rm{3}}}}{4}}}{y_{\rm{2}}}{y_{\rm{3}}}{({y_1}{y_2}{y_3})^{\frac{{m - 1}}{2}}}{X_{ - m + k + 2}}{X_{ - m + k + 3}} + {q^{\frac{{5k + 2}}{4}}}{y_2}{({y_1}{y_2}{y_3})^{\frac{k}{2}}};
\end{eqnarray*}

(b)\ if $k$ is odd, the proof is similar.

(ii)\ \textcircled{1}\ When $n=\frac{{m + 1}}{2}$ and $m \geq 3$, by Lemma \ref{8} the case holds.
When $n=\frac{{m + 1}}{2} + 1$ and $m \geq 5$, by Lemma \ref{1}, \ref{2} and \ref{8} we have that

(a)\ if $4 \mid m-1$, then
\begin{eqnarray*}
&&{X_{ - m}}{X_3}\\
 &=& {X_{ - m}}({X_2}w - {q^{ - \frac{1}{2}}}{X_1}{y_2})\\
 &=& {q^{\frac{{m + 1}}{2}}}{y_1}{y_2}{X_{\frac{{1 - m}}{2}}}{X_{\frac{{3 - m}}{2}}}w + ({u_{\frac{{m - 1}}{2}}} + ({q^{\frac{3}{2}}} + {q^{\frac{5}{2}}}){u_{\frac{{m - 5}}{2}}}{y_1}{y_2}{y_3} +  \cdots \\
 &&+ ({q^{\frac{{3m - 3}}{8}}} + {q^{\frac{{3m + 5}}{8}}} +  \cdots  + {q^{\frac{{5m - 5}}{8}}})u_0^{}{({y_1}{y_2}{y_3})^{\frac{{m - 1}}{4}}})w + {y_2}({q^{\frac{1}{2}}}{u_{\frac{{m - 3}}{2}}}\\
 &&+ ({q^2} + {q^3}){u_{\frac{{m - 7}}{2}}}{y_1}{y_2}{y_3} +  \cdots  + ({q^{\frac{{3m - 11}}{8}}} + {q^{\frac{{3m - 3}}{8}}} +  \cdots  + {q^{\frac{{5m - 21}}{8}}})\\
 &&\cdot {u_1}{({y_1}{y_2}{y_3})^{\frac{{m - 5}}{4}}})w - {q^{\frac{1}{2}}}{y_2}{q^{\frac{{m - 2}}{2}}}{y_1}{X_{\frac{{ - 1 - m}}{2}}}{X_{\frac{{3 - m}}{2}}} - {q^{\frac{1}{2}}}{y_2}w({u_{\frac{{m - 3}}{2}}} + ({q^{\frac{3}{2}}}\\
 &&+ {q^{\frac{5}{2}}}){u_{\frac{{m - 7}}{2}}}{y_1}{y_2}{y_3} +  \cdots  + ({q^{\frac{{3m - 15}}{8}}} + {q^{\frac{{3m - 7}}{8}}} +  \cdots  + {q^{\frac{{5m - 25}}{8}}}){u_1}{({y_1}{y_2}{y_3})^{\frac{{m - 5}}{4}}})\\
 &=& {q^{\frac{{m + 1}}{2}}}{y_1}{y_2}{X_{\frac{{1 - m}}{2}}}{X_{\frac{{3 - m}}{2}}}w - {q^{\frac{{m + 1}}{2}}}{y_1}{y_2}{X_{\frac{{ - 1 - m}}{2}}}{X_{\frac{{3 - m}}{2}}} + w({u_{\frac{{m - 1}}{2}}} + ({q^{\frac{3}{2}}} + {q^{\frac{5}{2}}})\\
 &&\cdot {u_{\frac{{m - 5}}{2}}}{y_1}{y_2}{y_3} +  \cdots  + ({q^{\frac{{3m - 3}}{8}}} + {q^{\frac{{3m + 5}}{8}}} +  \cdots  + {q^{\frac{{5m - 5}}{8}}}){({y_1}{y_2}{y_3})^{\frac{{m - 1}}{4}}}),
\end{eqnarray*}
thus, the proof follows that
\begin{eqnarray*}
&&{q^{\frac{{m + 1}}{2}}}{y_1}{y_2}{X_{\frac{{1 - m}}{2}}}{X_{\frac{{3 - m}}{2}}}w - {q^{\frac{{m + 1}}{2}}}{y_1}{y_2}{X_{\frac{{ - 1 - m}}{2}}}{X_{\frac{{3 - m}}{2}}}\\
 &=& {q^{\frac{{m + 1}}{2}}}{y_1}{y_2}{X_{\frac{{3 - m}}{2}}}(q{X_{\frac{{3 - m}}{2}}}{y_1}{y_3} + {X_{\frac{{ - 1 - m}}{2}}}) - {q^{\frac{{m + 1}}{2}}}{y_1}{y_2}{X_{\frac{{ - 1 - m}}{2}}}{X_{\frac{{3 - m}}{2}}}\\
 &=& {q^{\frac{{m + 5}}{2}}}{y_1}({y_1}{y_2}{y_3}){X_{\frac{{3 - m}}{2}}^{2}};
\end{eqnarray*}

(b)\ if $4 \nmid m-1$, the proof is similar.

For $m \geq 7$, we suppose that the case holds for $\frac{{m + 1}}{2}+1 \leq n \leq k < m-1 $. When $n=k+1$, by Lemma \ref{1}, \ref{2} and \ref{4} we have that

(a)\ if $k$ is odd, then
\begin{eqnarray*}
&&{X_{ - m}}{X_{ - m + 2k + 2}}\\
 &=& {X_{ - m}}({X_{ - m + 2k}}{u_1} - {q^{ - \frac{5}{2}}}{y_3}{y_2}{y_1}{X_{ - m + 2k - 2}})\\
 &=& {q^{\frac{{14k - 5m - 11}}{4}}}{y_1}{({y_1}{y_2}{y_3})^{\frac{{2k - m - 1}}{2}}}{X_{ - m + k - 1}}{X_{ - m + k + 1}}{u_1} + w({u_{k - 2}} + ({q^{\frac{3}{2}}} + {q^{\frac{5}{2}}})\\
 &&\cdot {u_{k - 4}}{y_1}{y_2}{y_3} +  \cdots  + ({q^{\frac{{3k - 9}}{4}}} + {q^{\frac{{3k - 5}}{4}}} +  \cdots  + {q^{\frac{{5k - 15}}{4}}}){u_1}{({y_1}{y_2}{y_3})^{\frac{{k - 3}}{2}}}){u_1}\\
 &&- {q^{\frac{{14k - 5m - 25}}{4}}}{y_1}{({y_1}{y_2}{y_3})^{\frac{{2k - m - 3}}{2}}}{X_{ - m + k - 1}}{X_{ - m + k - 1}}{q^{\frac{3}{2}}}{y_1}{y_2}{y_3} - w({u_{k - 3}} + ({q^{\frac{3}{2}}} + {q^{\frac{5}{2}}})\\
 &&\cdot {u_{k - 5}}{y_1}{y_2}{y_3} +  \cdots  + ({q^{\frac{{3k - 9}}{4}}} + {q^{\frac{{3k - 5}}{4}}} +  \cdots  + {q^{\frac{{5k - 15}}{4}}}){u_0}{({y_1}{y_2}{y_3})^{\frac{{k - 3}}{2}}}){q^{\frac{3}{2}}}{y_1}{y_2}{y_3}\\
 &=& {q^{\frac{{14k - 5m - 11}}{4}}}{y_1}{({y_1}{y_2}{y_3})^{\frac{{2k - m - 1}}{2}}}{X_{ - m + k - 1}}{X_{ - m + k + 1}}{u_1} - {q^{\frac{{14k - 5m - 25}}{4}}}{y_1}{({y_1}{y_2}{y_3})^{\frac{{2k - m - 3}}{2}}}\\
 &&\cdot {X_{ - m + k - 1}}{X_{ - m + k - 1}}{q^{\frac{3}{2}}}{y_1}{y_2}{y_3} + w({u_{k - 1}} + ({q^{\frac{3}{2}}} + {q^{\frac{5}{2}}}){u_{k - 3}}{y_1}{y_2}{y_3}\\
 &&+  \cdots  + ({q^{\frac{{3k - 3}}{4}}} + {q^{\frac{{3k + 1}}{4}}} +  \cdots  + {q^{\frac{{5k - 9}}{4}}}){u_0}{({y_1}{y_2}{y_3})^{\frac{{k - 1}}{2}}}),
\end{eqnarray*}
thus, the proof follows that
\begin{eqnarray*}
&&{q^{\frac{{14k - 5m - 11}}{4}}}{y_1}{({y_1}{y_2}{y_3})^{\frac{{2k - m - 1}}{2}}}{X_{ - m + k - 1}}{X_{ - m + k + 1}}{u_1}\\
 &&- {q^{\frac{{14k - 5m - 25}}{4}}}{y_1}{({y_1}{y_2}{y_3})^{\frac{{2k - m - 3}}{2}}}{X_{ - m + k - 1}}{X_{ - m + k - 1}}{q^{\frac{3}{2}}}{y_1}{y_2}{y_3}\\
 &=& {q^{\frac{{14k - 5m - 11}}{4}}}{y_1}{({y_1}{y_2}{y_3})^{\frac{{2k - m - 1}}{2}}}{X_{ - m + k + 1}}({q^{ - \frac{1}{2}}}{y_3}{y_2}{y_1}{X_{ - m + k + 1}} + {X_{ - m + k - 3}})\\
 &&- {q^{\frac{{14k - 5m - 25}}{4}}}{y_1}{({y_1}{y_2}{y_3})^{\frac{{2k - m - 3}}{2}}}{X_{ - m + k - 1}}{X_{ - m + k - 1}}{q^{\frac{3}{2}}}{y_1}{y_2}{y_3}\\
 &=& {q^{\frac{{14k - 5m + 3}}{4}}}{y_1}{({y_1}{y_2}{y_3})^{\frac{{2k - m + 1}}{2}}}{X_{ - m + k + 1}^{2}} + w{q^{\frac{{5k - 5}}{4}}}{({y_1}{y_2}{y_3})^{\frac{{k - 1}}{2}}};
\end{eqnarray*}

(b)\ if $k$ is even, the proof is similar.

\textcircled{2}\ It is obvious for $m=1$. When $m = 3$, by Lemma \ref{1} and \ref{4} we have
\begin{eqnarray*}
{X_{ - 3}}{X_3}
 &=& {X_{ - 3}}({X_1}{u_1} - {q^{ - \frac{1}{2}}}{y_1}{X_{ - 1}})\\
 &=& ({q^{\frac{1}{2}}}X_{ - 1}^2{y_1} + w){u_1} - {q^{ - \frac{1}{2}}}{X_{ - 3}}{X_{ - 1}}{y_1}\\
 &=& {q^{\frac{1}{2}}}{y_1}{X_{ - 1}}({y_3}{y_2}{X_1} + {X_{ - 3}}) + w{u_1} - {q^{ - \frac{1}{2}}}{X_{ - 3}}{X_{ - 1}}{y_1}\\
 &=& {q^{\frac{5}{2}}}{y_1}{y_2}{y_3}{X_{ - 1}}{X_1} + w{u_1}.
\end{eqnarray*}

When $m \geq 5$, by Lemma \ref{1}, \ref{2}, \ref{4} and \textcircled{1} of (ii) we have
\begin{eqnarray*}
&&{X_{ - m}}{X_m}\\
 &=& {X_{ - m}}({X_{m - 2}}{u_1} - {q^{ - \frac{5}{2}}}{y_3}{y_2}{y_1}{X_{m - 4}})\\
 &=& {q^{\frac{{9m - 25}}{4}}}{y_1}{({y_1}{y_2}{y_3})^{\frac{{m - 3}}{2}}}{X_{ - 1}}{X_{ - 1}}{u_1} + w({u_{m - 3}} + ({q^{\frac{3}{2}}} + {q^{\frac{5}{2}}}){u_{m - 5}}{y_1}{y_2}{y_3}\\
 &&+  \cdots  + ({q^{\frac{{3m - 9}}{4}}} + {q^{\frac{{3m - 5}}{4}}} +  \cdots  + {q^{\frac{{5m - 15}}{4}}}){u_0}{({y_1}{y_2}{y_3})^{\frac{{m - 3}}{2}}}){u_1} - {q^{\frac{3}{2}}}{q^{\frac{{9m - 39}}{4}}}{y_1}\\
 &&\cdot {({y_1}{y_2}{y_3})^{\frac{{m - 5}}{2}}}{X_{ - 3}}{X_{ - 1}}{y_1}{y_2}{y_3} - {q^{\frac{3}{2}}}w({u_{m - 4}} + ({q^{\frac{3}{2}}} + {q^{\frac{5}{2}}}){u_{m - 6}}{y_1}{y_2}{y_3}\\
 &&+  \cdots  + ({q^{\frac{{3m - 15}}{4}}} + {q^{\frac{{3m - 11}}{4}}} +  \cdots  + {q^{\frac{{5m - 25}}{4}}}){u_1}{({y_1}{y_2}{y_3})^{\frac{{m - 5}}{2}}}){y_1}{y_2}{y_3}\\
   \end{eqnarray*}

\begin{eqnarray*}
 &=& {q^{\frac{{9m - 25}}{4}}}{y_1}{({y_1}{y_2}{y_3})^{\frac{{m - 3}}{2}}}{X_{ - 1}}{X_{ - 1}}{u_1} - {q^{\frac{3}{2}}}{q^{\frac{{9m - 39}}{4}}}{y_1}{({y_1}{y_2}{y_3})^{\frac{{m - 5}}{2}}}\\
 &&\cdot {X_{ - 3}}{X_{ - 1}}{y_1}{y_2}{y_3} + w({u_{m - 2}} + ({q^{\frac{3}{2}}} + {q^{\frac{5}{2}}}){u_{m - 4}}{y_1}{y_2}{y_3} +  \cdots \\
 &&+ ({q^{\frac{{3m - 9}}{4}}} + {q^{\frac{{3m - 5}}{4}}} +  \cdots  + {q^{\frac{{5m - 15}}{4}}}){u_1}{({y_1}{y_2}{y_3})^{\frac{{m - 3}}{2}}}),
\end{eqnarray*}
thus, the proof follows that
\begin{eqnarray*}
&&{q^{\frac{{9m - 25}}{4}}}{y_1}{({y_1}{y_2}{y_3})^{\frac{{m - 3}}{2}}}{X_{ - 1}}{X_{ - 1}}{u_1} - {q^{\frac{3}{2}}}{q^{\frac{{9m - 39}}{4}}}{y_1}{({y_1}{y_2}{y_3})^{\frac{{m - 5}}{2}}}{X_{ - 3}}{X_{ - 1}}{y_1}{y_2}{y_3}\\
 &=& {q^{\frac{{9m - 25}}{4}}}{y_1}{({y_1}{y_2}{y_3})^{\frac{{m - 3}}{2}}}{X_{ - 1}}({y_3}{y_2}{X_1} + {X_{ - 3}})\\
 &&- {q^{\frac{3}{2}}}{q^{\frac{{9m - 39}}{4}}}{y_1}{({y_1}{y_2}{y_3})^{\frac{{m - 5}}{2}}}{X_{ - 3}}{X_{ - 1}}{y_1}{y_2}{y_3}\\
 &=& {q^{\frac{{5m - 5}}{4}}}{({y_1}{y_2}{y_3})^{\frac{{m - 1}}{2}}}{X_{ - 1}}{X_1}.
\end{eqnarray*}

\textcircled{3}\ Firstly, we consider the case for $n=m+1$: when $m = 1$, by Lemma \ref{1} the case holds; when $m \geq 3$, by Lemma \ref{1}, \ref{2}, \ref{4}, \textcircled{1}  and \textcircled{2} of (ii) we have
\begin{eqnarray*}
&&{X_{ - m}}{X_{m + 2}}\\
 &=& {X_{ - m}}({X_m}{u_1} - {q^{ - \frac{5}{2}}}{y_3}{y_2}{y_1}{X_{m - 2}})\\
 &=& {q^{\frac{{5m - 5}}{4}}}{({y_1}{y_2}{y_3})^{\frac{{m - 1}}{2}}}{X_{ - 1}}{X_1}{u_1} + w({u_{m - 2}} + ({q^{\frac{3}{2}}} + {q^{\frac{5}{2}}}){u_{m - 4}}{y_1}{y_2}{y_3}\\
 &&+  \cdots  + ({q^{\frac{{3m - 9}}{4}}} + {q^{\frac{{3m - 5}}{4}}} +  \cdots  + {q^{\frac{{5m - 15}}{4}}}){u_1}{({y_1}{y_2}{y_3})^{\frac{{m - 3}}{2}}}){u_1} - {q^{\frac{3}{2}}}{q^{\frac{{9m - 25}}{4}}}\\
 &&\cdot {y_1}{({y_1}{y_2}{y_3})^{\frac{{m - 3}}{2}}}{X_{ - 1}}{X_{ - 1}}{y_1}{y_2}{y_3} - {q^{\frac{3}{2}}}w({u_{m - 3}} + ({q^{\frac{3}{2}}} + {q^{\frac{5}{2}}}){u_{m - 5}}{y_1}{y_2}{y_3}\\
 &&+  \cdots  + ({q^{\frac{{3m - 9}}{4}}} + {q^{\frac{{3m - 5}}{4}}} +  \cdots  + {q^{\frac{{5m - 15}}{4}}}){u_0}{({y_1}{y_2}{y_3})^{\frac{{m - 3}}{2}}}){y_1}{y_2}{y_3}\\
 &=& {q^{\frac{{5m - 5}}{4}}}{({y_1}{y_2}{y_3})^{\frac{{m - 1}}{2}}}{X_{ - 1}}{X_1}{u_1} - {q^{\frac{3}{2}}}{q^{\frac{{9m - 25}}{4}}}{y_1}{({y_1}{y_2}{y_3})^{\frac{{m - 3}}{2}}}{X_{ - 1}}{X_{ - 1}}{y_1}{y_2}{y_3}\\
 &&+ w({u_{m - 1}} + ({q^{\frac{3}{2}}} + {q^{\frac{5}{2}}}){u_{m - 3}}{y_1}{y_2}{y_3} +  \cdots  + ({q^{\frac{{3m - 3}}{4}}} + {q^{\frac{{3m + 1}}{4}}} +  \cdots  + {q^{\frac{{5m - 9}}{4}}}){({y_1}{y_2}{y_3})^{\frac{{m - 1}}{2}}}),
\end{eqnarray*}
thus, the proof follows that
\begin{eqnarray*}
&&{q^{\frac{{5m - 5}}{4}}}{({y_1}{y_2}{y_3})^{\frac{{m - 1}}{2}}}{X_{ - 1}}{X_1}{u_1} - {q^{\frac{3}{2}}}{q^{\frac{{9m - 25}}{4}}}{y_1}{({y_1}{y_2}{y_3})^{\frac{{m - 3}}{2}}}{X_{ - 1}}{X_{ - 1}}{y_1}{y_2}{y_3}\\
 &=& {q^{\frac{{5m - 5}}{4}}}{({y_1}{y_2}{y_3})^{\frac{{m - 1}}{2}}}{X_{ - 1}}({q^{ - \frac{1}{2}}}{y_1}{X_{ - 1}} + {X_3})\\
 &&- {q^{\frac{3}{2}}}{q^{\frac{{9m - 25}}{4}}}{y_1}{({y_1}{y_2}{y_3})^{\frac{{m - 3}}{2}}}{X_{ - 1}}{X_{ - 1}}{y_1}{y_2}{y_3}\\
 &=& {q^{\frac{{5m - 5}}{4}}}w{({y_1}{y_2}{y_3})^{\frac{{m - 1}}{2}}} + {q^{\frac{{m + 3}}{4}}}{y_2}{y_3}{({y_1}{y_2}{y_3})^{\frac{{m - 1}}{2}}}{X_1^2}.
\end{eqnarray*}

Secondly, we consider the case for $n=m+2$: when $m = 1$, by Lemma \ref{1} and \ref{4} we have
\begin{eqnarray*}
{X_{ - 1}}{X_5}
 &=& {X_{ - 1}}({X_3}{u_1} - {q^{ - \frac{5}{2}}}{y_3}{y_2}{y_1}{X_1})\\
 &=& (q{y_2}{y_3}{X_1}{X_1} + w){u_1} - {q^{\frac{3}{2}}}{X_{ - 1}}{X_1}{y_1}{y_2}{y_3}\\
 &=& q{y_2}{y_3}{X_1}({q^{ - \frac{1}{2}}}{y_1}{X_{ - 1}} + {X_3}) + w{u_1} - {q^{\frac{3}{2}}}{X_{ - 1}}{X_1}{y_1}{y_2}{y_3}\\
 &=& q{y_2}{y_3}{X_1}{X_3} + w{u_1};
\end{eqnarray*}

when $m\geq3$, by Lemma \ref{1}, \ref{2}, \ref{4}, \textcircled{2} of (ii) we have
\begin{eqnarray*}
&&{X_{ - m}}{X_{m + 4}}\\
 &=& {X_{ - m}}({X_{m + 2}}{u_1} - {q^{ - \frac{5}{2}}}{y_3}{y_2}{y_1}{X_m})\\
 &=& {q^{\frac{{m{\rm{ + }}3}}{4}}}{y_2}{y_3}{({y_1}{y_2}{y_3})^{\frac{{m - 1}}{2}}}{X_{\rm{1}}}{X_1}{u_1} + w({u_{m - 1}} + ({q^{\frac{3}{2}}} + {q^{\frac{5}{2}}}){u_{m - 3}}{y_1}{y_2}{y_3} +  \cdots \\
 &&+ ({q^{\frac{{3m - 3}}{4}}} + {q^{\frac{{3m + 1}}{4}}} +  \cdots  + {q^{\frac{{5m - 5}}{4}}}){({y_1}{y_2}{y_3})^{\frac{{m - 1}}{2}}}){u_1} - {q^{\frac{3}{2}}}{q^{\frac{{{\rm{5}}m - 5}}{4}}}{({y_1}{y_2}{y_3})^{\frac{{m - {\rm{1}}}}{2}}}\\
 &&\cdot {X_{ - 1}}{X_{\rm{1}}}{y_1}{y_2}{y_3} - {q^{\frac{3}{2}}}w({u_{m - 2}} + ({q^{\frac{3}{2}}} + {q^{\frac{5}{2}}}){u_{m - 4}}{y_1}{y_2}{y_3} +  \cdots \\
 &&+ ({q^{\frac{{3m - 9}}{4}}} + {q^{\frac{{3m - 5}}{4}}} +  \cdots  + {q^{\frac{{5m - 15}}{4}}}){u_1}{({y_1}{y_2}{y_3})^{\frac{{m - 3}}{2}}}){y_1}{y_2}{y_3}\\
 &=& {q^{\frac{{m{\rm{ + }}3}}{4}}}{y_2}{y_3}{({y_1}{y_2}{y_3})^{\frac{{m - 1}}{2}}}{X_{\rm{1}}}{X_1}{u_1} - {q^{\frac{3}{2}}}{q^{\frac{{{\rm{5}}m - 5}}{4}}}{({y_1}{y_2}{y_3})^{\frac{{m - {\rm{1}}}}{2}}}{X_{ - 1}}{X_{\rm{1}}}{y_1}{y_2}{y_3} + w({u_m}\\
 &&+ ({q^{\frac{3}{2}}} + {q^{\frac{5}{2}}}){u_{m - 2}}{y_1}{y_2}{y_3} +  \cdots  + ({q^{\frac{{3m - 3}}{4}}} + {q^{\frac{{3m + 1}}{4}}} +  \cdots  + {q^{\frac{{5m - 5}}{4}}}){u_1}{({y_1}{y_2}{y_3})^{\frac{{m - 1}}{2}}}),
\end{eqnarray*}
thus, the proof follows that
\begin{eqnarray*}
&&{q^{\frac{{m{\rm{ + }}3}}{4}}}{y_2}{y_3}{({y_1}{y_2}{y_3})^{\frac{{m - 1}}{2}}}{X_{\rm{1}}}{X_1}{u_1} - {q^{\frac{3}{2}}}{q^{\frac{{{\rm{5}}m - 5}}{4}}}{({y_1}{y_2}{y_3})^{\frac{{m - {\rm{1}}}}{2}}}{X_{ - 1}}{X_{\rm{1}}}{y_1}{y_2}{y_3}\\
 &=& {q^{\frac{{m{\rm{ + }}3}}{4}}}{y_2}{y_3}{({y_1}{y_2}{y_3})^{\frac{{m - 1}}{2}}}{X_{\rm{1}}}({q^{ - \frac{1}{2}}}{y_1}{X_{ - 1}} + {X_3})\\
 &&- {q^{\frac{3}{2}}}{q^{\frac{{{\rm{5}}m - 5}}{4}}}{({y_1}{y_2}{y_3})^{\frac{{m - {\rm{1}}}}{2}}}{X_{ - 1}}{X_{\rm{1}}}{y_1}{y_2}{y_3}\\
 &=& {q^{\frac{{m{\rm{ + }}3}}{4}}}{y_2}{y_3}{({y_1}{y_2}{y_3})^{\frac{{m - 1}}{2}}}{X_{\rm{1}}}{X_3}.
\end{eqnarray*}

Now suppose the case holds for $m+2 \leq n \leq k $. When $n=k+1$, by Lemma \ref{1}, \ref{2} and \ref{4} we have that

(a)\ if $k$ is odd, then
\begin{eqnarray*}
&&{X_{ - m}}{X_{ - m + 2k + 2}}\\
 &=& {X_{ - m}}({X_{ - m + 2k}}{u_1} - {q^{ - \frac{5}{2}}}{y_3}{y_2}{y_1}{X_{ - m + 2k - 2}})\\
 &=& {q^{\frac{{m + {\rm{3}}}}{4}}}{y_{\rm{2}}}{y_{\rm{3}}}{({y_1}{y_2}{y_3})^{\frac{{m - 1}}{2}}}{X_{ - m + k - 1}}{X_{ - m + k + 1}}{u_1} + w({u_{k - 2}} + ({q^{\frac{3}{2}}} + {q^{\frac{5}{2}}})\\
 &&\cdot {u_{k - 4}}{y_1}{y_2}{y_3} +  \cdots  + ({q^{\frac{{3k - 9}}{4}}} + {q^{\frac{{3k - 5}}{4}}} +  \cdots  + {q^{\frac{{5k - 15}}{4}}}){u_1}{({y_1}{y_2}{y_3})^{\frac{{k - 3}}{2}}}){u_1}\\
 &&- {q^{\frac{{m + 3}}{4}}}{y_2}{y_3}{({y_1}{y_2}{y_3})^{\frac{{m - 1}}{2}}}{X_{ - m + k - 1}}{X_{ - m + k - 1}}{q^{\frac{3}{2}}}{y_1}{y_2}{y_3} - w({u_{k - 3}} + ({q^{\frac{3}{2}}} + {q^{\frac{5}{2}}})\\
 &&\cdot {u_{k - 5}}{y_1}{y_2}{y_3} +  \cdots  + ({q^{\frac{{3k - 9}}{4}}} + {q^{\frac{{3k - 5}}{4}}} +  \cdots  + {q^{\frac{{5k - 15}}{4}}}){u_0}{({y_1}{y_2}{y_3})^{\frac{{k - 3}}{2}}}){q^{\frac{3}{2}}}{y_1}{y_2}{y_3}\\
 &=& {q^{\frac{{m + {\rm{3}}}}{4}}}{y_{\rm{2}}}{y_{\rm{3}}}{({y_1}{y_2}{y_3})^{\frac{{m - 1}}{2}}}{X_{ - m + k - 1}}{X_{ - m + k + 1}}{u_1} - {q^{\frac{{m + 3}}{4}}}{y_2}{y_3}{({y_1}{y_2}{y_3})^{\frac{{m - 1}}{2}}}\\
 &&\cdot {X_{ - m + k - 1}}{X_{ - m + k - 1}}{q^{\frac{3}{2}}}{y_1}{y_2}{y_3} + w({u_{k - 1}} + ({q^{\frac{3}{2}}} + {q^{\frac{5}{2}}}){u_{k - 3}}{y_1}{y_2}{y_3}\\
 &&+  \cdots  + ({q^{\frac{{3k - 3}}{4}}} + {q^{\frac{{3k + 1}}{4}}} +  \cdots  + {q^{\frac{{5k - 9}}{4}}}){u_0}{({y_1}{y_2}{y_3})^{\frac{{k - 1}}{2}}}),
\end{eqnarray*}
thus, the proof follows that
\begin{eqnarray*}
&&{q^{\frac{{m + {\rm{3}}}}{4}}}{y_{\rm{2}}}{y_{\rm{3}}}{({y_1}{y_2}{y_3})^{\frac{{m - 1}}{2}}}{X_{ - m + k - 1}}{X_{ - m + k + 1}}{u_1}\\
 &&- {q^{\frac{{m + 3}}{4}}}{y_2}{y_3}{({y_1}{y_2}{y_3})^{\frac{{m - 1}}{2}}}{X_{ - m + k - 1}}{X_{ - m + k - 1}}{q^{\frac{3}{2}}}{y_1}{y_2}{y_3}\\
 &=& {q^{\frac{{m + {\rm{3}}}}{4}}}{y_{\rm{2}}}{y_{\rm{3}}}{({y_1}{y_2}{y_3})^{\frac{{m - 1}}{2}}}{X_{ - m + k - 1}}({q^{ - \frac{5}{2}}}{y_3}{y_2}{y_1}{X_{ - m + k - 1}} + {X_{ - m + k + 3}})\\
 &&- {q^{\frac{{m + 3}}{4}}}{y_2}{y_3}{({y_1}{y_2}{y_3})^{\frac{{m - 1}}{2}}}{X_{ - m + k - 1}}{X_{ - m + k - 1}}{q^{\frac{3}{2}}}{y_1}{y_2}{y_3}\\
 &=& {q^{\frac{{m + {\rm{3}}}}{4}}}{y_{\rm{2}}}{y_{\rm{3}}}{({y_1}{y_2}{y_3})^{\frac{{m - 1}}{2}}}{X_{ - m + k + 1}^2} + w{q^{\frac{{5k - 5}}{4}}}{({y_1}{y_2}{y_3})^{\frac{{k - 1}}{2}}};
\end{eqnarray*}

(b)\ if $k$ is even, the proof is similar.
\end{proof}
\begin{remark}\label{opp}
By using bar-involution on these cluster multiplication formulas, one can
obtain the cluster multiplication formulas in the other direction.
\end{remark}

\section{A bar-invariant positive $\mathbb{ZP}$-basis}
In this section, by using the cluster multiplication formulas  established in Section 3, we  construct an explicit bar-invariant positive $\mathbb{ZP}$-basis of $\mathcal{A}_{q} (Q)$, which can be viewed  as a quantum analogue of the atomic basis in \cite{CI-1}.

\begin{definition}
An element in $\mathcal{A}_{q} (Q)$ is called positive if the coefficients of its Laurent expansion associated with any cluster belong to
$\mathbb{Z}_{\geq 0}\mathbb{P}$.
\end{definition}

\begin{definition}
A $\mathbb{ZP}$-basis of $\mathcal{A}_{q} (Q)$ is called  positive if its
structure constants belong to $\mathbb{Z}_{\geq 0}\mathbb{P}$.
\end{definition}

Denote by
$$\mathcal{B}=\{ \text{quantum cluster monomials} \} \sqcup \{ u_{n}\omega^{k}, u_{n}z^{k}| n\geq 1, k\geq 0\}. $$
It follows that $\mathcal{B}\subset \mathcal{A}_{q} (Q)$.

\begin{theorem}\label{positive}
The set $\mathcal{B}$ is a bar-invariant positive  $\mathbb{ZP}$-basis of $\mathcal{A}_{q} (Q)$.
\end{theorem}

\begin{proof}
By Lemma~\ref{0}, Lemma \ref{regular} and the fact that quantum cluster monomials are bar-invariant,  any element in $\mathcal{B}$ is bar-invariant.

According to Lemma~\ref{3}, Theorem~\ref{5},  Theorem~\ref{7}, Theorem~\ref{9}, Remark \ref{opp},  and the same arguments as \cite[Subsection 6.3]{CI-1}, one can deduce
that the set $\mathcal{B}$ spans  $\mathcal{A}_{q} (Q)$ over $\mathbb{ZP}$.

The positivity of quantum cluster variables follows from \cite{KQ,fanqin1}, and combining  with Theorem~\ref{5}, we  obtain that  $u_n  (n\geq 1)$ are positive.
Note that the Laurent monomials  in the expansion of $b\in \mathcal{B}$ in any cluster are corresponding to those  of $b|_{q=1}$ due to the positivity of elements in $\mathcal{B}$.
Thus, using the same arguments as \cite[Subsection 6.1]{CI-1}, we can deduce that the elements in $\mathcal{B}$ are linearly independent over $\mathbb{ZP}$.

Again from Lemma~\ref{3}, Theorem~\ref{5},  Theorem~\ref{7}, Theorem~\ref{9} and Remark \ref{opp},    it follows  that  the
structure constants belong to $\mathbb{Z}_{\geq 0}\mathbb{P}$.
\end{proof}

\begin{remark}\label{atomic}
By specializing $q=1$, the obtained basis  is exactly atomic in \cite{CI-1}.
\end{remark}

\section*{Acknowledgments}
Ming Ding was supported by NSF of China (No. 12371036) and  Guangdong Basic and Applied Basic Research
Foundation (2023A1515011739), and Fan Xu was supported by NSF of China (No. 12031007).


\begin{thebibliography}{99}

\bibitem{BCDX} L. Bai, X. Chen, M. Ding and F. Xu, \emph{Cluster multiplication theorem in the quantum cluster algebra of type $A^{(2)}_2$ and the triangular basis},
J. Algebra \textbf{533} (2019), 106--141.

\bibitem{berzel}
A.~Berenstein and A.~Zelevinsky, {\em Quantum cluster algebras,}
Adv. Math. \textbf{195} (2005), 405--455.

\bibitem{ck}
P.~Caldero and B.~Keller, \emph{From triangulated categories to cluster algebras}, Invent. Math. 172 (2008), 169--211.

\bibitem{CL}
$\mathrm{\dot{I}}$. Canaki and  P. Lampe,  \emph{An expansion formula for type A and Kronecker quantum cluster algebras}, J Combin Theory Ser
A, 2020, 171: 105132.

\bibitem{CI-0}
G. Cerulli Irelli, \emph{Structural theory of rank three cluster algebras of affine type}, PhD thesis,
Universit$\grave{a}$ degli studi di Padova (2008).

\bibitem{CI-1}
G. Cerulli Irelli, \emph{Cluster algebras of type $A_2^{(1)}$}, Algebr. Represent. Theory 15 (2012), no. 5, 977--1021.

\bibitem{cdz} X. Chen, M. Ding and H.Zhang, \emph{ The cluster multiplication theorem for acyclic quantum cluster algebras},
Int. Math. Res. Not. IMRN 2023, no. 23, 20533--20573.

\bibitem{dx}
M. Ding and F. Xu,
\emph{Bases of the quantum cluster algebra of the Kronecker quiver},
Acta Math. Sin. (Engl. Ser.) \textbf{28} (2012), no. 6, 1169--1178.


\bibitem{ca1}
S.~Fomin and A.~Zelevinsky, \emph{Cluster algebras. I. Foundations,}
J. Amer. Math. Soc.  \textbf{15}  (2002),  no. 2, 497--529.

\bibitem{Hubery}
A.~Hubery, \emph{Acyclic cluster algebras via Ringel-Hall algebras},
preprint (2005).

\bibitem{KS}
B. Keller and S. Scherotzke, \emph{Linear recurrence relations for cluster variables of affine quivers}, Adv. Math. 228 (2011), no. 3, 1842--1862.

\bibitem{KQ}
Y. Kimura and F. Qin,
\emph{Graded quiver varieties, quantum cluster algebras and dual canonical basis},
Adv. Math. \textbf{262} (2014), 261--312.

\bibitem{li-pan}
F. Li and J. Pan,  \emph{Recurrence formula, positivity and polytope basis in cluster algebras via Newton polytopes},
arXiv: 2201.01440 [math.RT].

\bibitem{Pa}
J. Pallister, \emph{Linear relations and integrability for cluster algebras from affine quivers}, Glasg. Math. J.   63 (2021), no. 3, 584--621.


\bibitem{fanqin1}
F. Qin,
\emph{t-analog of q-characters, bases of quantum cluster algebras, and a correction technique}, Int. Math. Res. Not. IMRN 2014, no. 22, 6175--6232.

\bibitem{Q3}
F. Qin, \emph{Cluster algebras and their bases}, arXiv: 2108.09279  [math.RT].

\bibitem{SZ}
P. Sherman and A. Zelevinsky,
\emph{Positivity and canonical bases in rank $2$ cluster algebras of finite and affine types},
Mosc. Math. J. \textbf{4} (2004), no. 4, 947--974, 982.

\bibitem{XX}
J. Xiao and F. Xu, \emph{Green's formula with $\mathbb{C}^{*}$-action and Caldero-Keller's formula for cluster algebras},
in: Representation Theory of Algebraic Groups and Quantum Groups, in: Progr. Math., vol. 284,
Birkhauser/Springer, New York, 2010: 313-348.

\bibitem{X}
 F. Xu, \emph{On the cluster multiplication theorem for acyclic cluster algebras}, Trans. Amer. Math. Soc.
362(2) (2010), 753--776.
\end{thebibliography}
\end{document}